\documentclass[11pt,a4paper,reqno]{amsart}
\usepackage[english]{babel}
\usepackage[T1]{fontenc}
\usepackage{verbatim}
\usepackage{mathabx}
\usepackage{palatino}
\usepackage{amsmath}
\usepackage{amssymb}
\usepackage{amsthm}
\usepackage{amsfonts}
\usepackage{graphicx}
\usepackage{color}
\usepackage{mathtools}
\usepackage{dsfont}
\usepackage{mathrsfs}
\usepackage{comment}
\usepackage{overpic}

\DeclarePairedDelimiter\ceil{\lceil}{\rceil}

\usepackage[colorlinks = true, citecolor = black]{hyperref}
\title{Rickman rugs and intrinsic bilipschitz graphs}
\author{Tuomas Orponen}
\address{University of Jyv\"askyl\"a, Department of Mathematics and Statistics}
\email{tuomas.t.orponen@jyu.fi}
\date{\today}
\thanks{T.O. is supported by the Academy of Finland via the project \emph{Quantitative rectifiability in Euclidean and non-Euclidean spaces}, grant Nos. 309365, 314172, and via the project \emph{Incidences on Fractals}, grant No. 321896.}

\newcommand{\R}{\mathbb{R}}
\newcommand{\W}{\mathbb{W}}
\newcommand{\He}{\mathbb{H}}
\newcommand{\N}{\mathbb{N}}

\newcommand{\Z}{\mathbb{Z}}

\newcommand{\calT}{\mathcal{T}}

\newcommand{\calD}{\mathcal{D}}
\newcommand{\calH}{\mathcal{H}}

\newcommand{\spt}{\operatorname{spt}}

\newcommand{\diam}{\operatorname{diam}}
\newcommand{\card}{\operatorname{card}}
\newcommand{\dist}{\operatorname{dist}}

\newcommand{\V}{\mathbb{V}}

\def\Barint_#1{\mathchoice
          {\mathop{\vrule width 6pt height 3 pt depth -2.5pt
                  \kern -8pt \intop}\nolimits_{#1}}%
          {\mathop{\vrule width 5pt height 3 pt depth -2.6pt
                  \kern -6pt \intop}\nolimits_{#1}}%
          {\mathop{\vrule width 5pt height 3 pt depth -2.6pt
                  \kern -6pt \intop}\nolimits_{#1}}%
          {\mathop{\vrule width 5pt height 3 pt depth -2.6pt
                  \kern -6pt \intop}\nolimits_{#1}}}

\numberwithin{equation}{section}

\theoremstyle{plain}
\newtheorem{thm}[equation]{Theorem}

\newtheorem{lemma}[equation]{Lemma}

\newtheorem{cor}[equation]{Corollary}
\newtheorem{proposition}[equation]{Proposition}
\newtheorem{question}{Question}
\newtheorem{claim}[equation]{Claim}

\theoremstyle{definition}

\newtheorem{definition}[equation]{Definition}

\newtheorem{notation}[equation]{Notation}

\theoremstyle{remark}
\newtheorem{remark}[equation]{Remark}

\newcommand{\nref}[1]{(\hyperref[#1]{#1})}

\begin{document}

\begin{abstract}  This paper studies the geometry of bilipschitz maps $f \colon \W \to \He$, where $\He$ is the first Heisenberg group, and $\W \subset \He$ is a vertical subgroup of co-dimension $1$. The images $f(\W)$ of such maps are called \emph{Rickman rugs} in the Heisenberg group.

The main theorem states that a Rickman rug in the Heisenberg group admits a corona decomposition by intrinsic bilipschitz graphs. As a corollary, Rickman rugs are countably rectifiable by intrinsic bilipschitz graphs. Here, an intrinsic bilipschitz graph is an intrinsic Lipschitz graph, which is simultaneously a Rickman rug. General intrinsic Lipschitz graphs need not be Rickman rugs, even locally, by an example of Bigolin and Vittone. \end{abstract}

\maketitle

\tableofcontents

\section{Introduction}

The purpose of this paper is to introduce the notion of \emph{intrinsic bilipschitz graphs} in the first Heisenberg group $\He$, and show that bilipschitz images of the parabolic plane inside $\He$ are countably rectifiable by intrinsic bilipschitz graphs -- in a strong quantitative sense. Before the details, I will explain why such a result might be useful, and how it is connected to the surrounding area of research on \emph{quantitative rectifiability in $\He^{n}$}.

In $\R^{n}$, Lipschitz graphs, and domains bounded by Lipschitz graphs, are basic building blocks for performing harmonic analysis on "non-smooth" surfaces and domains. To name a prototypical example, consider the problem of showing that a singular integral operator (SIO) is $L^{2}$-bounded on a surface $S \subset \R^{n}$. By the works of Calder\'on, Coifman, David, Macintosh, and Meyer \cite{Calderon,CMM,MR700980}, it is known that SIOs with odd kernels are $L^{2}$-bounded on Lipschitz graphs. With this fundamental result in hand, there are at least two distinct "transference principles" which enable one to treat more general surfaces: David \cite{MR744071,MR956767} developed the notion of \emph{big pieces of Lipschitz graphs}, and David, Jones, and Semmes \cite{DS1,MR1069238,Semmes} pioneered the notion of \emph{corona decompositions by Lipschitz graphs}. The rough idea is that surfaces with either big pieces of, or a corona decomposition by, Lipschitz graphs behave just as well as Lipschitz graphs relative to many harmonic analysis problems. For example, SIOs with odd kernels are $L^{2}$-bounded on such surfaces. This paradigm was formalised in the monographs \cite{DS1,MR1251061} of David and Semmes, which gave birth to the vibrant research area of quantitative rectifiability. 

There have been recent attempts to work out a theory of quantitative rectifiability in the Heisenberg group $\He^{n}$. Following the train of thought outlined above, one should identify the correct "building blocks" for performing harmonic analysis on surfaces, that is, the "Lipschitz graphs of the Heisenberg group", then establish transference principles to treat more general surfaces, and of course verify that interesting classes of general surfaces satisfy the hypotheses of the transference principles. 

Regarding the question of "building blocks", there is an obvious answer: the \emph{intrinsic Lipschitz graphs} (iLGs) introduced by Franchi, Serapioni, and Serra Cassano \cite{FSS} in 2006. These share many of the elementary properties of Lipschitz graphs in $\R^{n}$: for example, iLGs are Ahlfors-regular, have approximate tangents almost everywhere by \cite{FSSC2}, are invariant under group automorphisms, and can be extended. There is a growing mountain of evidence showing that iLGs are solid building blocks for a theory of \emph{qualitative rectifiability} in Heisenberg groups: for a few references, see \cite{zbMATH05015569,BCSC,2020arXiv200201433C,MR4127898,MR2789472,2019arXiv190811639M,2020arXiv200703236M,MR3194680,MR3682744,MR3587666,2020arXiv200714286V}.  

There is also evidence, but not so compelling, that intrinsic Lipschitz graphs are suitable building blocks for a theory quantitative rectifiability in Heisenberg groups. In the first Heisenberg group, natural SIOs are known to be $L^{2}$-bounded on $1$-dimensional iLGs by \cite{2019arXiv191103223F}, but the same problem remains open for $1$-codimensional iLGs. Under an "$\epsilon$" of additional regularity, the answer is positive by \cite{CFO2,2018arXiv181013122F}. Very little is known about other classical topics in harmonic analysis, such as boundary value problems for (sub-elliptic) PDEs, or the behaviour of of (sub-)harmonic measure on domains bounded by iLGs; see the introduction to \cite{2020arXiv200608293O} for a description of the current state-of-the art on these questions.

We arrive at a point which is centrally relevant to the current paper: it is conceivable that $1$-codimensional iLGs in the first Heisenberg group $\He$ are \textbf{not} suitable building blocks for a theory of quantitative rectifiability, but they are suitable in $\He^{n}$ for $n \geq 2$. The recent preprint \cite{2020arXiv200411447C} of Chousionis, Li, and Young shows that $1$-codimensional iLGs in $\He^{n}$, for $n \geq 2$, admit approximations by planes in a quantitative sense which precisely matches the Euclidean behaviour of Lipschitz graphs. In contrast, an example in the preprint \cite{2020arXiv200412522N} by Naor and Young suggests that the same is not true for $1$-codimensional iLGs in $\He$. In $\R^{n}$, it is known by \cite[(C1) $\Longleftrightarrow$ (C3)]{DS1} that the correct rate of affine approximation is \textbf{equivalent} with the $L^{2}$-boundedness of SIOs. So, in the light of current information, one may speculate that natural SIOs are $L^{2}$-bounded on $1$-codimensional iLGs in $\He^{n}$, $n \geq 2$, but not in $\He$. This remains a key open problem in the area.

\subsection{Intrinsic bilipschitz graphs} In the sequel, "iLG" refers to a $1$-codimensional iLG in the first Heisenberg group $\He$. These sets are homeomorphic to $\R^{2}$, and $3$-regular relative to the metric in $\He$. If iLGs are not suitable building blocks for a quantative theory of rectifiability in $\He$ -- this remains unclear -- which sets could take their place? In $\R^{n}$, the notions of "$m$-dimensional Lipschitz graph" and "bilipschitz image of $\R^{m}$" are fairly interchangeable from the harmonic analysis point of view, at least for the $L^{2}$-theory of SIOs. Indeed, bilipschitz images of $\R^{n - 1}$ inside $\R^{n}$ have big pieces of Lipschitz graphs by \cite[Theorem 1.14]{MR1132876} (combined with a theorem of V\"ais\"al\"a \cite{MR942829} as explained in \cite[p. 857]{MR1132876}). They also admit corona decompositions by Lipschitz graphs; in fact, big pieces imply corona decompositions, which in turn imply $($big pieces$)^{2}$, see \cite{MR2989430,2020arXiv200811544B}. Conversely, every $(n - 1)$-dimensional Lipschitz graph is trivially a bilipschitz image of $\R^{n - 1}$.

In the first Heisenberg group, analogous questions concern the relationship between iLGs and bilipschitz images of the \emph{parabolic plane} $\W := (\R^{2},d_{\mathrm{par}})$, where 
\begin{equation}\label{dpar} d_{\mathrm{par}}((x,t),(\xi,\tau)) = \max\{|x - \xi|,|t - \tau|^{1/2}\}. \end{equation} 
In the sequel, bilipschitz images of the parabolic plane in $\He$ will be called \emph{Rickman rugs} for brevity. In stark contrast to $\R^{n}$, iLGs are not Rickman rugs, not even locally. A counterexample is due to Bigolin and Vittone \cite{MR2603594}. This point cannot be over-emphasised: it is, in fact, an open problem to determine if general iLGs can even be covered by countably many Lipschitz images of compact subsets of $\W$, up to an $\mathcal{H}^{3}$ null set. The strongest partial result is \cite{2019arXiv190610215D}, which yields a positive answer under an "$\epsilon$" of extra regularity.

So, iLGs are definitely not Rickman rugs. Are Rickman rugs iLGs? Yes, almost:
\begin{proposition}\label{toy} A Rickman rug has big pieces of intrinsic Lipschitz graphs. \end{proposition}
This proposition is not difficult: one verifies that Rickman rugs have \emph{big vertical projections} (see Proposition \ref{prop:bvp}) and satisfy the \emph{weak geometric lemma} (see Corollary \ref{wgl}). Then, the proposition follows from \cite[Theorem 1.2]{CFO}, which is a Heisenberg variant of \cite[Theorem 1.14]{MR1132876}. However, the proposition is as unsatisfactory as it is easy: iLGs are not Rickman rugs, so Proposition \ref{toy} wastes information. Given the apparent difficulties in using iLGs in the theory of quantitative rectifiability, Proposition \ref{toy} looks particularly useless. The main theorem has now been sufficiently motivated:
\begin{thm}\label{main} A Rickman rug has a corona decomposition by intrinsic bilipschitz graphs. \end{thm}
Here, an \emph{intrinsic bilipschitz graph} is simply an intrinsic Lipschitz graph which is also a Rickman rug. So, in contrast to Proposition \ref{toy}, one cannot end up from a rock to a hard place by applying Theorem \ref{main}; one perhaps remains on the rock, but at least gains an intrinsic graph parametrisation. 

The notion of a corona decomposition is defined in Section \ref{s:corona}. It is slightly weaker than "big pieces of iLGs", but typically strong enough for applications in harmonic analysis. The recent preprint \cite{2020arXiv200811544B} shows, in general metric spaces, that corona decompositions are equivalent to $($big pieces$)^{2}$. In other words, Rickman rugs have big pieces of sets which have big pieces of iLGs. A corona decomposition implies qualitative rectifiability: every Rickman rug can be covered, up to an $\mathcal{H}^{3}$ null set, by countably many intrinsic \textbf{bi}lipschitz graphs, see Remark \ref{r:qualitativeRectifiability}. The following quantification still remains open:
\begin{question} Do Rickman rugs have big pieces of intrinsic bilipschitz graphs? \end{question}

\begin{remark} With Theorem \ref{main} in hand, an obvious question concerns the regularity of intrinsic bilipschitz graphs, but this will have to wait until another paper. For readers familiar with the theory of intrinsic Lipschitz graphs, the intrinsic \textbf{bi}lipschitz graphs are precisely those whose \emph{characteristic straightening map} $(s,t) \mapsto (s,\tau(s,t))$ is bilipschitz $\R^{2} \to \R^{2}$. (This map is commonly denoted $(s,t) \mapsto (s,\chi(s,t))$, see for example \cite[p. 3]{BCSC} or \cite[Definition 3.1]{MR3984100}). So, the \emph{characteristics} of an intrinsic bilipschitz function neither converge nor diverge: $|\tau(s,t_{1}) - \tau(s,t_{2})| \sim |t_{1} - t_{2}|$ for all $s,t_{1},t_{2} \in \R$. Maintaining this property is a challenge in proving Theorem \ref{main}, see condition $(\tau_{2})$ of Proposition \ref{mainProp2}.

Since it is not necessary for this paper, and would take a non-trivial amount of space, I omit a discussion on (the existence and uniqueness of) characteristic curves associated with general intrinsic Lipschitz functions; see \cite{BCSC} or \cite[Section 2.3]{2020arXiv200412522N}.

 \end{remark}

\begin{remark} The paper \cite{2020arXiv200412522N} is entitled "Foliated corona decompositions", and its main result, \cite[Theorem 1.17]{2020arXiv200412522N}, states that an iLG in $\He$ has a \emph{foliated corona decomposition}. As far as I understand, the definition of "corona decomposition" in \cite{2020arXiv200412522N} is different than the one used in the present paper (which is the one defined e.g. by David and Semmes in \cite{DS1}). Notably, one cannot deduce Theorem \ref{main} by concatenating Proposition \ref{toy} and \cite[Theorem 1.17]{2020arXiv200412522N}, and the arguments here have no overlap with those in \cite{2020arXiv200412522N} -- as far as this is possible for two papers on closely related questions. \end{remark}

\subsection{Corona decompositions}\label{s:corona} In this section, I define the notion of corona decompositions properly, and also fix metric notations for $\He$ and $\W$. Corona decompositions can even be meaningfully defined in metric spaces, see \cite{2020arXiv200811544B}, and the notion below is just a specialisation to Rickman rugs and intrinsic bilipschitz graphs. 

Let $\mathcal{D}$ be the family of dyadic parabolic rectangles, that is rectangles of the form $Q = I \times J \subset \W$, where $I,J \in \mathcal{D}_{\R}$ are standard dyadic intervals in $\R$, and $|J| = |I|^{2}$. In other words, dyadic parabolic rectangles have the form
\begin{displaymath} Q = [k 2^{-n},(k + 1)2^{-n}) \times [l 4^{-n},(l + 1)4^{-n}), \qquad k,l,n \in \Z. \end{displaymath}
The \emph{side-length} of a rectangle $Q = I \times J \in \mathcal{D}$ is defined as $\ell(Q) := |I|$. Note that, since the parabolic metric was defined with the $L^{\infty}$-formula \eqref{dpar}, we have $\diam(Q) = \ell(Q)$ for $Q \in \mathcal{D}$. For $Q \in \mathcal{D}$, let $c_{Q}$ be the centre of $Q$. For $\lambda > 0$, define
\begin{displaymath} \lambda Q := B_{\mathrm{par}}(c_{Q},\lambda \ell(Q)) := \{w \in \W : d_{\mathrm{par}}(w,c_{Q}) \leq \lambda \ell(Q)\}. \end{displaymath}
This notation may appear a little awkward because $1Q \neq Q$, but it has a benefit: we will often define concepts for (parabolic) balls, such as "$\beta(B(w,r))$", and then the same definitions apply without separate comment to $\lambda Q$. Also, since $d_{\mathrm{par}}$ was defined as the $L^{\infty}$-metric on $\W$, parabolic balls of radius $r$ are rectangles of (Euclidean) dimensions $2r \times 2r^{2}$, and in particular $\lambda Q$ is a rectangle of dimensions $2\lambda \ell(Q) \times 2(\lambda \ell(Q))^{2}$.

Corona decompositions involve the notion of \emph{trees} in $\mathcal{D}$, which are defined next:

\begin{definition}[Tree] A family $\mathcal{T} \subset \mathcal{D}$ is a \emph{tree} if the following conditions are met:
\begin{itemize}
\item[(T1) \phantomsection\label{T1}] There is a \emph{root} $Q(\calT) \in \mathcal{T}$ such that all rectangles in $\mathcal{T}$ are contained in $Q(\mathcal{T})$.
\item[(T2) \phantomsection\label{T2}] Trees are \emph{connected}: if $Q_{1},Q_{2},Q_{3} \in \calD$ are rectangles with $Q_{1} \subset Q_{2} \subset Q_{3}$, and $Q_{1},Q_{3} \in \mathcal{T}$, then also $Q_{2} \in \mathcal{Q}$.
\item[(T3) \phantomsection\label{T3}] Trees are \emph{coherent}: if $Q \in \mathcal{T}$, then either all the $\mathcal{D}$-children of $Q$ are contained in $\mathcal{T}$, or then none of them are.
\end{itemize}
The rectangles $Q \in \mathcal{T}$ with no $\mathcal{D}$-children in
$\mathcal{D}$ are called the \emph{leaves of $\mathcal{T}$}, denoted $\mathbf{Leaves}(\mathcal{T})$.
\end{definition}

We may then define corona decompositions:

\begin{definition}[Corona decomposition] Let $f \colon \W \to \He$ be an $M$-bilipschitz map, $M \geq 1$. We say that $f$ (or its image) has a \emph{corona decomposition by intrinsic bilipschitz graphs} if the following holds for every parameter $\eta > 0$. The family $\mathcal{D}$ can be partitioned into disjoint collections of \emph{bad rectangles} $\mathcal{B}$ and \emph{good rectangles} $\mathcal{G}$, which may depend on $\eta$, and which have the following properties.
\begin{itemize}
\item[(C1) \phantomsection\label{C1}] The bad rectangles satisfy a Carleson packing condition: there exists a constant $C_{1} = C_{1}(M,\eta) > 0$ such that
\begin{displaymath} \mathop{\sum_{Q \in \mathcal{B}}}_{Q \subset Q_{0}} |Q| \leq C_{1}|Q_{0}|, \qquad Q_{0} \in \mathcal{D}. \end{displaymath}
\item[(C2) \phantomsection\label{C2}] The good rectangles are partitioned into a countable family of \emph{trees} $\mathcal{T}_{1},\mathcal{T}_{2},\ldots \subset \mathcal{G}$. Every tree $\mathcal{T}_{j}$ is associated with an intrinsic bilipschitz graph $\Gamma_{j} \subset \He$ in such a way that
\begin{equation}\label{form143} \dist(f(w),\Gamma_{j}) \leq \eta \cdot \ell(Q), \qquad w \in 2Q, \, Q \in \mathcal{T}_{j}. \end{equation}
\item[(C3) \phantomsection\label{C3}] The top rectangles $Q(\mathcal{T}_{j})$ of the trees $\mathcal{T}_{1},\mathcal{T}_{2},\ldots$ satisfy a Carleson packing condition: there exists a constant $C_{2} = C_{2}(M,\eta) > 0$ such that 
\end{itemize}
\begin{displaymath} \mathop{\sum_{j \geq 1}}_{Q(\mathcal{T}_{j}) \subset Q_{0}} |Q(\mathcal{T}_{j})| \leq C_{2}|Q_{0}|, \qquad Q_{0} \in \mathcal{D}. \end{displaymath}
\end{definition}

The choice of "$2$" in \eqref{form143} is quite arbitrary, and could be replaced by any number $k \geq 2$. In fact, the existence of a corona decomposition with factor "$2$" formally implies the existence of corona decompositions with factor "$k \geq 2$", see \cite[p. 20]{DS1}.

\begin{notation}
The notation "$|A|$" in \nref{C1} and \nref{C3} refers to Lebesgue measure on $\W$ (or $\R^{2}$), which coincides, up to a multiplicative constant, with the $3$-dimensional Hausdorff measure $\mathcal{H}^{3} = \mathcal{H}^{3}_{d_{\mathrm{par}}}$. In particular, $\mathcal{H}^{3}(Q) \sim |Q| \sim \ell(Q)^{3}$ for all $Q \in \mathcal{D}$. The distance in \eqref{form143} is defined relative to the metric $d(p,q) := \|q^{-1} \cdot p\|$, where
\begin{displaymath} \|(x,y,t)\| := \max\{\sqrt{x^{2} + y^{2}},\sqrt{|t|}\}, \end{displaymath}
and $(x_{1},y_{1},t_{1}) \cdot (x_{2},y_{2},t_{2}) := (x_{1} + x_{2},y_{1} + y_{2},t_{1} + t_{2} + \tfrac{1}{2}(x_{1}y_{2} - x_{2} y_{1}))$ is the Heisenberg group law. This metric $d$ will always be used to measure distances in $\He$, and to define metric concepts such as "$\diam$", "$\dist$", closed balls $B(p,r) := \{q \in \He : d(p,q) \leq r\}$, and Hausdorff measures "$\mathcal{H}^{s}$", $s \geq 0$. 

Distances and metric concepts in the parabolic plane $\W$ are defined relative to the metric $d_{\mathrm{par}}$. With the choices of metrics $d$ and $d_{\mathrm{par}}$, the parabolic plane is isometric, and also group-theoretically isomorphic, to the (Abelian) subgroup $\{(0,y,t) : y,t \in \R\} \subset \He$ equipped with the metric $d$. For this reason, the parabolic plane and this subgroup can be safely identified, and the notation "$\W$" will also be used for $\W = \{(0,y,t) : y,t \in \R\}$. In particular, points in $\W$ are denoted "$(y,t)$" if $\W$ refers to the parabolic plane, and "$(0,y,t)$" if $\W$ refers to the $yt$-plane inside $\He$. \end{notation}

\begin{remark}\label{r:qualitativeRectifiability} If the reader has not seen a corona decomposition before, start by thinking that $\mathcal{B} = \emptyset$, and $\mathcal{G} = \mathcal{T} = \mathcal{D}$. In this special situation, there is only one "approximating" intrinsic bilipschitz graph $\Gamma = \Gamma_{\mathcal{D}}$, and \eqref{form143} implies that $f(w) \in \Gamma$ for all $w \in \W$. This can be seen by taking a sequence of rectangles in $\mathcal{D}$ converging to $w$. In general, the Carleson packing conditions in \nref{C1} and \nref{C3} imply the following for Lebesgue almost every $w \in \W$: there exists a tree $\mathcal{T}_{j}$ which contains an infinite "branch" of dyadic rectangles converging to $w$. It follows that $f(w) \in \Gamma_{j}$ by \eqref{form143}. Consequently, the set
\begin{displaymath} \bigcup_{j = 1}^{\infty} \{w \in \W : f(w) \notin \Gamma_{j}\} = f^{-1}\Big(f(\W) \, \setminus \, \bigcup_{j = 1}^{\infty} \Gamma_{j} \Big) \end{displaymath}
has $\mathcal{H}^{3}$ measure zero. Since $\mathcal{H}^{3}|_{f(\W)} \ll f_{\sharp}\mathcal{H}^{3}|_{\W}$ by the bilipschitz property of $f$, this implies that $f(\W)$ is countably rectifiable by intrinsic bilipschitz graphs. \end{remark}

\subsection{Outline of the paper} Section \ref{s:prelim} contains a number of preliminary result on "dyadic" lines and "corner points" in $\W$, and horizontal lines in $\He$. Section \ref{s:rulerCoefficients} focuses on the approximability of Rickman rugs by vertical planes: the main result is Corollary \ref{wgl}, which implies that Rickman rugs $f(\W)$ are approximable by vertical planes in $\He$ in a weak, but quantitative, sense (the "weak geometric lemma"). The proof is based on the notion of \emph{ruler coefficients}, which measure the linear approximability of $f$ along individual horizontal lines. If the ruler coefficients are sufficiently small -- and they often are by Proposition \ref{propCarleson} -- then $f(\W)$ is also approximable by vertical planes, see Corollary \ref{cor2}.

Section \ref{s:projections} verifies that Rickman rugs have \emph{big vertical projections}. The proof is elementary if the rug is assumed to be "$\epsilon$-flat" to begin with, see Proposition \ref{projProp}. But the results in Section \ref{s:rulerCoefficients} verify that Rickman rugs are $\epsilon$-flat "at almost all scales and locations", and these observations combined give the result.

Section \ref{s:treeSelection} constructs the partition $\mathcal{D} = \mathcal{B} \cup \mathcal{G}$, and the trees $\mathcal{T}_{j}$, which appear in the definition of the corona decomposition. This part of the proof of Theorem \ref{main} is not difficult, and follows the ideas of David and Semmes \cite{DS1} closely. The main novelty is that, unlike in \cite{DS1}, the information from Section \ref{s:projections} on big vertical projections is employed to prove the Carleson packing estimate \nref{C3} for the top rectangles $Q(\mathcal{T}_{j})$, $j \geq 1$.

The most involved, and novel, part of the paper starts in Section \ref{s:bilipschitzGraphs}, where we begin the construction of the well-approximating intrinsic bilipschitz graphs $\Gamma_{j}$ required by \nref{C2} (or \eqref{form143}) of the corona decomposition. Fix a tree $\mathcal{T} = \mathcal{T}_{j}$. The following two facts will be available by this point: (a) for $Q \in \mathcal{T}$, the image $f(Q)$ is $\epsilon$-approximable by a vertical plane $\V_{Q}$, and (b) the planes $\V_{Q}$ corresponding to various rectangles $Q \in \mathcal{T}$ form an angle $\angle(\V_{Q},\W) \leq \eta$ with some fixed vertical plane $\W$, which may be assumed to be $\W = \{(0,y,t) : y,t \in \R\}$ with no loss of generality. If the reader is familiar with the corona decomposition from \cite{DS1}, these are verbatim the same conditions required, there, to construct well-approximating Lipschitz graphs in $\R^{n}$. If the aim of the current paper was only to prove a corona decomposition with intrinsic Lipschitz (not bilipschitz) graphs, one could probably adapt the argument of \cite{DS1}.

The catch is that the conditions (a)-(b) alone are -- quite likely -- not strong enough to produce an intrinsic \textbf{bi}lipschitz graph. There are no recorded counterexamples, but quite likely one can construct an intrinsic Lipschitz graph $\Gamma \subset \He$ with the following three properties for any $\epsilon,\eta > 0$. First, $\Gamma$ is $\epsilon$-approximable by vertical planes at all scales and locations. Second, all the the best-approximating vertical planes $\V$ satisfy $\angle(\V,\W) \leq \eta$. Third, $\Gamma$ is not a Rickman rug. The first two conditions do not seem to be powerful enough to prevent a phenomenon of "crossing characteristics". So, in addition to the properties (a)-(b), we need to keep in mind that the Rickman rug, which we are trying to approximate by an intrinsic bilipschitz graph $\Gamma$, is a bilipschitz image of $\W$. Once this information is married with (a) and (b), it is possible to secure a crucial separation property for the characteristics of $\Gamma$, which guarantees that $\Gamma$ is an intrinsic bilipschitz graph over $\W$. The separation property is stated in $(\tau_{2})$ of Proposition \ref{mainProp2}.

\section{Acknowledgements} I'm grateful to Katrin F\"assler for extensive discussions during the early stages of this project; these were invaluable in bringing some rigour to the first sketches of the corona decomposition. Also, Katrin helped in proving multiple lemmas Sections \ref{s:prelim} and \ref{s:rulerCoefficients}. I am also grateful to Sebastiano Nicolussi Golo and Michele Villa for many discussions during the project; in particular, Sebastiano taught me the idea behind Lemma \ref{lemma1}.

\section{Preliminaries}\label{s:prelim}

\subsection{Intrinsic Lipschitz functions and graphs}\label{s:iLip} For more information on intrinsic Lipschitz functions and graphs, see \cite{FSS,MR3587666}. The \emph{$1$-codimensional vertical subgroups} of $\He$ are the $yt$-plane $\W = \{(0,y,t) : y,t \in \R\}$, and any rotation of $\W$ around the $t$-axis. To every $1$-codimensional vertical subgroup, associate a \emph{complementary horizontal subgroup}, which, as a subset or $\R^{3}$, is the orthogonal complement of the vertical subgroup. The complementary horizontal subgroup of the $yt$-plane is, therefore, the $x$-axis $\mathbb{L} = \{(x,0,0) : x \in \R\}$. 

Let $\W,\mathbb{L}$ be a pair of complementary subgroups, as above, and let $\phi \colon \W \to \mathbb{L}$ be a map. The \emph{intrinsic graph of $\phi$} is the set $\Gamma(\phi) := \{w \cdot \phi(w) : w \in \W\} \subset \He$. The intrinsic graph $\Gamma(\phi)$ can be parametrised by the \emph{intrinsic graph map} $\Phi(w) := w \cdot \phi(w)$. If a set $\Gamma \subset \He$ has a parametrisation $\Gamma = \Phi(\W) = \Gamma(\phi)$ as the intrinsic graph of a function $\phi \colon \W \to \mathbb{L}$, we say that $\Gamma$ is an \emph{intrinsic graph over $\W$}.

Intrinsic graphs over a vertical subgroup $\W$ can be characterised as those subsets of $\He$ for which the \emph{vertical projection} $\Pi_{\W} \colon \He \to \W$ is a bijection. The vertical projection $\Pi_{\W}$, here, is defined by $\Pi_{\W}(p) := w$, where $p = w \cdot l$ is the unique splitting of $p$ as a group product of elements $w \in \W$ and $l \in \mathbb{L}$. This splitting also gives rise to the \emph{horizontal projection} $\Pi_{\mathbb{L}}(p) := l$. When $\W$ is the $yt$-plane, one may compute that $\Pi := \Pi_{\W}$ and $\Pi_{\mathbb{L}}$ have the explicit expressions 
\begin{displaymath} \Pi(x,y,t) := (0,y,t + \tfrac{1}{2}xy) \quad \text{and} \quad \Pi_{\mathbb{L}}(x,y,t) = (x,0,0), \qquad (x,y,t) \in \He. \end{displaymath}

Let $L > 0$. An intrinsic graph $\Gamma \subset \He$ over $\W$ is an \emph{intrinsic $L$-Lipschitz graph over $\W$} if $\Gamma$ satisfies the cone condition
\begin{equation}\label{coneCondition} [p \cdot \mathcal{C}_{\alpha}(\W,\mathbb{L})] \cap \Gamma = \{p\}, \qquad p \in \Gamma, \, 0 < \alpha < \tfrac{1}{L}. \end{equation}
Here $\mathcal{C}_{\alpha}(\W,\mathbb{L}) = \{p \in \He : \|\Pi_{\W}(p)\| \leq \alpha \|\Pi_{\mathbb{L}}(p)\|\}$, so for instance $\mathcal{C}_{0}(\W,\mathbb{L}) = \mathbb{L}$. It is easy to check that cone condition \eqref{coneCondition} can be equivalently written as
\begin{equation}\label{coneCondition2} \|\Pi_{\mathbb{L}}(q^{-1} \cdot p)\| \leq L\|\Pi_{\W}(q^{-1} \cdot p)\|, \qquad p,q \in \Gamma. \end{equation}
If $\Gamma = \Gamma(\phi)$ is an intrinsic $L$-Lipschitz graph over $\W$, we say that $\phi$ is an intrinsic $L$-Lipschitz function.

As we already defined after the statement of Theorem \ref{main}, an \emph{intrinsic bilipschitz graph over $\W$} is an intrinsic Lipschitz graph over $\W$ which is bilipschitz equivalent to $\W$. It is worth emphasising that the intrinsic graph map $\Phi \colon \W \to \Gamma(\phi)$ is \textbf{almost never} a bilipschitz map. For example, if $\phi(0,y,t) \equiv (c,0,0) \neq 0$, then $\Phi(0,y,t) = (0,y,t) \cdot (c,0,0)$ is a \emph{right translation} by a non-zero vector, and right translations are not bilipschitz maps in $\He$. For example, $\Phi$ here maps every horizontal line on $\W$ to a non-horizontal line on $\Phi(\W)$ (which in this example is a vertical plane parallel to $\W$).

\subsection{Lemmas on dyadic lines and corners} We already introduced the dyadic parabolic rectangles "$\mathcal{D}$" in Section \ref{s:corona}, but we continue with a few additional pieces of notation. For $n \in \Z$ fixed, we write $\mathcal{D}_{n} := \{Q \in \mathcal{D} : \ell(Q) = 2^{-n}\}$ for the \emph{dyadic parabolic rectangles of generation $n$.} Every parabolic rectangle $Q \in \mathcal{D}_{n}$ is partitioned by $8$ elements of $\mathcal{D}_{n + 1}$, and these are known as the \emph{children of $Q$}, denoted $\mathbf{ch}(Q)$. The rectangle $Q$ is then the \emph{parent} of its children. Every rectangle $Q = I \times J \in \mathcal{D}$ has four \emph{corners}, namely $\mathcal{C}(Q) = \partial I \times \partial J$. The set of corners of all parabolic rectangles of generation $n$ will be denoted $\mathcal{C}(\mathcal{D}_{n})$. 

We will need a fairly tedious lemma, below, on the dyadic structure of the corner points $\mathcal{C}(\mathcal{D}_{n})$. For this purpose, let $\mathcal{L}$ be the set of all horizontal lines in $\W$, and let $\mathcal{L}_{n} := \{\R \times \{k \cdot 2^{-n}\} : k \in \Z\}$ be the \emph{dyadic horizontal lines of generation $n$.} 

For purposes of "ordering" the lines in $\mathcal{L}$, it will often be convenient to identify them with their second coordinates. Let
\begin{displaymath} \pi_{1}(y,t) := y \quad \text{and} \quad \pi_{2}(y,t) := t, \qquad (y,t) \in \W, \end{displaymath}
be the coordinate projections $\W \to \R$. If $\ell \in \mathcal{L}$, then $\pi_{2}(\ell)$ is a singleton, and we identify $\ell \cong t$, where $\pi_{2}(\ell) = \{t\}$. With this in mind, we will, for example, write $\ell_{1} > \ell_{2}$ if $\pi_{2}(\ell_{j}) = \{t_{j}\}$ with $t_{1} > t_{2}$. Also, for $\ell \in \mathcal{L}$ with $\pi_{2}(\ell) = \{r\}$, and for $t \in \R$, the notation $\ell + t$ will mean the line $\ell' \in \mathcal{L}$ with $\pi_{2}(\ell') = \{r + t\}$.

\begin{remark} In the next lemma, we will use the following simple observation. If $Q \in \mathcal{D}_{n}$, then the horizontal edges of $Q$ lie on consecutive lines in $\mathcal{L}_{2n}$. The horizontal edges of the $8$ rectangles in $\mathbf{ch}(Q) \subset \mathcal{D}_{n + 1}$ span altogether $5$ horizontal lines $\ell_{1} < \ell_{2} < \ell_{4} < \ell_{4} < \ell_{5}$, which satisfy $\ell_{1},\ell_{5} \in \mathcal{L}_{2n}$, $\ell_{3} \in \mathcal{L}_{2n + 1} \, \setminus \, \mathcal{L}_{2n}$, and $\ell_{2},\ell_{4} \in \mathcal{L}_{2n + 2} \, \setminus \, \mathcal{L}_{2n + 1}$. \end{remark}

\begin{lemma}\label{lemma10} Let $n \in \Z$, let $\ell \in \mathcal{L}_{n}$, and let $w \in \ell \cap \mathcal{C}(Q)$ for some $Q \in \mathcal{D}_{\ceil{n/2}}$. Let $\hat{Q} \in \mathcal{D}_{\ceil{n/2} - 1}$ be the parent of $Q$ in $\mathcal{D}$.
\begin{itemize}
\item[(a)] If $n \in 2\Z$ and $\ell \notin \mathcal{L}_{n - 1}$, then either
\begin{itemize}
\item both $\ell + 2^{-n}$ and $\ell - 3 \cdot 2^{-n}$ contain an edge of $\hat{Q}$, or
\item both $\ell - 2^{-n}$ and $\ell + 3 \cdot 2^{-n}$ contain an edge of $\hat{Q}$.
\end{itemize}
\item[(b)] If $n \in 2\Z$ and $\ell \in \mathcal{L}_{n - 1} \, \setminus \, \mathcal{L}_{n - 2}$, then both $\ell + 2^{-n + 1}$ and $\ell - 2^{-n + 1}$ contain an edge of $\hat{Q}$.
\item[(c)] If $n \in 2\Z$ and $\ell \in \mathcal{L}_{n - 2}$, then $\ell$ contains an edge of $\hat{Q}$. 
\item[(d)] If $n \in 2\Z + 1$ and $\ell \notin \mathcal{L}_{n - 1}$, then both $\ell + 2^{-n}$ and $\ell - 2^{-n}$ contain an edge of $\hat{Q}$.
\item[(e)] If $n \in 2\Z + 1$ and $\ell \in \mathcal{L}_{n - 1}$, then $\ell$ contains an edge of $\hat{Q}$.
\end{itemize}
\end{lemma}

\begin{proof} We first prove (a), (b), and (c), which have in common the assumption $n = 2m \in 2\Z$. This implies $\ceil{n/2} = m$ and $\hat{Q} \in \mathcal{D}_{m - 1}$. The rectangles in $\mathbf{ch}(\hat{Q})$ have dimensions $2^{-m} \times 4^{-m}$. Let $\ell_{1} < \ell_{2} < \ldots < \ell_{5}$ be the five horizontal lines which are spanned by $\mathcal{C}(\mathbf{ch}(\hat{Q}))$. Thus, the lines $\ell_{j + 1} = \ell_{j} + 4^{-m}$, and $\ell_{1},\ell_{5}$ are spanned by $\mathcal{C}(\hat{Q})$. Since $w \in \ell \cap \mathcal{C}(\mathbf{ch}(\hat{Q}))$, we have $\ell \in \{\ell_{1},\ldots,\ell_{5}\}$. 

To prove (a), assume that $\ell \notin \mathcal{L}_{n - 1}$. We claim that $\ell \in \{\ell_{2},\ell_{4}\}$. Indeed, since $\hat{Q} \in \mathcal{D}_{m - 1}$, we have $\ell_{1},\ell_{5} \in \mathcal{L}_{2m - 2}$ and $\ell_{3} \in \mathcal{L}_{2m - 1}$. Since we assume that $\ell \notin \mathcal{L}_{n - 1} = \mathcal{L}_{2m - 1}$, we infer that $\ell \notin \{\ell_{1},\ell_{3},\ell_{5}\}$, and hence $\ell \in \{\ell_{2},\ell_{4}\}$. If $\ell = \ell_{2}$, then $\ell - 2^{-n} = \ell - 4^{-m}$ and $\ell + 3 \cdot 2^{-n}$ contain an edge of $\hat{Q}$. If $\ell = \ell_{4}$, then $\ell + 2^{-n}$ and $\ell - 3 \cdot 2^{-n}$ contain an edge of $\hat{Q}$.

To prove (b), assume that $\ell \in \mathcal{L}_{n - 1} \, \setminus \, \mathcal{L}_{n - 2}$. As before, $\ell \in \{\ell_{1},\ldots,\ell_{5}\}$, and since $\ell_{1},\ell_{5} \in \mathcal{L}_{2m - 2} = \mathcal{L}_{n - 2}$, we have $\ell \in \{\ell_{2},\ell_{3},\ell_{4}\}$. But $\ell_{2},\ell_{4} \notin \mathcal{L}_{2m - 1} = \mathcal{L}_{n - 1}$, so in fact $\ell = \ell_{3}$. Therefore, the lines $\ell \pm 2 \cdot 4^{-m} = 2^{-n + 1}$ contain the horizontal edges of $\hat{Q}$. To prove (c), note that the assumption $\ell \in \mathcal{L}_{n - 2} = \mathcal{L}_{2m - 2}$ forces $\ell \in \{\ell_{1},\ell_{5}\}$, which is what we wanted.

Let us then prove (d) and (e), where we assume that $n = 2m + 1$ for some $n \in \Z$. Hence $\ceil{n/2} = m + 1$, and $\hat{Q} \in \mathcal{D}_{m}$. This time the children of $\hat{Q}$ have dimensions $2^{-m - 1} \times 4^{-m - 1}$. With the previous notation, we still have $\ell \in \{\ell_{1},\ldots,\ell_{5}\}$. To prove (b), assume that $\ell \notin \mathcal{L}_{n - 1}$. Then, we claim that $\ell = \ell_{3}$. Indeed, now $\ell_{1},\ell_{5} \in \mathcal{L}_{2m}$, so $\ell \notin \{\ell_{1},\ell_{5}\}$ by the assumption $\ell \notin \mathcal{L}_{n - 1} = \mathcal{L}_{2m}$. Moreover, $\ell_{2},\ell_{4} \notin \mathcal{L}_{2m + 1} = \mathcal{L}_{n}$, so $\ell \notin \{\ell_{2},\ell_{4}\}$. Hence $\ell = \ell_{3}$, and it follows that the horizontal edges of $\hat{Q}$ are contained on $\ell \pm 2 \cdot 4^{-m - 1} = \ell \pm 2^{-n}$.

To prove (d), assume finally that $\ell \in \mathcal{L}_{n - 1} = \mathcal{L}_{2m}$. Since $\ell_{1},\ell_{5}$ are the only lines in $\mathcal{L}_{2m} \cap \{\ell_{1},\ldots,\ell_{5}\}$, we infer that $\ell \in \{\ell_{1},\ell_{5}\}$, as claimed. \end{proof}

We will use the lemma only via the following corollary:

\begin{cor}\label{cor5} Let $\ell_{1} < \ell_{2} < \ell_{3} < \ell_{4} < \ell_{5}$ be consecutive lines in $\mathcal{L}_{n}$ such that $\ell_{2},\ell_{4} \in \mathcal{L}_{n - 1}$ and hence $\ell_{1},\ell_{3},\ell_{5} \in \mathcal{L}_{n} \, \setminus \, \mathcal{L}_{n - 1}$.  Assume that one of the lines $\ell_{j}$ contains the corner of a rectangle $Q \in \mathcal{D}_{\ceil{n/2}}$. Then, the dyadic parent $\hat{Q} \in \mathcal{D}_{\ceil{n/2} - 1}$ has a corner on $\ell_{2} \cup \ell_{4}$. \end{cor}

\begin{proof} We may assume that the corner "$w$" of $Q \in \mathcal{D}_{\ceil{n/2}}$ is contained on $\ell_{1} \cup \ell_{2} \cup \ell_{3}$ (the cases $w \in \ell_{4} \cup \ell_{5}$ are symmetric to $w \in \ell_{1} \cup \ell_{2}$). We then consider the following cases:
\begin{itemize}
\item Assume $n \in 2\Z$ and $w \in \ell_{1}$. Since $\ell_{1} \in \mathcal{L}_{n} \, \setminus \, \mathcal{L}_{n - 1}$, we are in case (a) of Lemma \ref{lemma10}. We infer that either $\ell_{1} + 2^{-n} = \ell_{2}$ or $\ell_{1} + 3 \cdot 2^{-n} = \ell_{4}$ contains an edge of $\hat{Q}$.
\item Assume $n \in 2\Z$ and $w \in \ell_{2}$. Since $\ell_{2} \in \mathcal{L}_{n - 1}$, we are in case (b) or (c) of Lemma \ref{lemma10}, depending on whether $\ell_{2} \notin \mathcal{L}_{n - 2}$ or $\ell_{2} \in \mathcal{L}_{n - 2}$. In both cases, either $\ell_{2}$ or $\ell_{2} + 2^{-n + 1} = \ell_{4}$ contains an edge, hence a corner, of $\hat{Q}$.
\item Assume $n \in 2\Z$ and $w \in \ell_{3}$. Since $\ell_{3} \in \mathcal{L}_{n} \, \setminus \, \mathcal{L}_{n - 1}$, we are in case (a) of Lemma \ref{lemma10}. Consequently, either $\ell_{3} + 2^{-n} = \ell_{4}$ or $\ell_{3} - 2^{-n} = \ell_{2}$ contains an edge of $\hat{Q}$.
\item Assume $n \in 2\Z + 1$, and $w \in \ell_{1} \in \mathcal{L}_{n} \, \setminus \, \mathcal{L}_{n - 1}$. Then we are in case (d) of Lemma \ref{lemma10}, and we infer that $\ell_{1} + 2^{-n} = \ell_{2}$ contains an edge of $\hat{Q}$.
\item Assume $n \in 2\Z + 1$ and $w \in \ell_{2}$. Then we are in case (e) of Lemma \ref{lemma10}, and $\ell_{2}$ contains an edge of $\hat{Q}$.
\item Assume $n \in 2\Z + 1$, and $w \in \ell_{3} \in \mathcal{L}_{n} \, \setminus \, \mathcal{L}_{n - 1}$. Then we are in case (d) of Lemma \ref{lemma10}, and $\ell_{3} \pm 2^{-n} = \{\ell_{2},\ell_{4}\}$ both contain an edge of $\hat{Q}$.
\end{itemize}
We have covered all the cases, so the proof is complete. \end{proof}

\subsection{Lemmas on horizontal lines}

We start with an elementary lemma on horizontal lines, which quantifies the fact that if two horizontal $L_{1},L_{2} \subset \He$ are "close" to each other, then the coefficients in their parametrisations do not differ very much.
\begin{lemma}\label{lemma0} Let $L_{1},L_{2} \subset \He$ be horizontal lines which are not contained in any translate of the $xt$-plane. Then $L_{1}$ and $L_{2}$ can be written as
\begin{displaymath} L_{j} = \{(a_{j}y + b_{j},y,\tfrac{1}{2}b_{j}y + c_{j}) : y \in \R\}, \qquad j \in \{1,2\}, \end{displaymath}
for a unique triple $a_{j},b_{j},c_{j} \in \R$. Assume that
$|a_{1}| \leq \Sigma$ for some $\Sigma >
0$. Assume, moreover, that there exist two points $p,q \in \He$ with $d(p,q) = r$
 such that $\dist(p,L_{j}) \leq \epsilon r$ and $\dist(q,L_{j}) \leq \epsilon r$ for some $\epsilon \in (0,\tfrac{1}{64(1+\Sigma)})$. Then, $|a_{2}| \leq 2\Sigma$,
\begin{displaymath} |a_{1} - a_{2}| \lesssim (1+\Sigma)^2\epsilon, \quad |b_{1} - b_{2}| \lesssim (1+\Sigma)\,r\epsilon,
\quad \text{and} \quad |c_{1} - c_{2}| \lesssim
(1+\Sigma)(r\epsilon)^{2}.
\end{displaymath}
\end{lemma}

\begin{proof} We may assume without loss of generality that $r = 1$, since dilating the lines $L_{1},L_{2}$ does not change the slopes $a_{j}$, and changes the parameters $b_{j}$ and $c_{j}$ by multiplicative factors $r$ and $r^{2}$, respectively. After left translation (which notably does not affect the hypothesis $|a_{1}| \leq \Sigma$), we may also arrange that $p = 0$. We use the following well-known fact: if $p_{1},p_{2} \in L$ are points on a common horizontal line, then $|\pi(p_{1}) - \pi(p_{2})| = d(p_{1},p_{2})$, where $\pi(z,t) := z$ is the $1$-Lipschitz coordinate projection to $\R^{2}$.

The first priority will be to show $|a_{2}| \leq 2\Sigma$. To this end, let $p_{j},q_{j} \in L_{j}$, $j \in \{1,2\}$, be the closest points of $p,q$ on $L_{j}$, so $\|p_{j}\| \leq \epsilon$ and $d(q,q_{j}) \leq \epsilon$ for $j \in \{1,2\}$. Then,
\begin{displaymath} |\pi(p) - \pi(q)| \geq |\pi(p_{1}) - \pi(q_{1})| - |\pi(p) - \pi(p_{1})| - |\pi(q) - \pi(q_{1})| = d(p_{1},q_{1}) - 2\epsilon \geq d(p,q) - 4\epsilon \geq \tfrac{1}{2}.  \end{displaymath}
Further,
\begin{displaymath} |\pi(p) - \pi(p_{j})| \leq \epsilon \quad \text{and} \quad |\pi(q) - \pi(q_{j})| \leq \epsilon, \qquad j \in \{1,2\}. \end{displaymath}
Since $\epsilon \leq \tfrac{1}{4}$, it follows that
\begin{equation}\label{form137} |\pi(p_{2}) - \pi(q_{2})| \geq \tfrac{1}{4} \quad \text{and} \quad \max\{|\pi(p_{1}) - \pi(p_{2})|,|\pi(q_{1}) - \pi(q_{2})|\} \leq 2\epsilon. \end{equation}
Note that the point $p$ no longer appears in \eqref{form137}. To use \eqref{form137} to infer the estimate $|a_{2}| \leq 2|a_{1}|$, we need to write the points $\pi(p_{j}),\pi(q_{j}) \in \pi(L_{j})$ in coordinates:
\begin{displaymath} \pi(p_{j}) =: (a_{j}x_{j} + b_{j},x_{j}) \quad \text{and} \quad \pi(q_{j}) := (a_{j}y_{j} + b_{j},y_{j}), \qquad j \in \{1,2\}. \end{displaymath}
Then, two applications of the second estimate in \eqref{form137} implies 
\begin{align} 4\epsilon & \geq |[(a_{2}x_{2} + b_{2}) - (a_{1}x_{1} + b_{1})] - [(a_{2}y_{2} + b_{2}) - (a_{1}y_{1} + b_{1})]| \notag\\
& = |(a_{2} - a_{1})x_{2} - a_{1}(x_{1} - x_{2}) - (a_{2} - a_{1})y_{2} + a_{1}(y_{1} - y_{2})| \notag\\
& \geq |a_{2} - a_{1}||x_{2} - y_{2}| - |a_{1}||(x_{1} - x_{2}) + (y_{1} - y_{2})| \notag\\
&\label{form140} \geq |a_{2} - a_{1}||x_{2} - y_{2}| - 4\epsilon \Sigma. \end{align} 
So, if -- contrary to what we claim -- it were the case that $|a_{2}| \geq 2\Sigma \geq 2|a_{1}|$, then
\begin{equation}\label{form138} 4\epsilon \geq \tfrac{1}{2}|a_{2}||x_{2} - y_{2}| - 2\epsilon \Sigma. \end{equation}
On the other hand, by the first estimate in \eqref{form137}, we have
\begin{equation}\label{form139} \max\{|x_{2} - y_{2}|,|a_{2}||x_{2} - y_{2}|\} \geq \tfrac{1}{4}. \end{equation} 
If $|x_{2} - y_{2}| \geq \tfrac{1}{4}$, then \eqref{form138} and the assumption $64 \epsilon (1 + \Sigma) < 1$ yields $4\epsilon \geq \tfrac{1}{4}\Sigma - 2\epsilon \Sigma \geq \tfrac{1}{8}\Sigma$, a contradiction. And if $|a_{2}||x_{2} - y_{2}| \geq \tfrac{1}{4}$ instead, then \eqref{form138} yields $4 \epsilon \geq \tfrac{1}{8} - 2\epsilon \Sigma \geq \tfrac{1}{16}$, another contradiction. So, we conclude that $|a_{2}| \leq 2\Sigma$, as claimed.

The estimates for $|a_{2} - a_{1}|$ and $|b_{2} - b_{1}|$ are now straightforward. First, a combination of \eqref{form140}, \eqref{form139}, and $|a_{2}| \leq 2\Sigma$, yields
\begin{displaymath} |a_{2} - a_{1}|  \leq \frac{4\epsilon(1 + \Sigma)}{|x_{2} - y_{2}|} \leq 4\epsilon(1 + \Sigma) \cdot (4 + 8\Sigma) \leq 64(1 + \Sigma)^{2}\epsilon. \end{displaymath}
Next, using \eqref{form137}, $|a_{1}|,|a_{2}| \leq 2\Sigma$, and $|x_{j}| \leq |\pi(p_{j})| \leq |\pi(p)| + \epsilon = \epsilon$ by $p = 0$, we have
\begin{equation}\label{form141} |b_{j}| \leq |a_{j}x_{j} + b_{j}| + |a_{j}x_{j}| \leq \epsilon + 2\Sigma \epsilon \leq 2(1 + \Sigma)\epsilon, \qquad j \in \{1,2\}, \end{equation} 
so in particular $|b_{2} - b_{1}| \leq 4(1 + \Sigma)\epsilon$. It remains to check that $|c_{2} - c_{1}| \leq \epsilon^{2}$. This is based on $\max\{\|p_{1}\|,\|p_{2}\|\} \leq \epsilon$. Since
\begin{displaymath} p_{j} = (a_{j}x_{j} + b_{j},x_{j},\tfrac{1}{2}b_{j}x_{j} + c_{j}), \qquad j \in \{1,2\}, \end{displaymath}
it follows that $|\tfrac{1}{2}b_{j}x_{j} + c_{j}| \lesssim \epsilon^{2}$ and
$|x_{j}| \leq \epsilon$ for $j \in \{1,2\}$. Combining this with \eqref{form141}, we find $|c_{j}|
\lesssim \epsilon^{2} + |b_{j}x_{j}| \lesssim
(1+\Sigma)\epsilon^{2}$, and finally $|c_{1} - c_{2}| \lesssim
(1+\Sigma) \epsilon^{2}$, as claimed. \end{proof}

The next lemma states that two (nearly) equidistant horizontal line segments are roughly contained in a small neighbourhood of a single vertical plane. For $E \subset \He$ and $\delta > 0$, we use the notation $N(E,\delta) := \{p \in \He : \dist(p,E) \leq \delta\}$ for the closed $\delta$-neighbourhood of $E$.

\begin{lemma}\label{lemma1} There exists an absolute constant $A \geq 10$ such that the following holds. Let $0 < \epsilon < 1$, $H \geq 1$, let $L_{1},L_{2} \subset \He$ be horizontal lines, let $p \in L_{1}$, and let $r > 0$. Assume that
\begin{equation}\label{form1} B(p,AH^{3}\epsilon^{-2}r) \cap L_{2} \subset N(L_{1},Hr). \end{equation}
Then, there exists a horizontal line $L \subset \V(L_{1})$ (the vertical plane spanned by $L_{1}$) with the property
\begin{equation}\label{form6} B(p,Hr) \cap L_{2} \subset N(L,\epsilon r). \end{equation}
\end{lemma}

\begin{proof} We first reduce to the case $H = 1$. Indeed, assume that this case is already known. Then, let $H \geq 1$ be arbitrary, and write $r' := Hr$ and $\epsilon' := \epsilon/H$. With this notation, the hypothesis \eqref{form1} can be rephrased as $\dist(q,L_{1}) \leq Hr = r'$ for all $q \in B(p,AH^{3}\epsilon^{-2}r) \cap L_{2} = B(p,A(\epsilon')^{-2}r')$. Consequently, the version of the lemma with "$H = 1$" implies the existence of $L \subset \V(L_{1})$ with the property $\dist(q,L) \leq \epsilon'r' = \epsilon r$ for all $q \in B(p,r') \cap L_{2} = B(p,Hr) \cap L_{2}$, as desired. So, we may assume $H = 1$. Translating, dilating, and rotating, we may also assume with no loss of generality that $B(p,r) = B(0,1)$, $L_{1} = \{(x,0,0) : x \in \R\}$, and $B(0,1) \cap L_{2} \neq \emptyset$ (otherwise there is nothing to prove).

Write $\hat{L}_{j} := \pi(L_{j})$ for $j \in \{1,2\}$, where $\pi(x,y,t) := (x,y)$ is the $1$-Lipschitz coordinate projection to $\R^{2}$. Write $R := A\epsilon^{-2} \geq 10$, where the value of the constant $A \geq 10$ will become apparent during the proof. Then, it follows from \eqref{form1} that $\hat{L}_{2} \cap B_{\R^{2}}(0,R)$ is contained in the $1$-neighbourhood of $\hat{L}_{1} = \{(x,0) : x \in \R\}$. Since $R \geq 10$, and $\hat{L}_{2} \cap B_{\R^{2}}(0,1) \neq \emptyset$, it follows that $\hat{L}_{2}$ can be written as a graph over the $x$-axis $\hat{L}_{1}$:
\begin{displaymath} \hat{L}_{2} = \{(x,ax + b) : x \in \R\}. \end{displaymath}
Moreover, $|ax + b| \leq 1$ for all $|x| \leq R$. In particular, $|b| \leq 1$, and $|a| \leq 2/R$. Since $L_{2}$ is a horizontal line with projection $\hat{L}_{2}$, there exists $c \in \R$ such that
\begin{equation}\label{form4} L_{2} = \{(x,ax + b,c - \tfrac{1}{2}bx) : x \in \R\}. \end{equation}
We will eventually show that \eqref{form6} holds with $L := \{(x,0,c) : x \in \R\}$. For this purpose, we first claim that $|c| \lesssim 1$ and $|b| \lesssim 1/R$, improving on the previous "trivial" bound $|b| \leq 1$. Since $L_{2} \cap B(0,1) \neq \emptyset$, there exists a point $(x,ax + b,c - \tfrac{1}{2}bx) \in B(0,1)$, whence $|x| \leq 1$, and $|c| \lesssim 1 + |bx| \lesssim 1$.

From the estimates $|a| \lesssim R^{-1}$, $|b| \leq 1$ and $|c| \lesssim 1$, we can now deduce that, for a suitable small constant $\theta > 0$, we have
\begin{displaymath} |x| \leq \theta R \quad \Longrightarrow \quad (x,ax + b,c - \tfrac{1}{2}bx) \in L_{2} \cap B(0,R). \end{displaymath}
Motivated by this, we define the two points
\begin{displaymath} q_{-} := (-\theta R,-\theta aR + b, c + \tfrac{\theta}{2}bR) \quad \text{and} \quad q_{+} := (\theta R,\theta aR + b, c - \tfrac{\theta}{2}bR), \end{displaymath}
both lying in $L_{2} \cap B(0,R)$. By the assumption \eqref{form1} that $\dist(q_{\pm},L_{1}) \leq 1$, we now find $\epsilon_{-},\epsilon_{+} \in \R$ such that $d((\pm \theta R + \epsilon_{\pm},0,0),q_{\pm}) \leq 1$. Expanding what this means,
\begin{equation}\label{form2} 1 \geq d((\theta R + \epsilon_{+},0,0),q_{+}) \gtrsim |\epsilon_{+}| + \sqrt{|c -  \theta b R - \tfrac{1}{2}\theta^{2}aR^{2} - \tfrac{1}{2}\theta a R\epsilon_{+} - \tfrac{1}{2}\epsilon_{+}b|} \end{equation}
and
\begin{equation}\label{form3} 1 \geq d((-\theta R + \epsilon_{-},0,0),q_{-}) \gtrsim |\epsilon_{-}| + \sqrt{|c +  \theta b R - \tfrac{1}{2}\theta^{2}aR^{2} + \tfrac{1}{2}\theta a R\epsilon_{-} - \tfrac{1}{2}\epsilon_{-}b|}. \end{equation}
It follows immediately that $|\epsilon_{\pm}| \leq 1$. Consequently (using also the previous bounds $|a| \lesssim R^{-1}$ and $|b| \leq 1$), we find that
\begin{displaymath} |\tfrac{1}{2}\theta a R\epsilon_{+} \pm \tfrac{1}{2}\epsilon_{\pm}b| \lesssim 1. \end{displaymath}
Combining this information with the bounds from \eqref{form2}-\eqref{form3}, we deduce that also
\begin{displaymath} 2\theta |b R| \leq |\theta bR + c - \tfrac{1}{2}\theta^{2}aR^{2}| + |\theta bR - c + \tfrac{1}{2}\theta^{2}aR^{2}| \lesssim 1, \end{displaymath}
and hence finally $|b| \lesssim R^{-1}$, as claimed.

Now, we complete the proof of \eqref{form6} for $L = \{(x,0,c) : x \in \R\}$. Note that
\begin{displaymath} d((x,0,c),(x,ax + b,c - \tfrac{1}{2}bx)) \sim |ax + b| + \sqrt{|bx + \tfrac{1}{2}ax^{2}|}, \qquad x \in \R. \end{displaymath}
For $|x| \lesssim 1$, the bound $\max\{|a|,|b|\} \lesssim 1/R$ shows that $d((x,0,c),(x,ax + b,c - \tfrac{1}{2}bx)) \lesssim 1/\sqrt{R}$. In particular, if the constant "$A$" in $R = A\epsilon^{-2}$ was chosen large enough, and recalling the parametrisation \eqref{form4} of the line $L_{2}$, we see that $d(q,L) \leq \epsilon$ for all $q \in L_{2} \cap B(0,1)$, as claimed. \end{proof}

%%%%%%%%%%%%%%%%%%%%%%%%%%%%%%%%%%%%%%%%%%%

\subsection{Quasi-isometries on $\R$ and $\W$}\label{s:QIE} Quasi-isometric embeddings are maps between metric spaces which behave like bilipschitz maps down to a fixed (typically small) resolution "$\epsilon$". These maps arise in the present paper, roughly speaking, as compositions $\pi_{\V} \circ f$, where $f \colon B(w,r) \to \He$ is a bilipschitz map defined a parabolic ball $B(w,r) \subset \W$,  and $\pi_{\V} \colon \He \to \V$ is the "closest point projection" to some vertical plane $\V = \V_{x,r} \subset \He$, which approximates $f(B(w,r))$ up to error "$\epsilon r$". 

\begin{definition}[$(M,\epsilon)$-quasi-isometric embedding]\label{QIE} A map $f \colon (X,d) \to (Y,d')$ between two metric spaces is an \emph{$(M,\epsilon)$-quasi-isometric embedding} (abbreviated $(M,\epsilon)$-QIE) if
\begin{displaymath} M^{-1}d(x,y) - \epsilon \leq d'(f(x),f(y)) \leq Md(x,y) + \epsilon, \qquad x,y \in X. \end{displaymath}
\end{definition}
\begin{remark}\label{rem1} Note that if $f \colon (X,d) \to (Y,d')$ is an $(M,\epsilon)$-QIE, and $x,x' \in X$ are two points with $d(x,x') \geq 2M\epsilon$, then $d'(f(x),f(x')) \geq \epsilon$. In particular, $f(x) \neq f(x')$.
\end{remark}
\begin{definition}[Orderly metric] A metric $d$ on $\R$ is \emph{orderly} if $\mathrm{Id} \colon (\R,|\cdot|) \to (\R,d)$ is continuous, and $\max\{d(x,y),d(y,z)\} \leq d(x,z)$ for all triples $x,y,z \in \R$ with $x \leq y \leq z$. \end{definition}
We will mainly use the definition in the form of the following corollary: if $[a,b] \subset [c,d]$, then $d(a,b) \leq d(c,d)$. This follows by chaining $d(a,b) \leq d(c,b) \leq d(c,d)$. The only metrics on $\R$ we will actually encounter are the Euclidean one, and the square root metric $\|x - y\| := \sqrt{|x - y|}$. Both are evidently orderly.
\begin{definition} Let $E \subset \R$, let $d$ be a metric on $\R$, let $f \colon E \to \R$ be a map, and let $\delta > 0$. We say that $f$ is \emph{$\delta$-monotone relative to $d$} if either $f(a) < f(b)$ for all $a,b \in E$ with $a < b$ and $d(a,b) > \delta$, or then $f(a) > f(b)$ for all $a,b \in E$ with $a < b$ and $d(a,b) > \delta$.
\end{definition}

The definition can be posed for arbitrary metrics, but most likely is not very useful if $d$ is not orderly. We start with a simple lemma on $(M,\epsilon)$-QIEs on the real line.

\begin{lemma}\label{lemma2} Let $I \subset \R$ be an interval (either open or closed), and let $f \colon (I,d) \to (\R,d)$ be an $(M,\epsilon)$-QIE, where $d$ is an orderly metric. Then $f$ is $2\epsilon M$-monotone relative to $d$, and $f([a,b])$ is $\epsilon$-dense in $[f(a),f(b)]$ for all $a,b \in I$.
\end{lemma}

The last statement means that $\dist(y,f([a,b])) \leq \epsilon$ for all $y \in [f(a),f(b)]$.

\begin{proof}[Proof of Lemma \ref{lemma2}] We start by proving the second assertion, which is then used to establish the first. Fix $a,b \in I$ with $a < b$, and assume without loss of generality that $f(a) < f(b)$ (if $f(a) = f(b)$, there is nothing to prove). Pick $f(a) < y < f(b)$, and fix $\epsilon' > \epsilon$. Consider the points
\begin{displaymath} c := \sup \{x \in [a,b] : f(x) \leq y\} \quad \text{and} \quad d := \inf \{x \in [c,b] : f(x) \geq y\}. \end{displaymath}
Clearly $c \leq d$ by definition, and in fact $c = d$; namely, if $c < d$, any point $x \in (c,d)$ would satisfy $f(x) < y$ by the definition of $d$, and then $c \geq x > c$ by the definition of $c$.

Next, by definitions of $c$ and $d$, there are points $x_{1} \leq
c = d \leq x_{2}$ with $|x_{1} - x_{2}|$ arbitrarily small such
that $f(x_{1}) \leq y \leq f(x_{2})$. Since $\mathrm{Id} \colon
(\R,|\cdot|) \to (\R,d)$ is continuous, such points can also be
chosen with $d(x_{1},x_{2}) < (\epsilon' - \epsilon)/M$. Then,
using the assumption that $d$ is orderly, we find that
$d(f(x_{1}),y) \leq d(f(x_{1}),f(x_{2})) \leq Md(x_{1},x_{2}) +
\epsilon < \epsilon'$, as desired.

We then proceed to the first statement of the lemma. Let $a,b,c,d \in I$ with $a < b$, $c < d$, and $\min\{d(a,b),d(c,d)\} > 2\epsilon M$. Therefore $\min\{d(a,b)/M - \epsilon,d(c,d)/M - \epsilon\} > \epsilon'$ for some $\epsilon' > \epsilon$. By Remark \ref{rem1}, we have $f(a) \neq f(b)$ and $f(c) \neq f(d)$, and we may assume that $f(a) < f(b)$. Our task is to prove that also $f(c) < f(d)$. The argument splits into two cases: either $[c,d] \subset [a,b]$ or $[c,d] \not\subset [a,b]$.

Assume first that $[c,d] \subset [a,b]$. We first check that $f(d) > f(a)$. If this were not the case, then $f(a) \in [f(d),f(b)]$. Since $f([d,b])$ is $\epsilon$-dense in $[f(d),f(b)]$, this implies that there exists a point $x \in [d,b]$ with $d(f(x),f(a)) \leq \epsilon'$. However, $d(a,x) \geq d(c,d)$, hence
\begin{equation}\label{form17} d(f(x),f(a)) \geq \frac{d(a,x)}{M} - \epsilon \geq \frac{d(c,d)}{M} - \epsilon > \epsilon', \end{equation}
a contradiction. So, we already know that $f(d) > f(a)$. Assume, then, to reach another contradiction, that $f(d) < f(c)$. Consequently $f(d) \in [f(a),f(c)]$, and by the reasoning above, there exists a point $x \in [a,c]$ with $d(f(x),f(d)) \leq \epsilon'$. However, $d(x,d) \geq d(c,d)$, which leads to $d(f(x),f(d)) > \epsilon'$, a contradiction. This settles the case $[c,d] \subset [a,b]$.

Consider then the case $[c,d] \not\subset [a,b]$. Then, either $c < a$, or $d > b$. These sub-cases are similar, and we only treat $c < a$. We first claim that, in this case, $f(c) < f(b)$. If this failed, then $f(b) \in [f(a),f(c)]$, and we may find a point $x \in [c,a]$ with $d(f(x),f(b)) \leq \epsilon'$. However, $d(x,b) \geq d(a,b)$, so $d(f(x),f(b)) > \epsilon'$ by the computation in \eqref{form17}, a contradiction. Finally, we check that $f(d) > f(c)$ by a fourth iteration of the same argument: if this fails, then $f(c) \in [f(d),f(b)]$. Consequently, there exists a point $x \in [d,b]$ or $x \in [b,d]$ such that $d(f(x),f(c)) \leq \epsilon'$. If here $d \leq b$, then $d(c,x) \geq d(c,d)$, leading to the familiar contradiction. And if $b \leq d$, then $d(c,x) \geq d(a,b)$, leading once more to the same contradiction. The proof of the lemma is complete. \end{proof}

\begin{cor}\label{cor1} Let $f \colon I := [a,c] \to \R$ be an $(M,\epsilon)$-QIE, and assume that $b \in (a + 2M\epsilon,c - 2M\epsilon)$. Then $f(b) \in (f(a),f(c))$ or $f(b) \in (f(c),f(a))$.
\end{cor}

\begin{proof} By Remark \ref{rem1}, we have $f(a) \neq f(b) \neq f(c)$. If $f(a) < f(b)$, then Lemma \ref{lemma2} implies that also $f(b) < f(c)$, and hence $f(b) \in (f(a),f(c))$. Otherwise, $f(b) \in (f(c),f(a))$ by the same argument. \end{proof}

We next prove a similar statement concerning \emph{horizontal
$(M,\epsilon)$-QIEs} defined on rectangles in the parabolic plane
$\W$; keep in mind that parabolic balls, in particular, are rectangles, since $d_{\mathrm{par}}$ was defined as the $L^{\infty}$-metric.
\begin{definition}[Horizontal $(M,\epsilon)$-QIE]\label{hQIE} Let $Q := I \times J \subset \W$ be a rectangle. A map $F \colon Q \to \W$ is called a \emph{horizontal $(M,\epsilon)$-QIE} if $F$ is an $(M,\epsilon)$-QIE in the sense of Definition \ref{QIE}, and moreover, for every $t \in J$ there exists a horizontal line $\ell = \ell_{t} \subset \W$ such that $F(I \times \{t\}) \subset \ell$.  \end{definition}

Every horizontal $(M,\epsilon)$-QIE $F \colon I \times J \to \R$ \emph{induces} the following map $f \colon J \to \R$:
\begin{displaymath} f(t) := \pi_{2}(F(y,t)), \end{displaymath}
where $\pi_{2}$ is the projection to the second coordinate, and $y \in J$ is arbitrary (the value of $y \mapsto \pi_{2}(F(y,t))$ does not depend on $y \in J$ by the horizontality assumption). The projection $\pi_{2}$ is a $1$-Lipschitz map $(\W,d_{\mathrm{par}}) \to (\R,\|\cdot\|)$, where $\|x - y\| := \sqrt{|x - y|}$ is the square root metric. From this, it follows quite easily that $f$ is also a QIE $(I,\|\cdot\|) \to (\R,\|\cdot\|)$ if the interval $I \subset \R$ is long enough relative to $\sqrt{|J|}$. Details follow:

\begin{lemma}\label{lemma3} Let $\epsilon > 0$, $M \geq 1$, and let $I,J \subset \R$ be intervals with
\begin{displaymath} |I| > 2M^{2}\sqrt{|J|} + 4M\epsilon. \end{displaymath}
If $F \colon Q := I \times J \to \W$ is a horizontal $(M,\epsilon)$-QIE, then the induced map $f \colon (J,\|\cdot\|) \to (\R,\|\cdot\|)$ is an $(M,2\epsilon)$-QIE.
\end{lemma}

\begin{proof} For simplicity of notation, assume that $I = [-a,a]$ for some $a > M^{2}\sqrt{|J|} + 2M\epsilon$. In particular, $0 \in I$. Then, fix $t_{1},t_{2} \in J$, and note that $f(t_{j}) = \pi_{2}(F(0,t_{j}))$ for $j \in \{1,2\}$. Using that $F$ is an $(M,\epsilon)$-QIE, we infer that
\begin{displaymath} \|f(t_{1}) - f(t_{2})\| \leq d_{\mathrm{par}}(F(0,t_{1})),F((0,t_{2})) \leq M\|t_{1} - t_{2}\| + \epsilon, \end{displaymath}
since the parabolic distance of $(0,t_{1}),(0,t_{2})$ is simply $\|t_{1} - t_{2}\|$. Proving the lower bound is a little trickier. The plan is the following. Let $\pi_{1}(y,t) := y$. Consider the auxiliary point
\begin{displaymath} v := (\pi_{1}(F(0,t_{2})),\pi_{2}(F(0,t_{1}))) \in \W, \end{displaymath}
which is vertically aligned with $F(0,t_{2})$, but has the same $\pi_{2}$-projection as $F(0,t_{1})$. Then,
\begin{displaymath} \|f(t_{1}) - f(t_{2})\| = \|\pi_{2}(F(0,t_{1})) - \pi_{2}(F(0,t_{2})\| = \|\pi_{2}(v) - \pi_{2}(F(0,t_{2}))\| = d_{\mathrm{par}}(v,F(0,t_{2})). \end{displaymath}
To complete the proof from here, we will for any $\epsilon' > \epsilon$ find a point $y_{0} \in I$ such that
\begin{equation}\label{form16} d_{\mathrm{par}}(F(y_{0},t_{1}),v) < \epsilon'. \end{equation}
With such a point $y_{0} \in I$, in hand, we deduce that
\begin{align*} \|f(t_{1}) - f(t_{2})\| & \geq d_{\mathrm{par}}(F(y_{0},t_{1}),F(0,t_{2})) - \epsilon'\\
& \geq \frac{d_{\mathrm{par}}((y_{0},t_{1}),(0,t_{2}))}{M} - 2\epsilon' \geq \frac{\|t_{1} - t_{2}\|}{M} - 2\epsilon', \end{align*}
and the proof is finished by letting $\epsilon' \searrow \epsilon$.

To find the special point $y_{0} \in I$, consider the map $F_{j} \colon I \to \R$ defined by $F_{j}(y) := \pi_{1}(F(y,t_{j}))$ for $j \in \{1,2\}$. Since $F$ maps $I \times \{t_{j}\}$ inside a fixed horizontal line, and $\pi_{1}$ restricted to this line is an isometry, we see that $F_{1}$ is an $(M,\epsilon)$-QIE $I \to \R$ (with Euclidean metrics on $I$ and $\R$). Moreover, since $-a < -2M\epsilon$ and $a > 2M\epsilon$, we infer from Corollary \ref{cor1} that $F_{1}(0) \in [F_{1}(-a),F_{1}(a)]$ or $F_{1}(0) \in [F_{1}(a),F_{1}(-a)]$. Assume with no loss of generality that $F_{1}(-a) < F_{1}(0) < F_{1}(a)$. Then
\begin{equation}\label{form14} F_{1}(a) - F_{1}(0) > \frac{a}{M} - \epsilon > M\sqrt{|J|} + \epsilon \quad \text{and} \quad F_{1}(0) - F_{1}(-a) > Mr + \epsilon. \end{equation}
On the other hand, since $t_{1},t_{2} \in J$, we note that
\begin{equation}\label{form15} |F_{1}(0) - F_{2}(0)| \leq d_{\mathrm{par}}(F(0,t_{1}),F(0,t_{2})) \leq M\sqrt{|J|} + \epsilon, \end{equation}
so we infer from a combination of \eqref{form14}-\eqref{form15} that $F_{2}(0) \in [F_{1}(a_{1}),F_{1}(c_{1})]$. Finally, by Lemma \ref{lemma2}, the image $F_{1}(I)$ is $\epsilon$-dense in $[F_{1}(-a),F_{1}(a)]$, so there exists a point $y_{0} \in I$ such that
\begin{displaymath} |\pi_{1}(F(y_{0},t_{1})) - \pi_{1}(F(0,t_{2}))| = |F_{1}(y) - F_{2}(0)| \leq \epsilon'. \end{displaymath}
Since $\pi_{2}(F(y_{0},t_{1})) = \pi_{2}(F(0,t_{1})) = \pi_{2}(v)$, we find that
\begin{align*} d_{\mathrm{par}}(F(y_{0},t_{1}),v) & \leq |\pi_{1}(F(y,t_{1})) - \pi_{1}(v)| = |\pi_{1}(F(y_{0},t_{1})) - \pi_{1}(F(0,t_{2}))| \leq \epsilon',  \end{align*}
as desired in \eqref{form16}. \end{proof}

As the first corollary of the previous lemma, we infer that both $\pi_{j}$-projections, $j \in \{1,2\}$, of horizontal QIEs are QIEs on the real line (relative to two different metrics):
\begin{lemma}\label{lemma6} Let $F \colon Q = I \times J \to \W$ be a horizontal $(M,\epsilon)$-QIE, where $|I| > 2M^{2}\sqrt{|J|} + 4M\epsilon$. Then, for $(y_{0},t_{0}) \in Q$ fixed, the maps
\begin{displaymath} y \mapsto \pi_{1}(F(y,t_{0})) \quad \text{and} \quad t \mapsto \pi_{2}(F(y_{0},t)) \end{displaymath}
are $(M,2\epsilon)$-QIEs on $(I,|\cdot|) \to (\R,|\cdot|)$ and $(J,\|\cdot\|) \to (\R,\|\cdot\|)$, respectively. \end{lemma}

\begin{proof} The claim about $y \mapsto \pi_{1}(F(y,t_{0}))$ follows immediately from the horizontality of $F$, and the fact that $\pi_{1}$ restricted to any horizontal line is an isometry. The claim about $t \mapsto \pi_{2}(F(y_{0},t))$ is a restatement of Lemma \ref{lemma3}. \end{proof}

As a second corollary or Lemma \ref{lemma3}, we infer a \emph{vertical monotonicity} property of horizontal QIEs $F \colon Q \to \W$. Let us first define this property:
\begin{definition} Let $Q = I \times J \to \W$ be a rectangle. A map $F \colon Q \to \W$ is called \emph{vertically $\delta$-monotone} if $t \mapsto \pi_{2}(F(y,t))$ is $\delta$-monotone on $J$ relative to the square root metric for every $y \in I$, and the signature of monotonicity is independent of $y \in I$.  \end{definition} 

The final condition means that either $t \mapsto \pi_{2}(F(y,t))$ is $\delta$-increasing on $J$ for every $y \in I$, or then $t \mapsto \pi_{2}(F(y,t))$ is $\delta$-decreasing on $J$ for every $y \in I$. 

\begin{cor}\label{cor3} Let $F \colon Q := I \times J \to \W$ be a horizontal $(M,\epsilon)$-QIE, where $|I| > 2M^{2}\sqrt{|J|} + 4M\epsilon$. Then $F$ is vertically $2M\epsilon$-monotone.
\end{cor}

\begin{proof} By Lemma \ref{lemma3}, we know that the induced map $f \colon (J,\|\cdot\|) \to (\R,\|\cdot\|)$ is an $(M,2\epsilon)$-QIE. Since the square root metric is orderly, Lemma \ref{lemma2} applies to $f$ and shows that $f$ is $4\epsilon M$-monotone. Assume, for example, that $f$ is $4M\epsilon$-increasing, that is,
\begin{displaymath} t_{1},t_{2} \in J \text{ and } \|t_{1} - t_{2}\| > 4M\epsilon \quad \Longrightarrow \quad f(t_{2}) > f(t_{1}). \end{displaymath}
By the horizontality assumption on $F$, this implies that
\begin{displaymath} (y_{1},t_{1}),(y_{2},t_{2}) \in Q \text{ and } \|t_{1} - t_{2}\| > 4M\epsilon \quad \Longrightarrow \quad \pi_{2}(F(y_{2},t_{2})) > \pi_{2}(F(y_{1},t_{1})). \end{displaymath}
In particular, $t \mapsto \pi_{2}(F(y,t))$ is $4\epsilon M$-increasing for all $y \in I$. \end{proof}

For future reference, we formalise the notion of \emph{signature of monotonicity}: 

\begin{definition}[Signature]\label{signature} Let $F \colon Q = I \times J \to \W$ be a vertically $\delta$-monotone map. The \emph{signature} of $F$ is "$+$" if the maps $t \mapsto \pi_{2}(F(y,t))$ are $\delta$-increasing on $J$ for all $y \in I$. Otherwise the signature of $F$ is "$-$". 
\end{definition}

We close this section with another simple property of QIEs on the parabolic plane: 

\begin{lemma}\label{lemma7} Let $Q = I \times J \subset \W$ be a rectangle, and let $F \colon Q \to \W$ be a horizontal $(M,\epsilon)$-QIE. Assume that $w_{1} = (y_{1},t_{1}), w_{2} = (y_{2},t_{2}) \in Q$ with
\begin{equation}\label{eq:y_2-y_1_ass} |y_{2} - y_{1}| > \max\{8M\epsilon,4M^{2}\|t_{1} - t_{2}\|\}. \end{equation}
Then,
\begin{displaymath} |\pi_{1}(F(w_{2})) - \pi_{1}(F(w_{1}))| \geq \frac{|y_{2} - y_{1}|}{2M}. \end{displaymath}
\end{lemma}

\begin{proof} Consider the point $w_{2}' := (y_{2},t_{1}) \in Q$. Then,
\begin{displaymath} |\pi_{1}(F(w_{2}')) - \pi_{1}(F(w_{2}))| \leq d_{\mathrm{par}}(F(w_{2}'),F(w_{2})) \leq M\|t_{1} - t_{2}\| + \epsilon. \end{displaymath}
Further, using the $(M,\epsilon)$-QIE property of $\pi_{1} \circ F$ on $\{(y,t_{1}) : y \in \R\}$, we have
\begin{displaymath} |\pi_{1}(F(w_{2}')) - \pi_{1}(F(w_1))| \geq \frac{|y_{2} - y_{1}|}{M} - \epsilon. \end{displaymath}
Therefore, by the triangle inequality,
\begin{displaymath} |\pi_{1}(F(w_{2})) - \pi_{1}(F(w_{1}))| \geq \frac{|y_{2} - y_{1}|}{M} - M\|t_{1} - t_{2}\| - 2\epsilon \geq \frac{|y_{2} - y_{1}|}{2M}, \end{displaymath} 
using the main hypothesis \eqref{eq:y_2-y_1_ass} (or rather its corollary $|y_{2} - y_{1}| > 4M\epsilon + 2M^{2}\|t_{1} - t_{2}\|$) in the final inequality.  \end{proof}

%%%%%%%%%%%%%%%%%%%%%%%%%%%%%%

\section{Ruler coefficients and $\beta$-numbers}\label{s:rulerCoefficients}

The aim of this section is to show that Rickman rugs $f(\W) \subset \He$ are $\epsilon$-approximable by vertical planes at "almost all scales and locations". The formal statement is Corollary \ref{wgl}. The idea is that $f$ restricted to every horizontal line defines a $1$-regular curve in $\He$, and such curves are well-known to admit the correct type of $\epsilon$-approximations by horizontal lines. It follows that for "most" $w \in \W$ and $r > 0$, the image $f(B(w,r))$ is a surface in $\He$ which is approximately ruled by horizontal lines. Relying on the assumption that $f$ is bilipschitz, and using Lemma \ref{lemma1}, this will imply that $f(B(w,r/C))$ lies $\epsilon$-close to a vertical plane for a suitable constant $C = C(\epsilon,M) \geq 1$, see Proposition \ref{prop1}.

\begin{definition} Let $w \in \W$, $r > 0$, and let $f \colon B(w,r) \to \He$ be a map. Define $\rho_{f}(B(w,r))$ to be the infimum of numbers $\rho > 0$ with the following property: for every horizontal line $\ell \subset \W$, there exists a horizontal line $L \subset \He$ such that
\begin{equation}\label{form128} f(\ell \cap B(w,r)) \subset N(L,\rho r). \end{equation}
Here $N(E,\delta) := \{p \in \He : \dist(p,E) \leq \delta\}$ refers to the closed $\delta$-neighbourhood of $E$.
\end{definition}

\begin{remark} Of course, the requirement \eqref{form128} may be restricted to those lines $\ell \subset \W$ with $\ell \cap B(w,r) \neq \emptyset$, and in particular $\pi_{2}(\ell) \cap \pi_{2}(B(w,r)) \neq \emptyset$. \end{remark}

The discs $B(w,r)$, $x \in \W$, $r > 0$, with $\rho_{f}(B(w,r)) \geq \epsilon$ satisfy a \emph{Carleson packing condition}, for any $\epsilon > 0$ fixed. We only formulate a dyadic (and more easily applicable) variant of this result. To this end, recall that if $Q \in \mathcal{D}$ and $C > 0$, then $CQ := B(c_{Q},C\ell(Q))$, where $c_{Q} \in Q$ is the centre of $Q$, and $\ell(Q) = \diam(Q)$ is the side-length of $Q$.
\begin{proposition}\label{propCarleson} Let $f \colon \W \to \He$ be an $M$-bilipschitz map, $M \geq 1$. For any $\epsilon \in (0,1]$ and $C \geq 1$, there exists a constant $C' = C'(C,\epsilon,M)$ such that the following holds:
\begin{equation}\label{form125} \mathop{\sum_{Q \in \mathcal{B}(C,\epsilon)}}_{Q \subset Q_{0}} \ell(Q)^{3} \leq C'\ell(Q_{0})^{3}, \qquad Q_{0} \in \mathcal{D}, \end{equation}
where $\mathcal{B}(C,\epsilon) = \{Q \in \mathcal{D} : \rho_{f}(CQ) \geq \epsilon\}$.
\end{proposition}
\begin{proof} The map $f$ restricted to an arbitrary horizontal line $\ell_{t} = \{(y,t) : y \in \R\} \subset \W$ defines a $1$-regular curve $\gamma_{t} \colon \R \to \He$ by $\gamma_{t}(y) := f(y,t)$. The regularity constant of $\gamma_{t}$ only depends on $M$. To apply this fact, we record that every $1$-regular curve $\gamma \colon \R \to \He$ satisfies a Carleson packing condition known as the (strong) \emph{geometric lemma}:
\begin{equation}\label{form124} \mathop{\sum_{I \in \mathcal{D}_{\R}}}_{I \subset I_{0}} \beta_{\gamma}(CI)^{4}|I| \leq C'|I_{0}|, \qquad I_{0} \in \mathcal{D}_{\R}, \end{equation}
where $\mathcal{D}_{\R}$ is the standard dyadic grid on $\R$, and $\beta_{\gamma}(CI)$ is the \emph{horizontal $\beta$-number}
\begin{displaymath} \beta_{\gamma}(J) := \inf\{\epsilon > 0 : \gamma(J) \subset N(L,\epsilon |J|) \text{ for some horizontal line } L \subset \He\}. \end{displaymath}
The geometric lemma \eqref{form124} for $1$-regular curves in $\He$ is literally contained in \cite[Proposition 3.1]{MR3678492}, but as the authors of \cite{MR3678492} also say, it is a consequence of the main theorem in \cite{MR3512421}. The constant $C' \geq 1$ in \eqref{form124} only depends on $C$ and the $1$-regularity constant of $\gamma$.

To pass from the condition \eqref{form124} concerning individual curves $\gamma_{t} = f(\cdot,t)$ to the Carleson packing condition \eqref{form124} concerning the ruler coefficients, let us abbreviate
\begin{displaymath} \beta_{t}(CQ) := \beta_{\gamma_{t}}(\pi_{1}(CQ)), \qquad Q \in \mathcal{D}, \, t \in \pi_{2}(CQ), \end{displaymath}
and also define $\beta_{t}(CQ) = 0$ if $t \notin \pi_{2}(CQ)$. Then, we define the $L^{4}$-averaged coefficient
\begin{displaymath} \alpha_{f}(CQ) := \left( \frac{1}{|\pi_{2}(CQ)|} \int_{\pi_{2}(CQ)} \beta_{t}(CQ)^{4} \, dt \right)^{1/4}, \qquad Q \in \mathcal{D}. \end{displaymath} 
It follows rather immediately from \eqref{form124}, and the relation $|\pi_{2}(CQ)| \sim_{C} \ell(Q)^{2}$, that the $4^{th}$ powers of the coefficients $\alpha_{f}(CQ)$ satisfy a Carleson packing condition:
\begin{align} \mathop{\sum_{Q \in \mathcal{D}}}_{Q \subset Q_{0}} \alpha_{f}(CQ)^{4}\ell(Q)^{3} & \sim_{C} \mathop{\sum_{Q \in \mathcal{D}}}_{Q \subset Q_{0}} \ell(Q) \int_{\pi_{2}(CQ)} \beta_{t}(CQ)^{4} \, dt \notag\\
& = \int_{\pi_{2}(CQ_{0})} \Bigg( \mathop{\sum_{Q \subset Q_{0}}}_{t \in \pi_{2}(CQ)} \beta_{\gamma_{t}}(\pi_{1}(CQ))^{4}\ell(Q) \Bigg) dt \notag\\
& \lesssim_{C} \int_{\pi_{2}(CQ_{0})} \Bigg( \mathop{\sum_{I \in \mathcal{D}_{\R}}}_{I \subset \pi_{1}(Q_{0})} \beta_{\gamma_{t}}(CI)^{4}|I| \Bigg) \, dt \notag\\
&\label{form126} \stackrel{\eqref{form124}}{\lesssim_{C,M}} \int_{\pi_{2}(CQ_{0})} \ell(Q_{0}) \, dt \sim \ell(Q_{0})^{3}, \qquad Q_{0} \in \mathcal{D}(Q). \end{align} 
In the only non-trivial inequality, we used that if $t \in \pi_{2}(CQ_{0})$ is fixed, then any interval $CI$ with $I \in \mathcal{D}_{\R}$, has $\lesssim_{C} 1$ representations as $CI = \pi_{1}(CQ)$ with $Q \in \mathcal{D}$ and $t \in \pi_{2}(CQ)$.

From the "geometric lemma for the $\alpha_{f}$-coefficients" established in \eqref{form126}, and Chebyshev's inequality, it follows for all $\epsilon > 0$ that
\begin{displaymath} \mathop{\sum_{Q \in \widetilde{\mathcal{B}}(C,\epsilon)}}_{Q \subset Q_{0}} \ell(Q)^{3} \lesssim_{C,M} \epsilon^{-4} \cdot \ell(Q_{0})^{3}, \qquad Q_{0} \in \mathcal{D},  \end{displaymath} 
where $\widetilde{B}(C,\epsilon) := \{Q \in \mathcal{D} : \alpha_{f}(CQ) \geq \epsilon\}$. Thus, to complete the proof of \eqref{form125}, hence the proposition, it remains to prove an inclusion of the form
\begin{displaymath} \mathcal{B}(C,\epsilon) \subset \widetilde{\mathcal{B}}(2C,\tilde{\epsilon}), \qquad C \geq 1, \epsilon > 0, \end{displaymath} 
where $\tilde{\epsilon} > 0$ is only allowed to depend on $C,\epsilon,M$. In other words, we need to show that if $Q \in \mathcal{B}(C,\epsilon)$, hence $\rho_{f}(CQ) \geq \epsilon$, then $\alpha_{f}(CQ) \geq \tilde{\epsilon}$, hence $Q \in \widetilde{B}(2C,\tilde{\epsilon})$. In fact, it is more convenient to prove an equivalent statement: for every $\epsilon > 0$, there exists $\tilde{\epsilon} > 0$, depending only on $C,\epsilon,M$ such that
\begin{equation}\label{form127} \alpha_{f}(2CQ) < \tilde{\epsilon} \quad \Longrightarrow \quad \rho_{f}(CQ) < \epsilon. \end{equation}
In fact, we claim that \eqref{form127} holds with $\tilde{\epsilon} := (\epsilon/10) \cdot (\delta/C)^{1/2}$, where $\delta := \epsilon/(2M)$. Assume that $\alpha_{f}(2CQ) < \tilde{\epsilon}$ with this choice of $\tilde{\epsilon}$, and fix a horizontal line $\ell = \ell_{t_{0}} \subset \W$ with $t_{0} \in \pi_{2}(CQ)$. We claim that there exists 
\begin{displaymath} t \in [t_{0} - (\delta \ell(Q))^{2},t_{0} + (\delta \ell(Q))^{2}] \subset \pi_{2}(2CQ) \text{ such that } \beta_{t}(2CQ) < \tfrac{\epsilon}{4}. \end{displaymath}
Indeed, if this fails, then we note that $|\pi_{2}(2CQ)| = (2C\ell(Q))^{2}$, and hence
\begin{displaymath} \alpha_{f}(2CQ) \geq \left( \frac{1}{|\pi_{2}(2CQ)|} \int_{t_{0} - \delta^{2} \ell(Q)^{2}}^{t_{0} + \delta^{2} \ell(Q)^{2}} \left(\frac{\epsilon}{4} \right)^{4} \, dt \right)^{1/4} \geq \frac{\epsilon}{4} \cdot \left(\frac{\delta^{2}}{4C^{2}} \right)^{1/4} > \tilde{\epsilon}, \end{displaymath}
a contradiction.

To complete the proof, fix $t \in [t_{0} - (\delta \ell(Q))^{2},t_{0} + (\delta \ell(Q))^{2}]$ such that $\beta_{t}(CQ) < \tfrac{\epsilon}{4}$. This means that there exists a horizontal line $L \subset \He$ with the property
\begin{displaymath} f(\ell_{t} \cap CQ) = \gamma_{t}(\pi_{1}(2CQ)) \subset N(L,2C\tfrac{\epsilon}{4} \cdot \ell(Q)) = N(L,\tfrac{\epsilon}{2} \cdot C\ell(Q)). \end{displaymath}
where $\ell_{t} = \{(y,t) : y \in \R\}$. Since $\ell_{t_{0}} \subset N_{\mathrm{par}}(\ell_{t},\delta \ell(Q)) \subset N_{\mathrm{par}}(\ell_{t},\tfrac{\epsilon}{2M} \cdot C\ell(Q))$, we may finally infer from the $M$-bilipschitz property of $f$ that
\begin{displaymath} f(\ell_{t_{0}} \cap CQ) \subset N(L,\epsilon \cdot C\ell(Q)). \end{displaymath}
Since $\ell = \ell_{t_{0}}$ was an arbitrary horizontal line intersecting $CQ$, we have proven that $\rho_{f}(CQ) < \epsilon$, as claimed in \eqref{form127}. The proof of Proposition \ref{propCarleson} is complete. \end{proof}

Proposition \ref{propCarleson} shows that dyadic rectangles with non-negligible ruler coefficient are not too common. We next record a key geometric property of $f$-images of discs with small ruler coefficients: $\rho_{f}(B(w,Cr)) < \delta$ forces $\beta_{f}(B(w,r)) < \epsilon$, where $\beta_{f}(B(w,r))$ is the \emph{strong vertical $\beta$-number}:

\begin{definition}[Strong vertical $\beta$-number] Let $f \colon \W \to \He$ be a map, and let $w \in \W$, $r > 0$. For a vertical plane $\V \subset \He$, define $\beta_{f}(\V ; B(w,r))$ as the infimum of the numbers $\epsilon > 0$ with the following property: for all horizontal lines $\ell \subset \W$, there corresponds a horizontal line $L \subset \V$ such that $f(\ell \cap B(w,r)) \subset N(L,\epsilon r)$. The \emph{strong vertical $\beta$-number} is then defined as 
\begin{displaymath} \beta_{f}(B(w,r)) := \inf_{\V} \beta_{f}(\V ; B(w,r)), \end{displaymath}
where the infimum runs over all vertical planes. \end{definition}

\begin{remark} It is sometimes convenient (but not really necessary) to talk about the \emph{best approximating plane} for $f$ at $B(w,r)$. This will refer to a vertical plane $\V \subset \He$ such that the infimum in the definition of $\beta_{f}(B(w,r))$ is attained. It follows from an easy compactness argument that best approximating planes exist (in verifying this, recall that $N(L,\epsilon r) = \{p \in \He : \dist(p,L) \leq \epsilon r\}$ is a closed neighbourhood in this paper).
 
Second remark: if $\beta_{f}(B(w,r)) < \epsilon$, then there exists a vertical plane $\V \subset \He$ with $f(B(w,r)) \subset N(\V,\epsilon r)$. This means that the strong vertical $\beta$-number is at least as large as the commonly used \emph{vertical $\beta$-number}, see \cite[Definition 3.3]{CFO}. Hence the word "strong". \end{remark}

\begin{proposition}\label{prop1} Let $f \colon \W \to \He$ be an $M$-bilipschitz map, $M \geq 1$. Then, for every $\epsilon \in (0,1]$, there exist $C = C(\epsilon,M) \geq 10$ and $\delta = \delta(C,\epsilon,M) > 0$ such that the following holds:
\begin{displaymath} \rho_{f}(B(w,Cr)) < \delta \quad \Longrightarrow \quad \beta_{f}(B(w,r)) < \epsilon. \end{displaymath}
\end{proposition}

\begin{remark}\label{rem2} The proof below shows that $\beta_{f}(\V ; B(w,r)) < \epsilon$ for \textbf{any} vertical plane $\V \subset \He$ which arises as follows: take an arbitrary horizontal line $\ell \subset \W$ with $\ell \cap B(w,r) \neq \emptyset$. Let $L \subset \He$ be some horizontal line with
\begin{displaymath} f(\ell \cap B(w,Cr)) \subset N(L,C\delta r), \end{displaymath}
(which exists by $\rho_{f}(B(w,Cr)) < \delta$) and let $\V := \V(L)$ be the unique vertical plane containing $L$. In the proof below, we will, for the sake of concreteness, pick the line $\ell$ which passes through $w$, but any line intersecting $B(w,r)$ would work equally well, modulo changing $C$ and $\delta$ by absolute constant factors.
\end{remark}

\begin{proof}[Proof of Proposition \ref{prop1}] We choose
\begin{equation}\label{form13} C := C_{0}M^{4}\epsilon^{-2} \quad \text{and} \quad \delta := \epsilon/(C_{0}CM) \end{equation}
for some absolute constant $C_{0} \geq 1$, whose value will be determined during the proof.

Let $\ell_{1} \subset \W$ be the horizontal line containing $w$. By definition of $\rho_{f}(B(w,Cr)) < \delta$, there exists a horizontal line $L_{1} \subset \He$ such that
\begin{displaymath}  f(\ell_{1} \cap B(w,Cr)) \subset N(L_{1},C\delta r) \subset N(\V,C\delta r), \end{displaymath}
where we define $\V := \V(L_{1})$ to be the vertical plane spanned by $L_{1}$.

Fix an arbitrary horizontal line $\ell_{2} \subset \W$. The task is to find a horizontal line $L \subset \V$ such that $f(\ell_{2} \cap B(w,r)) \subset N(L,\epsilon r)$. In doing so, we may and will assume that $\ell_{2} \cap B(w,r) \neq \emptyset$. Let $p_{1} \in L_{1}$ be the closest point to $f(w)$. Then $d(p_{1},f(w)) \leq \epsilon r \leq Mr$. Consequently,
\begin{equation}\label{form7} f(\ell_{2} \cap B(w,r)) \subset B(f(w),Mr) \subset B(p_{1},2Mr). \end{equation}
Now, by definition of $\rho_{f}(B(w,Cr)) \leq \delta$, there exists a horizontal line $L_{2} \subset \He$, not necessarily contained on $\V$, such that
\begin{displaymath} f(\ell_{2} \cap B(w,Cr)) \subset N(L_{2},C\delta r) \stackrel{\eqref{form13}}{\subset} N(L_{2},\epsilon r/2). \end{displaymath}
Combining this assumption with \eqref{form7}, we observe that, in particular,
\begin{displaymath} f(\ell_{2} \cap B(w,r)) \subset B(p_{1},2Mr) \cap N(L_{2},\epsilon r/2). \end{displaymath}
So, the proof of the proposition will be complete once we manage to show that there exists a further horizontal line $L \subset \V$ such that
\begin{equation}\label{form8} B(p_{1},2Mr) \cap L_{2} \subset N(L,\epsilon r/2). \end{equation}
To prove \eqref{form8}, we apply Lemma \ref{lemma1} with the horizontal lines $L_{1},L_{2}$, with $p := p_{1} \in L_{1}$, and radius $r > 0$, with "$\epsilon/2$" in place of "$\epsilon$", and with constant $H := 2M$. In other words, we need to verify that
\begin{equation}\label{form9} B(p,AH^{3}(\epsilon/2)^{-2}r) \cap L_{2} \subset N(L_{1},Hr). \end{equation}
Once this has been accomplished, Lemma \ref{lemma1} yields a horizontal line $L \subset \V(L_{1})$ with the property
\begin{displaymath} B(p,2Mr) \cap L_{2} = B(p,Hr) \cap L_{2} \subset N(L,\epsilon r/2), \end{displaymath}
as required in \eqref{form8}.

It remains to verify \eqref{form9}. Abbreviate $R := AH^{3}(\epsilon/2)^{-2}r \geq r$. We first claim that it suffices to verify the inclusion
\begin{equation}\label{form10} B(p,R) \cap L_{2} \subset N(f(\ell_{2} \cap B(w,\tfrac{C}{2}r)),2C\delta r), \end{equation}
which informally says that $L_{2}$ is well-approximated by $f(\ell_{2})$. By the definition of $L_{2}$, we know that $f(\ell_{2})$ is well-approximated by $L_{2}$, and \eqref{form10} is the converse statement: it holds, as we will see, because $f|_{\ell_{2}}$ is bilipschitz.

To infer \eqref{form9} from \eqref{form10}, fix $q_{2} \in B(p,R) \cap L_{2}$, and using \eqref{form10}, pick
\begin{displaymath} v_{2} \in \ell_{2} \cap B(w,\tfrac{C}{2}r) \text{ such that } q_{2} \in B(f(v_{2}),2C\delta r). \end{displaymath}
Then, noting that $\dist_{\mathrm{par}}(\ell_{1},\ell_{2}) \leq r$ (since $\ell_{1}$ passes through $w$, and $\ell_{2}$ intersects $B(w,r)$, as we assumed), we may find a point $v_{1} \in \ell_{1} \cap B(w,Cr)$ with the property $d_{\mathrm{par}}(v_{1},v_{2}) \leq r$. In particular, $d(f(v_{1}),f(v_{2})) \leq M r$. In addition, since $v_{1} \in \ell_{1} \cap B(w,Cr)$, the definition of $\rho_{f}(B(w,Cr)) \leq \delta$ yields a point $q_{1} \in L_{1}$ with $d(f(v_{1}),q_{1}) \leq C\delta r$. Finally, by the triangle inequality,
\begin{align*} \dist(q_{2},L_{1}) & \leq d(q_{2},q_{1})\\
& \leq d(q_{1},f(v_{1})) + d(f(v_{1}),f(v_{2})) + d(f(v_{2}),q_{2})\\
& \leq Mr + 3C\delta r \leq 2Mr, \end{align*}
as claimed in \eqref{form10} (here we took $\delta = \delta(C,M) > 0$ sufficiently small).

Finally, we turn our attention to \eqref{form10}. We will use Lemma \ref{lemma2}. To do so, we identify both $(\ell_{2},d_{\mathrm{par}})$ and $(L_{2},d)$ with $\R$, and we first use the restriction of $f$ to $I := \ell_{2} \cap B(w,\tfrac{C}{2}r)$ to define a QIE $i \colon I \to \R$. Simply, for $v \in I$, let $i(v)$ be the closest point on $L_{2}$ to $f(v)$. Since $f$ is $M$-bilipschitz, and $d(f(v),i(v)) \leq C\delta r$ for all $v \in I$, we see that $i \colon I \to \R$ is an $(M,C\delta r)$-QIE.

Identify $J := B(p,R) \cap L_{2}$ to a sub-interval of $\R$ (of length $\leq 2R$). Then, pick an arbitrary point $w_{2} \in \ell_{2} \cap B(w,r)$, and write $p_{2} := i(w_{2}) \in L_{2}$. Noting that $d(p,p_{2}) \leq 2C\delta r + d(f(w),f(w_{2})) \leq 2Mr$, it follows that
\begin{equation}\label{form12} J \subset [i(w_{2}) - R - 2Mr,i(w_{2}) + R + 2Mr] =: [i(w_{2}) - C'r, i(w_{2}) + C'r], \end{equation}
where $C' := AH^{3}(\epsilon/2)^{-2} + 2M$. Recall here that $H = 4M$, so $C' \lesssim_{\epsilon,M} 1$.

We first claim that there exist two points $a',b' \in I = \ell_{2} \cap B(w,\tfrac{C}{2}r)$ such that $J \subset [i(a'),i(b')]$. To see this, write $a := w_{2}$, and pick an arbitrary point $b \in \ell_{2} \cap B(w,5r)$ with $|a - b| = r > 2MC\delta r$. Since $i$ is an $(M,C\delta r)$-QIE, it follows from Remark \ref{rem1} that $i(a) \neq i(b)$. We assume with no loss of generality that
\begin{equation}\label{form11} a < b \quad \text{and} \quad i(a) < i(b). \end{equation}
According to Lemma \ref{lemma2}, the map $i$ is $2MC\delta r$-monotone, and since $a < b - 2MC\delta r$, we infer from \eqref{form11} that $i$ is $2MC\delta r$-increasing. In particular, if $a' \in I$ is a further point with $a' < a - 2MC\delta r$, then $i(a') < i(a)$. In particular, this holds if we pick $a' \in I = \ell_{2} \cap B(w,\tfrac{C}{2}r)$ with $a' < a$ at distance $\geq Cr/10$ from $a$. Then
\begin{displaymath} i(a) - i(a') \geq \frac{a - a'}{M} - C\delta r \geq \frac{Cr}{10M} - C\delta r \geq \frac{Cr}{20M}, \end{displaymath}
recalling from \eqref{form13} that $\delta < 1/(100M)$. Similarly, we may find a point $b' \in I$ with $b' > b$ at distance $\geq Cr/10$ from $b$ with the property
\begin{displaymath} i(b') - i(a) \stackrel{\eqref{form11}}{>} i(b') - i(b) > Cr/(20M). \end{displaymath}
Combining these estimates with \eqref{form12}, and noting that $C > 20MC'$ by \eqref{form13}, we find that $J \subset [i(a'),i(b')]$, as required. Next, we infer from Lemma \ref{lemma2} that $i(I)$ is $(C\delta r)$-dense in $[i(a'),i(b')]$, in particular in $J$. Unwrapping the definitions, this means that for every $q \in J = B(p,R) \cap L_{2}$, we may find a point $v \in I = \ell_{2} \cap B(w,\tfrac{C}{2}r)$ with the property $d(q,i(v)) \leq C\delta r$. Moreover, $i(v)$ was the closest point of $f(v)$ to $L_{2}$, hence $d(i(v),f(v)) \leq C\delta r$. It follows that $d(q,f(v)) \leq 2C\delta r$, and the proof of \eqref{form10} is complete. \end{proof}

As a corollary of Propositions \ref{propCarleson} and \ref{prop1}, we record that bilipschitz images of $\W$ inside $\He$ satisfy the \emph{weak geometric lemma} for the strong vertical $\beta$-numbers (hence also for the standard vertical $\beta$-numbers):
\begin{cor}[Weak geometric lemma]\label{wgl} Let $f \colon \W \to \He$ be an $M$-bilipschitz map, $M \geq 1$. Then, for any $\epsilon \in (0,1]$ and $C \geq 1$, we have
\begin{equation}\label{form129} \mathop{\sum_{Q \subset Q_{0}}}_{\beta_{f}(CQ) \geq \epsilon} \ell(Q)^{3} \lesssim_{C,\epsilon,M} \ell(Q_{0})^{3}, \qquad Q_{0} \in \mathcal{D}. \end{equation}
\end{cor}
\begin{proof} Fix $C \geq 1$ and $\epsilon \in (0,1]$. Choose $C' = C'(C,\epsilon,M) \geq C$ and $\delta = \delta(C,\epsilon,M) > 0$ such that
\begin{displaymath} \rho_{f}(B(w,C'r)) < \delta \quad \Longrightarrow \quad \beta_{f}(B(w,Cr)) < \epsilon, \qquad w \in \W, \, r > 0. \end{displaymath}
This is possible by Proposition \ref{prop1}. Thus,
\begin{displaymath} \{Q \in \mathcal{D} : \beta_{f}(CQ) \geq \epsilon\} \subset \{Q \in \mathcal{D} : \rho_{f}(C'Q) \geq \delta\} = \mathcal{B}(C',\delta). \end{displaymath}
Now \eqref{form129} follows from Proposition \ref{propCarleson} applied with constants $C',\delta$. \end{proof}

\subsection{Ruled and vertical rectangles}\label{s:ruledRectangle} We have already learned from Proposition \ref{prop1} that small ruler coefficients imply small vertical $\beta$-numbers. In this section, we formalise the connection a little further, and add the notion of "slope" to the soup. 

We start with some notation on parametrising horizontal lines and vertical planes. Let $L \subset \He$ be a horizontal line which is not contained on any translate of the $xt$-plane. Then, $L$ can be written as
\begin{displaymath} L = \{(ay + b,y,\tfrac{1}{2}by + c) : y \in \R\} \end{displaymath}
for some unique triple $a,b,c \in \R$. The absolute value of the coefficient "$a$" will be called the \emph{slope of $L$}, denoted $\angle(L)$, or $\angle(L,\W)$ if the relation to the $yt$-plane needs emphasising. If $L$ is contained in the $xt$-plane, it cannot be expressed as above, and we declare $\angle(L) := \infty$. If $\V := \V(L)$ is the vertical plane spanned by $L$, then the \emph{slope of $\V$} equals, by definition, the slope of $L$, and we write $\angle(\V(L)) := \angle(L)$. 

\begin{definition}[Ruled rectangle]\label{goodness} Let $Q \in \calD$, let $C \geq 10$ and $\delta,\Sigma > 0$ be constants, and let
$f \colon B(c_{Q},C\ell(Q)) \to \mathbb{H}$ be a map. The
rectangle $Q \in \calD$ is called \emph{$(\delta,C)$-ruled} if
$\rho_{f}(CQ) < \delta/C$, where $CQ = B(c_{Q},C\ell(Q))$ as usual. This means, by definition, that for every horizontal line $\ell \subset \W$, there exists a horizontal line $L \subset \He$ such that
\begin{equation}\label{form18} f(\ell \cap CQ) \subset N(L,\delta \ell(Q)). \end{equation}
The rectangle $Q$ is called $(\delta,C,\Sigma)$-ruled if it is
$(\delta,C)$-ruled, and there exists a horizontal line $\ell
\subset \W$ such that $\ell \cap Q \neq \emptyset$, and such that
\textbf{some} horizontal line $L \subset \He$ satisfying \eqref{form18} has
slope $\angle(L) \leq \Sigma$.
\end{definition}

\begin{definition}[Vertical rectangle]\label{vertical}
Let $Q\in \mathcal{D}$, let $H\geq 1$, $\epsilon \in (0,1]$,
$\Sigma>0$ be constants and let $f:HQ \to \mathbb{H}$ be a map.
The rectangle $Q$ is called \emph{$(\epsilon,H,\Sigma)$-vertical}
if there exists a vertical plane $\mathbb{V}_Q\subset \mathbb{H}$
with $\angle(\V_{Q}) \leq \Sigma$ such that $\beta_{f}(\V_{Q}; HQ) < \epsilon/H$. \end{definition}

To avoid over-indexing, we omit reference to $f$ in the notation for ruled and vertical rectangles $Q \in \mathcal{D}$. We next record a straightforward corollary of Proposition \ref{prop1}, which says that ruled rectangles are vertical rectangles. The only additional information is that the slopes of the ruling control the slopes of well-approximating vertical planes.

\begin{cor}[to Proposition \ref{prop1}]\label{cor2} Let $f \colon \W \to \He$
be an $M$-bilipschitz map, $M \geq 1$, and let $\epsilon \in (0,1)$, $H \geq 1$, and $\Sigma > 0$.
 Then, there exists $C_{0} = C_{0}(\epsilon,H,M) \geq 10$ and $\delta_{0} = \delta(\epsilon,C_{0},H,M) > 0$ such that the following holds
 for all $C \geq C_{0}$ and $\delta \in (0,\delta_{0}]$. If $Q \in \calD$ is
 $(\delta,CH,\Sigma)$-ruled, then it is $(\epsilon,H,\Sigma)$-vertical.\end{cor}

\begin{proof} Recall that $HQ := B(c_{Q},H\ell(Q))$. We apply Proposition \ref{prop1} with centre $w = c_{Q}$ and radius $r = H\ell(Q)$, and $\epsilon' := \epsilon/H$ in place of $\epsilon$. The conclusion is that if
\begin{displaymath} \rho_{f}(B(w,Cr)) = \rho_{f}(CHQ) < \delta \end{displaymath}
for sufficiently large $C \geq 1$ and sufficiently small $\delta > 0$,
depending on $\epsilon,H,M$, then there exists a vertical plane $\V_{Q}$ with $\beta_{f}(\V_{Q};HQ) < \epsilon/H$. Moreover, as explained in Remark \ref{rem2}, the plane $\V_{Q}$ can be chosen as $\V_{Q} = \V(L)$ for any horizontal line $L \subset \He$ with the following property: there is a horizontal line $\ell \subset \W$ with $\ell \cap HQ = \ell \cap B(w,r) \neq \emptyset$ such that
\begin{equation}\label{form20} f(\ell \cap CHQ) \subset N(L,\delta CH\ell(Q)). \end{equation}
Now we use that $Q$ is $(\delta,CH,\Sigma)$-ruled: there exists a horizontal line $\ell
\subset \W$ with $\ell \cap Q \neq \emptyset$ such that \eqref{form20} holds, and such
that $\angle(L) \leq \Sigma$. Hence, $\angle(\V(L)) \leq \Sigma$, and the choice
$\V_{Q} := \V(L)$ satisfies all the requirements needed to show
that $Q$ is $(\epsilon,H,\Sigma)$-vertical.
\end{proof}

%%%%%%%%%%%%%%%%%%%%%%%%%%%%%%%%%%%%%%%%%%%

\subsection{The local quasi-isometry $F_{Q}$}\label{s:tangent}

Let $f \colon \W \to \He$ be an $M$-bilipschitz map, $M \geq 1$,
 and fix $\epsilon \in (0,1)$, $H \geq 1$, and $\Sigma > 0$. Then, if  $Q \in \calD$ is an $(\epsilon,H,\Sigma)$-vertical rectangle, Definition \ref{vertical} implies that $f|_{HQ}$ composed with
closest-point-projection to $\V_{Q}$ is an $(M,\epsilon
\ell(Q))$-QIE.
 However, it will be convenient to "normalise" this QIE in such a way that its target is always $\W$, instead of the changing vertical plane $\V_{Q}$.
 For this purpose, we will next construct a horizontal QIE
 $F_{Q} :HQ
 \to \mathbb{W}$ of the form $F_{Q} =
 \pi_{\V_{Q}} \circ \iota_{Q}$, where $\iota_{Q}\colon HQ \to \V_{Q}$ is the composition of $f|_{HQ}$ with (almost) the
 closest-point-projection to $\V_{Q}$,
 and $\pi_{\V_{Q}} : \mathbb{V}_{Q} \to \mathbb{W}$ is a bilipschitz map defined in \eqref{form21} below.
 
 We first define the map $\iota:=\iota_{Q} \colon HQ \to \V_{Q}$, where $Q \in \calD$ is an $(\epsilon,H,\Sigma)$-vertical
  rectangle (in fact, we may define $\iota_{Q}$ on $\W$, but it is useless outside $HQ$). Recall from Definition \ref{vertical} that for every horizontal line $\ell \subset \W$, there exists a horizontal line $L = L_{Q,\ell} \subset \V_{Q}$ with the property
 \begin{equation}\label{eq:LineCurveApprox}
 f(\ell \cap
HQ) \subset N(L,\epsilon \ell(Q)).
 \end{equation}
Since
%$Q$ is a $(\delta,CH,\Sigma)$ good rectangle,
the $\angle(\mathbb{V}_Q) \leq \Sigma$, there exist
constants $a_{Q},b_{Q},c_{Q,\ell} \in \mathbb{R}$ with $|a_{Q}|
\leq \Sigma$ such that
\begin{equation}\label{eq:L}
L_{Q,\ell} =\{(a_{Q}y+b_{Q},y,\tfrac{1}{2}b_{Q}y +
c_{Q,\ell}):\,y\in\mathbb{R}\}.
\end{equation}
The following lemma will be used to define the map $\iota_Q$.
\begin{lemma}\label{l:Proj_Same_y} Let $L :=\{(ay+b,y,\tfrac{1}{2}by + c):\,y\in\mathbb{R}\}$ be a horizontal line with finite slope. For $(x,y,t) \in \He$, define the map
\begin{displaymath} i_{L}(x,y,t) := (ay + b,y,\tfrac{1}{2}by + c), \end{displaymath} 
which picks the unique point on $L$ with which has the same second coordinate as $(x,y,t)$. Then,
\begin{displaymath}
d((x,y,t),\iota_{L}(x,y,t)) \leq
A(1+|a|)\,\mathrm{dist}((x,y,t),L),\qquad (x,y,t)\in \mathbb{H},
\end{displaymath}
for some absolute constant $A\geq 1$.
\end{lemma}

\begin{proof} Given $p:=(x,y,t)\in \mathbb{H}$, let $p':=(ay',y',\frac{1}{2}b
y'+c)$ be a point on the horizontal line $L$ that realizes the
distance to $(x,y,t)$. Then,
\begin{displaymath}
d(p,i_{L}(p)) \leq d(p,p')+
d(p',i_{L}(p)) \lesssim \mathrm{dist}(p,L) + (1+|a|)|y-y'|.
\end{displaymath}
Here the upper bound for the second summand is an easy consequence
of the fact that $i_{L}(p) = (ay+b,y,\tfrac{1}{2}by +c)$ and
$p'=(ay'+b,y',\tfrac{1}{2}by' +c)$ lie on the same horizontal line
$L$, and $L$ is a left translate of the line $\{(ay,y,0):\,y\in
\mathbb{R}\}$ by $(b,0,c)$.
\end{proof}

With the lemma above in mind, we define $\iota_{Q}|_{\ell}
\colon \ell \to L \subset \V_{Q}$ by setting $\iota_{Q}(w) := i_{L_{Q,\ell}}(f(w))$, that is
\begin{equation}\label{iota}
\iota_Q (w):= \left(a_{Q} f_2(w)+ b_{Q}, f_2(w), \tfrac{1}{2}
b_{Q} f_2(w)+ c_{Q,\ell}\right),\quad w\in \ell.
\end{equation}
The constants $a_{Q},b_{Q},c_{Q,\ell}$ are defined in
\eqref{eq:L}. As $\ell \subset \W$ ranges over horizontal lines,
this defines $\iota_{Q}$ on $\W$. Since $f$ is $M$-bilipschitz and
\eqref{eq:LineCurveApprox} holds, and $|a_{Q}|\leq \Sigma$, Lemma
\ref{l:Proj_Same_y} implies that $\iota_{Q}$ is an
$(M,A(1+\Sigma)\epsilon \ell(Q))$-QIE on $HQ$ (of course not necessarily on $\W$), which maps each horizontal line $\ell \subset \W$ inside the horizontal line $L_{Q,\ell} \subset \V_{Q}$. Thus, identifying $\V_{Q}$ with $\W$ in some abstract way, one can already view $\iota_{Q} \colon HQ \to \V_{Q}$ as a horizontal QIE in the sense of Definition \ref{hQIE}. It pays off to make the identification $\V_{Q} \cong \W$ more explicit, and this is what we do next.

Let $\V:= \{(ay+b,y,t):\, (y,t)\in \mathbb{R}^2\}$ be a vertical plane with $\angle(\V) < \infty$. Let $\Pi \colon \He \to \W$ be the standard vertical projection $\Pi(x,y,t) = (y,t + \tfrac{xy}{2})$, and set
 $\pi_{\V}:=
  \psi_{\V}\circ \Pi$, where $\psi_{\V}(y,t):= (y,t-\tfrac{1}{2}ay^2-by)$. In other words,
 \begin{equation}\label{form21}
\pi_{\V}(x,y,t) = \left(y, t+
\tfrac{1}{2}xy -\tfrac{1}{2}ay^2-by\right), \qquad (x,y,t) \in \He.
\end{equation}

\begin{lemma}\label{l:QIE_W} The map $\pi_{\V}$ maps horizontal lines in $\mathbb{V}$ onto horizontal lines in
$\mathbb{W}$, and is $A(1 + |a|)$-bilipschitz for some absolute constant $A \geq 1$.
\end{lemma}

\begin{proof} We first observe that
\begin{equation}\label{eq:ProjForm}
\pi_{\V}\left(ay+b,y,c+ \tfrac{1}{2}by\right)=(y,c),\quad y,c\in
\mathbb{R}.
\end{equation}
Since every horizontal line in $\mathbb{V}$ is of the form
\begin{displaymath}
L\cdot (0,0,c)=\{(ay+b,y,c+\tfrac{1}{2}by):\,y\in \mathbb{R}\}
\end{displaymath}
for some $c\in\mathbb{R}$, formula \eqref{eq:ProjForm} proves that
$\pi_{\V}$ maps horizontal lines in $\mathbb{V}$ to horizontal
lines in $\mathbb{W}$. To prove the bilipschitz continuity of
$\pi_{\V}$, let
\begin{displaymath}
p=\left(ay+b,y,c+\tfrac{b}{2}y\right)\quad\text{and}\quad
p'=\left(ay'+b,y',c'+\tfrac{b}{2}y'\right)
\end{displaymath}
be two points in $\mathbb{V}$. It follows from \eqref{eq:ProjForm}
that
\begin{displaymath}
d_{\mathrm{par}}(\pi_{\V}(p),\pi_{\V}(p'))=
d_{\mathrm{par}}((y,c),(y',c'))\sim |y-y'|+\sqrt{|c-c'|}.
\end{displaymath}
On the other hand, it is easy to check that
\begin{displaymath}
d(p,p') \sim (1 + |a|)|y-y'|+ \sqrt{|c-c'|}.
\end{displaymath}
Hence $d(p,p')/(1 + |a|) \lesssim
d_{\mathrm{par}}(\pi_{\V}(p),\pi_{\V}(p')) \lesssim (1 +
|a|)d(p,p')$, as claimed.
\end{proof}

Combined with the properties
 of $\iota_{Q}$, which we  stated below
Lemma \ref{l:Proj_Same_y}, Lemma \ref{l:QIE_W} yields the
following corollary (and definition).

\begin{cor}\label{c:HorizQIE_W_W} Let $A\geq 1$ be an absolute
constant as in Lemma  \ref{l:Proj_Same_y} and Lemma \ref{l:QIE_W}.
Let $f \colon \W \to \He$ be an $M$-bilipschitz map, $M \geq 1$,
 and fix $\epsilon \in (0,1)$, $H \geq 1$, and $\Sigma > 0$.
 If $Q\in \mathcal{D}$ is an
$(A^{-2}(1+\Sigma)^{-2}\epsilon,H,\Sigma)$-vertical rectangle,
then
\begin{displaymath}
F_{Q} := \pi_{\V_{Q}}\circ \iota_{Q} \colon \W \to \mathbb{W}
\end{displaymath}
is a horizontal $(A(1 + \Sigma)M,\epsilon\ell(Q))$-QIE on $HQ$ which has the explicit expression
\begin{equation}\label{form45} F_{Q}(w) = (f_{2}(w),c_{Q,\ell}), \qquad w \in \ell, \end{equation}
for any horizontal line $\ell \subset \W$. Here $c_{Q,\ell}$ is the coefficient defined in \eqref{eq:L}.
\end{cor}

So, the map $F_{Q}$ may be defined on the whole plane $\W$, but it is only a horizontal QIE restricted to $HQ$. For this reason, we will often treat $F_{Q}$ as if it only were defined on $HQ$.

\begin{proof}[Proof of Corollary \ref{c:HorizQIE_W_W}] As we discussed below Lemma \ref{l:Proj_Same_y},
with these choices of parameters, the maps $\iota_{Q}$ associated
to $(A^{-2}(1+\Sigma)^{-2}\epsilon,H,\Sigma)$-vertical rectangles are
$(M,A^{-1}(1 + \Sigma)^{-1}\epsilon \ell(Q))$-QIEs $HQ \to \V_{Q}$
which send horizontal line segments to horizontal line segments.
Moreover, since $\pi_{\V_{Q}} \colon \V_{Q} \to \W$ is $A(1 +
\Sigma)$-bilipschitz by Lemma \ref{l:QIE_W}, and sends horizontal
lines in $\V_{Q}$ to horizontal lines in $\W$, the first claim follows. The formula \eqref{form45} follows directly from the definitions of $\iota_{Q}$ in \eqref{iota}, and $\pi_{\V_{Q}} = \psi_{\V_{Q}} \circ \Pi$ in \eqref{form21}:
\begin{align*} F_{Q}(w) & = \psi_{\V_{Q}}(\Pi[a_{Q} f_2(w)+ b_{Q}, f_2(w), \tfrac{1}{2}
b_{Q} f_2(w)+ c_{Q,\ell}])\\
& = \psi_{\V_{Q}}(f_{2}(w), b_{Q} f_2(w)+ c_{Q,\ell} + \tfrac{1}{2}a_{Q}f_{2}(w)^{2})\\
& = (f_{2}(w),\tfrac{1}{2}
b_{Q} f_2(w)+ c_{Q,\ell} + \tfrac{1}{2}a_{Q}f_{2}(w)^{2} - \tfrac{1}{2}a_{Q}f_{2}(w)^{2} - b_{Q}f_{2}(w))\\
& = (f_{2}(w),c_{Q,\ell}), \qquad w \in \ell \subset \W. \end{align*} 
This completes the proof of the corollary -- and the definition of $F_{Q}$. \end{proof}

To end this section, we quantify how much the maps $F_{Q}$ and $F_{\hat{Q}}$ can
differ from each other if $\hat{Q} \in \calD$ is the parent of $Q
\in \calD$, and both $Q,\hat{Q}$ are $(\epsilon,H,\Sigma)$-vertical.

\begin{lemma}\label{lemma5} Let $Q,\hat{Q} \in \calD$ be $(\epsilon,H,\Sigma)$-vertical rectangles such that $\hat{Q}$ is
the parent of $Q$. Assume moreover that $\epsilon <
cM^{-1}(1 + \Sigma)^{-1}$ for a small absolute constant $c >
0$. Then,
\begin{displaymath}
d_{\mathrm{par}}(F_{Q}(w),F_{\hat{Q}}(w)) \lesssim
(1+\Sigma)^{\frac{1}{2}}\epsilon \ell(Q), \qquad w \in HQ.
\end{displaymath}
\end{lemma}

\begin{proof} Write $f = (f_{1},f_{2},f_{3}) \colon \W \to \He$, let $w \in HQ$, and let $\ell \subset \W$ be the unique horizontal line containing $w$. From the expression \eqref{form45}, we deduce that 
\begin{displaymath} F_{Q}(w) = (f_{2}(w),c_{Q,\ell}) \quad \text{and} \quad F_{\hat{Q}}(w) = (f_{2}(w),c_{\hat{Q},\ell}), \end{displaymath}
so $d_{\mathrm{par}}(F_{Q}(w),F_{\hat{Q}}(w)) = |c_{Q,\ell} -
c_{\hat{Q},\ell}|$. Let $L_{Q,\ell},L_{\hat{Q},\ell} \subset \V_{Q}$ be the horizontal lines familiar from \eqref{eq:LineCurveApprox}, with $\angle(L_{Q,\ell}),\angle(L_{\hat{Q},\ell}) \leq \Sigma$, and which satisfy
\begin{displaymath} f(\ell \cap HQ) \subset N(L_{Q},\epsilon \ell(Q)) \quad
\text{and} \quad f(\ell \cap HQ) \subset N(L_{\hat{Q}},2\epsilon
\ell(Q)). \end{displaymath} Let $w_{1},w_{2} \in \ell \cap Q$ with
$|w_{1} - w_{2}| = \ell(Q)$. Then $d(f(w_{1}),f(w_{2})) \geq
M^{-1}\ell(Q)$, and
\begin{displaymath} \dist(f(w_{j}),L_{Q}) \leq \epsilon \ell(Q) \quad
\text{and} \quad \dist(f(w_{j}),L_{\hat{Q}}) \leq 2\epsilon
\ell(Q) \end{displaymath} for $j \in \{1,2\}$. Now, since
$\epsilon \ll M^{-1}(1 + \Sigma)^{-1}$, Lemma \ref{lemma0} can be
applied with $r = M^{-1}\ell(Q)$ and $2M\epsilon$ in place of
$\epsilon$. The conclusion is that \begin{displaymath}|c_{Q,\ell}
- c_{\hat{Q},\ell}| \lesssim (1+\Sigma)(2M\epsilon \cdot
M^{-1}\ell(Q))^{2} \sim (1+\Sigma)(\epsilon \ell(Q))^{2},
\end{displaymath} and hence
$d_{\mathrm{par}}(F_{Q}(w),F_{\hat{Q}}(w)) \lesssim
(1+\Sigma)^{1/2}\epsilon \ell(Q)$, as claimed. \end{proof}

\subsection{Approximate quadrics}\label{s:appQuad} Recall that if $Q \in \mathcal{D}$ is an $(\epsilon,H,\Sigma)$-vertical rectangle, then $f(\ell \cap HQ) \subset N(L,\epsilon \ell(Q))$ for a certain horizontal line $L = L_{Q,\ell} \subset \He$ with $\angle(L) \leq \Sigma$. If we write $L = \{(ay + b,y,\tfrac{1}{2}by + c) : y \in \R\}$, as usual, then the image 
\begin{displaymath} \mathcal{Q}_{Q,\ell} := \Pi(L_{Q,\ell}) \subset \W \end{displaymath}
of $L_{Q,\ell}$ under the vertical projection $\Pi$ is the quadric $\mathcal{Q}_{Q,\ell} = \{(y,q_{Q,\ell}(y)) : y \in \R\}$ parametrised by 
\begin{displaymath} q_{Q,\ell}(y):= \tfrac{1}{2}ay^2 + by + c, \qquad y \in \R. \end{displaymath}
We will call $q_{Q,\ell}$ the \emph{$(Q,\ell)$-approximate quadric}. This notion will be central in the proof of Theorem \ref{main}: the intrinsic bilipschitz graphs $\Gamma_{j}$ appearing in \eqref{form143} of the corona decomposition are constructed, roughly speaking, by gluing together $(Q,\ell)$-approximate quadrics for various rectangles $Q$ and lines $\ell$.

Since the graph $\mathcal{Q}_{Q,\ell}$ is the $\Pi$-projection of the horizontal line $L_{Q,\ell}$ which approximates $f(\ell \cap HQ)$ well in $\He$, one expects $\mathcal{Q}_{Q,\ell}$ to approximate $\Pi(f(\ell \cap HQ))$ well in $\W$. The next two lemmas quantify this heuristic.

\begin{lemma}\label{lemma9} Let $L = \{(ay + b,y,\tfrac{1}{2}by + c) : y \in \R\}$ be a horizontal line. Write $q_{L}(y) := \tfrac{1}{2}ay^{2} + by + c$ and $\mathcal{Q}_{L}(y) := (y,q_{L}(y))$. Then, 
\begin{displaymath}
\max\{d_{\mathrm{par}}(\Pi(p),\mathcal{Q}_{L}(y)),|x - \dot{q}_{L}(y)|\}
\lesssim (1+|a|)\mathrm{dist}(p,L), \qquad p = (x,y,t) \in \He.
\end{displaymath}
\end{lemma}

\begin{proof}
Since $\Pi(p)=(y,t+\frac{xy}{2})$ and the points
$\Pi(p),\mathcal{Q}_{L}(y) \in \mathbb{W}$ have the same $y$-coordinate, and $\dot{q}_{L}(y) = ay + b$,
we find that
\begin{equation}\label{eq:Difference1}
d_{\mathrm{par}}(\Pi(p),\mathcal{Q}_{L}(y)) 
=\left|\left(t+\tfrac{xy}{2}\right)-\left(\tfrac{1}{2}ay^2 + by +
c\right)\right|^{\frac{1}{2}}
\end{equation}
and
\begin{equation}\label{eq:Difference2} |x - \dot{q}_{L}(y)| = |x - (ay + b)|. \end{equation}
Both \eqref{eq:Difference1} and \eqref{eq:Difference2} are bounded from above by (a constant times) 
\begin{displaymath}
d(p,(ay+b,y,\tfrac{1}{2}by+c)) = d(p,i_{L}(p)) \leq A (1+|a|)\mathrm{dist}(p,L),
\end{displaymath}
where we used Lemma \ref{l:Proj_Same_y} in the final estimate. This completes the proof.  \end{proof}

\begin{cor}\label{cor4} Let $Q \in \mathcal{D}$ be an $(\epsilon,H,\Sigma)$-vertical rectangle, let $\ell \subset \W$ be a horizontal line, and let $q := q_{Q,\ell}$ be the $(Q,\ell)$-approximate quadric. Then,
\begin{equation}\label{form49} |(f_{3}(w) + \tfrac{1}{2}f_{1}(w)f_{2}(w)) - q(f_{2}(w))| \lesssim (1 + \Sigma)^{2}\epsilon^{2} \ell(Q)^{2}, \qquad w \in \ell \cap HQ, \end{equation} 
and
\begin{equation}\label{form50} |f_{1}(w) - \dot{q}(f_{2}(w))| \lesssim (1 + \Sigma) \epsilon \ell(Q), \qquad w \in \ell \cap HQ. \end{equation}
\end{cor}
\begin{proof} Apply Lemma \ref{lemma9} to the point $p := f(w) = (f_{1}(w),f_{2}(w),f_{3}(w)) =: (x,y,t)$ for $w \in \ell \cap HQ$, which satisfies $\dist(p,L) \leq \epsilon \ell(Q)$ by the assumption that $Q$ is $(\epsilon,H,\Sigma)$-vertical. The horizontal line $L \subset \V_{Q}$ appearing here has slope $|a| \leq \Sigma$. Now, the left hand side of \eqref{form49} is the parabolic distance squared between $\Pi(p)$ and $\mathcal{Q}_{L}(y)$. Consequently, both \eqref{form49} and \eqref{form50} follow directly from the estimate in Lemma \ref{lemma9}. \end{proof} 

To emphasise the obvious, \eqref{form49}-\eqref{form50} say that the values of $q$ and $\dot{q}$ evaluated at $f_{2}(y,t)$ -- and \textbf{not} at $y$! -- are essentially determined by $f(y,t)$ for $(y,t) \in \ell \cap HQ$.

%%%%%%%%%%%%%%%%%%%%%%%%%%%%%%%%%%%%%%%%%%%

\subsection{Vertical trees} Recall the definition of \emph{trees} in $\mathcal{D}$ from Section \ref{s:corona}. Before introducing concrete trees, we make the following choices of constants: for given parameters $M \geq 1$ (the bilipschitz constant of $f$) and $\Sigma > 0$
(the "maximal slope" constant), we set
\begin{equation}\label{form22} N := A(1 + \Sigma)M \quad \text{and} \quad 0 < \epsilon < cN^{-10}, \end{equation}
where $A \geq 1$ is the absolute constant from Corollary \ref{c:HorizQIE_W_W}, and $c \in (0,\tfrac{1}{100})$ is another absolute constant, whose size will have to be adjusted in upcoming arguments. Corollary \ref{c:HorizQIE_W_W} verifies that if $Q$ is an $(A^{-2}(1 + \Sigma)^{-2}\epsilon,H,\Sigma)$-vertical rectangle, then $F_{Q}$ is a horizontal $(N,\epsilon \ell(Q))$-QIE $HQ \to \W$.

\begin{definition}[Vertical tree]\label{greenTree}
Let $f \colon \W \to \He$ be an $M$-bilipschitz map $M \geq 1$,
and fix parameters $\epsilon > 0$, $H \geq 1$, and $\Sigma > 0$. A
tree $\mathcal{T} \subset \mathcal{D}$ is called a
\emph{$(\epsilon,H,\Sigma)$-vertical tree} (or just a \emph{vertical
tree}) if it consists of $(A^{-2}(1 +
\Sigma)^{-2}\epsilon,H,\Sigma)$-vertical rectangles.
\end{definition}

We next verify, as a consequence of Corollary \ref{cor3}, that the QIEs $F_{Q} \colon HQ \to \W$ are vertically monotone on a smaller parabolic ball $BQ \subset HQ$ whenever $H \gg (1 + \Sigma)^{2}M^{2}B$.
\begin{lemma}\label{lemma4} Let $B,M \geq 1$ and $\epsilon > 0$.
 Let $H := A(1 + \Sigma)^{2}M^{2}B$ for a sufficiently large absolute constant $A \geq 1$.
 Let $Q \in \calD$ be an $(A^{-2}(1 + \Sigma)^{-2}\epsilon,H,\Sigma)$-vertical rectangle. Then, the horizontal $(N,\epsilon \ell(Q))$-QIE $F_{Q} \colon HQ \to \W$ is vertically $\epsilon' \ell(Q)$-monotone on $BQ$ with $\epsilon' := 2N\epsilon$.
\end{lemma}

\begin{proof} By Corollary \ref{cor3}, the vertical $2N\epsilon \ell(Q)$-monotonicity of $F_{Q}$ on $BQ$ will follow from the $(N,\epsilon \ell(Q))$-QIE property of $F_{Q}$ on $HQ$, provided that $H \geq AN^{2}B \sim (1 + \Sigma)^{2}M^{2}B$ for some absolute constant $A \geq 1$. Indeed, $BQ$ fits in a parabolic rectangle $I \times J$ of dimensions approximately $B\ell(Q) \times (B\ell(Q))^{2}$. To conclude that $F_{Q}$ is vertically $2N\epsilon \ell(Q)$-monotone on such a rectangle, Corollary \ref{cor3} needs $F_{Q}$ to be an $(N,\epsilon \ell(Q))$-QIE on a concentric parabolic rectangle of dimensions approximately
\begin{displaymath} [N^{2}B\ell(Q) + N\epsilon \ell(Q)] \times (B\ell(Q))^{2}. \end{displaymath}
Evidently $N\epsilon \leq N^{2}B$, so this his bigger rectangle fits inside $HQ$ for some $H \sim N^{2}B$. \end{proof}

Motivated by the lemma above, we fix a parameter $B \geq 5$, and in the sequel we will always require
 \begin{equation}\label{H} H \geq A(1 + \Sigma)^{2}M^{2}B \end{equation}
 for a sufficiently large absolute constant $A \geq 10$ such that the conclusion of Lemma \ref{lemma4} holds: whenever $Q$ is a rectangle in a vertical tree, then the restriction of $F_{Q} \colon HQ \to \W$ to $BQ$ is vertically $2N\epsilon \ell(Q)$-monotone. In the end, the choice $B := 5$ works, but we will keep the special notation "$B$" to emphasise the "meaning of the constant $5$".

\begin{definition}[Signature of $Q$]
 The \emph{signature of $Q$} is the signature (recall Definition \ref{signature}) of the vertically $\epsilon' \ell(Q)$-monotone map $F_{Q} \colon BQ \to \W$
 defined in Corollary \ref{c:HorizQIE_W_W}. \end{definition}

Recall, above, that $\epsilon' = 2N\epsilon < N^{-9}$ by
\eqref{form22}. We then arrive at the main result of this section, which states that signature within $(\epsilon,H,\Sigma)$-vertical trees are constant:
\begin{proposition}\label{prop2} Let $B \geq 5$, and let $\mathcal{T} \subset \mathcal{D}$ be an $(\epsilon,H,\Sigma)$-vertical tree, where $H$ is given by \eqref{H}. Then, all the rectangles $Q
\in \calT$ have the same signature.
\end{proposition}

\begin{proof} It suffices to show that the signatures of $Q$ and $\hat{Q}$ are the same for all $Q,\hat{Q} \in \mathcal{T}$, where $\hat{Q}$ is the $\mathcal{D}$-parent of $Q$. Fix $Q \in \calT$ such that $\hat{Q} \in \mathcal{T}$, and assume for example that the signature of $\hat{Q}$, hence $F_{\hat{Q}} \colon B\hat{Q} \to \W$, is $+$. This means, by definition, that the maps $t \mapsto \pi_{2}(F_{\hat{Q}}(y,t))$ are $\epsilon' \ell(\hat{Q}) = 2\epsilon' \ell(Q)$-increasing for all $y \in \pi_{1}(B\hat{Q})$. The task is to show that the maps $t \mapsto \pi_{2}(F_{Q}(y_{0},t))$ are $\epsilon' \ell(Q)$-increasing for all $y \in \pi_{1}(BQ)$. For this, we will use Lemma \ref{lemma5} to compare $F_{Q}$ to $F_{\hat{Q}}$.

Let us prove first the $\epsilon' \ell(Q)$-increasing property of $t \mapsto \pi_{2}(F_{Q}(y,t))$. Since we already know that $F_{Q}$ is vertically $\epsilon' \ell(Q)$-monotone on $BQ$, it suffices to "test" the signature with a parameter $y \in \pi_{1}(BQ)$ of our choosing. We choose $y$ so that the vertical line $\{(y,t) : t \in \R\}$ intersects $Q$. Then, we may pick $t_{1} < t_{2}$ such that $(y,t_{1}),(y,t_{2}) \in Q$ and $d_{\mathrm{par}}((y,t_{1}),(y,t_{2})) = \ell(Q)$, or in other words $t_{2} = t_{1} + \ell(Q)^{2}$. Since certainly $(y,t_{j}) \in HQ$ for $j \in \{1,2\}$, Lemma \ref{lemma5} is applicable and shows that
\begin{equation}\label{form24} d_{\mathrm{par}}(F_{Q}(y,t_{j}),F_{\hat{Q}}(y,t_{j})) \leq A(1 + \Sigma)^{\tfrac{1}{2}}\epsilon \ell(Q), \qquad j \in \{1,2\}. \end{equation}
On the other hand, since $F_{Q}$ is a horizontal $(N,\epsilon \ell(Q))$-QIE on $HQ$, Lemma \ref{lemma6} implies that the map $t \mapsto \pi_{2}(F_{\hat{Q}}(y,t))$ is an $(N,2\epsilon \ell(Q))$-QIE on $([t_{1},t_{2}],\|\cdot\|)$ (the constant $H \geq AM^{2}$ is certainly large enough for this purpose), and consequently
\begin{equation}\label{form23} \|\pi_{2}(F_{\hat{Q}}(y,t_{2})) - \pi_{2}(F_{\hat{Q}}(y,t_{1}))\| \geq \frac{\ell(Q)}{N} - 2\epsilon \ell(Q) \geq \frac{\ell(Q)}{2N},  \end{equation}
recalling from \eqref{form22} that $\epsilon < cN^{-10}$. Since, moreover, $t \mapsto \pi_{2}(F_{\hat{Q}}(y,t))$ is $2\epsilon' \ell(Q)$-increasing, and $\|t_{1} - t_{2}\| = \ell(Q) > 2\epsilon' \ell(Q)$, we may upgrade \eqref{form23} to
\begin{displaymath} \pi_{2}(F_{\hat{Q}}(y,t_{2})) - \pi_{2}(F_{\hat{Q}}(y,t_{1})) > \frac{\ell(Q)^{2}}{(2N)^{2}}. \end{displaymath}
Combining this separation property with \eqref{form24}, and noting that $A^{2}(1 + \Sigma)\epsilon^{2} < (4N)^{-2}$ by \eqref{form22} shows that
\begin{align} \pi_{2}(F_{Q}(y,t_{2})) - \pi_{2}(F_{Q}(y,t_{1})) & > \pi_{2}(F_{\hat{Q}}(y,t_{2})) - \pi_{2}(F_{\hat{Q}}(y,t_{1})) \notag \\
& \quad - d_{\mathrm{par}}(F_{\hat{Q}}(y,t_{1}),F_{Q}(y,t_{1}))^{2} \notag \\
& \quad - d_{\mathrm{par}}(F_{\hat{Q}}(y,t_{2}),F_{Q}(y,t_{2}))^{2} \notag \\
&\label{form25} > \frac{\ell(Q)^{2}}{(2N)^{2}} - 2 \cdot \frac{\ell(Q)^{2}}{(4N)^{2}} > 0. \end{align}
Hence, we have found two (more than) $\epsilon' \ell(Q)$-separated points $(y,t_{1}),(y,t_{2}) \in Q \subset BQ$ such that $\pi_{2}(F_{Q}(y,t_{2})) > \pi_{2}(F_{Q}(y,t_{1}))$. Since $F_{Q}$ is vertically $\epsilon' \ell(Q)$-monotone on $BQ$, this implies that the signature of $F_{Q}$ is "$+$", as claimed. \end{proof}

By the previous proposition, we may talk about the \emph{signature of a vertical tree $\mathcal{T} \subset \mathcal{D}$}. Since the treatments of the two possible signatures $+$ or $-$ are similar -- except that a large number of inequality signs need to be inverted -- we will typically assume that the signature of a vertical tree is $+$; vertical trees with this signature will be called \emph{positive}.

We close the section by recording the following global bilipschitz property of $f_{2}$:

\begin{proposition}\label{prop3} Let $f = (f_{1},f_{2},f_{3}) \colon \W \to \He$ be an $M$-bilipschitz map, $M \geq 1$, and let $\mathcal{T} \subset \mathcal{D}$ be a vertical tree
associated with $f$.\footnote{In other words, the rectangles in
$\mathcal{T}$ are $(A^{-2}(1 +
\Sigma)^{-2}\epsilon,H,\Sigma)$-vertical, where $H = A(1 +
\Sigma)^2M^2 B$, with $B \geq 5$, and the verticality is defined relative to $f$
via Definition \ref{vertical}.} Let $Q_{1},Q_{2} \in \mathcal{T}$,
and let $w_{j} = (y_{j},t_{j}) \in \tfrac{1}{2}HQ_{j}$, $j \in \{1,2\}$, be
points with the properties $y_{1} < y_{2}$ and
\begin{equation}\label{form26} |y_{1} - y_{2}| > \max\{ \tfrac{1}{100}
\min\{\ell(Q_1),\ell(Q_2)\},4N^{2}\|t_{1} - t_{2}\|\}, \end{equation}
where $N = A(1 + \Sigma)M$. In particular, $d_{\mathrm{par}}(w_{1},w_{2}) = |y_{1} - y_{2}|$. Then,
\begin{equation}\label{form27} \tfrac{1}{2N}d_{\mathrm{par}}(w_{1},w_{2}) \leq |f_{2}(w_{2}) - f_{2}(w_{1})| \leq Md_{\mathrm{par}}(w_{1},w_{2}). \end{equation}
\end{proposition}

\begin{proof} The upper bound in \eqref{form27} only uses the fact that $f$ is $M$-Lipschitz, so it suffices to prove the lower bound. We start by choosing $Q\in \mathcal{T}$ such that
\begin{equation}\label{eq:GoalB}
w_1,w_2\in HQ\quad\text{and}\quad \tfrac{1}{100}\ell(Q)\leq |y_1-y_2|< \tfrac{2}{100} \ell(Q).
\end{equation}
 To this
end, let $i\in \{1,2\}$ be such that $
\min\{\ell(Q_1),\ell(Q_2)\}=\ell(Q_i)$ and let $Q\in \mathcal{T}$
be the largest rectangle containing $Q_i$ with the property 
$\tfrac{1}{100} \ell(Q)\leq |y_1-y_2|$. Such $Q$ exists since $\tfrac{1}{100} \ell(Q_i)\leq
|y_1-y_2|\leq \ell(Q(\mathcal{T}))$. Then $w_{i} \in \tfrac{1}{2}HQ_{i} \subset \tfrac{1}{2}HQ$, and $d_{\mathrm{par}}(w_{1},w_{2}) = |y_{1} - y_{2}| \leq \tfrac{2}{100}\ell(Q) \leq \tfrac{1}{2}\diam(HQ)$, so in fact both $w_{1},w_{2} \in HQ$, as desired.  

Next we consider the horizontal $(N,\epsilon \ell(Q))$-QIE $F_{Q}
\colon HQ \to \W$. We apply Lemma \ref{lemma7} to the map $F_Q$ and
the rectangle $Q$. The main hypothesis \eqref{eq:y_2-y_1_ass} of
the lemma is valid by \eqref{form26} and \eqref{eq:GoalB}, noting that $8N
\epsilon \ell(Q) \leq |y_{1} - y_{2}|$ as long as $8N\epsilon < \tfrac{1}{100}$, and this follows from the choices made at \eqref{form22}. Therefore,
Lemma \ref{lemma7} implies
\begin{displaymath} |\pi_{1}(F_{Q}(w_{2})) - \pi_{1}(F_{Q}(w_{1}))| \geq \frac{|y_{1} - y_{2}|}{2N}. \end{displaymath}
To conclude the proof of \eqref{form27}, we just need to note that
$\pi_{1}(F_{Q}(w)) = f_{2}(w)$ for all $w \in HQ$; this follows from see \eqref{form45}.  \end{proof}

\section{Big projections}\label{s:projections}

In this section, we show that bilipschitz images of $\W$ inside $\He$ have \emph{big vertical projections} (BVP) in the following sense: if $f \colon \W \to \He$ is an $M$-bilipschitz map, $M \geq 1$, and $w \in \W$, $r > 0$, then there exists $\theta \in [-\tfrac{\pi}{2},\tfrac{\pi}{2})$, and a constant $\delta > 0$ depending only on $M$, such that
\begin{displaymath} \calH^{3}[\Pi_{\theta}(f(B(w,r))]) \geq \delta r^{3} \sim \mathcal{H}^{3}(B(w,r)). \end{displaymath}
Here $\Pi_{\theta}$ is the \emph{vertical projection} to the subgroup $\W_{\theta} := R_{\theta}\W$, where $R_{\theta}(z,t) := (e^{i\theta}z,t)$ is a rotation around the $t$-axis by angle $\theta$. In particular, $\Pi_{0}(x,y,t) = \Pi(x,y,t) = (y,t + \tfrac{1}{2}xy)$ and $\W_{0} = \W$. We will check in the next remark that $\Pi_{\theta}$ has the explicit expression
\begin{equation}\label{form130} \Pi_{\theta} = R_{\theta} \circ \Pi \circ R_{\theta}^{-1}. \end{equation}
\begin{remark} We recap standard definition of \emph{vertical projection to $\W_{\theta}$}: for $\theta \in [-\tfrac{\pi}{2},\tfrac{\pi}{2})$ fixed, every point $p \in \He$ can be decomposed uniquely as $p = w_{\theta} \cdot v_{\theta}$, where $w_{\theta} \in \W_{\theta}$ and $v_{\theta} \in \mathbb{L}_{\theta} = \{(e^{i\theta}x,0,0) : x \in \R\}$. This gives rise to the vertical projection $p \mapsto \Pi_{\theta}(p) := w_{\theta}$. To see that $\Pi_{\theta}$ has the expression \eqref{form130}, one observes that $R_{\theta}^{-1} \colon \He \to \He$ is a group homomorphism, so $R_{\theta}^{-1}p = R_{\theta}^{-1}(w_{\theta} \cdot v_{\theta}) = (R_{\theta}^{-1}w_{\theta}) \cdot (R_{\theta}^{-1}v_{\theta})$. Since $R_{\theta}^{-1}w_{\theta} \in \W$, this yields $\Pi(R^{-1}_{\theta}(p)) = R_{\theta}^{-1}w_{\theta} = R_{\theta}^{-1}(\Pi_{\theta}(p))$, hence $\Pi_{\theta} = R_{\theta} \circ \Pi \circ \R_{\theta}^{-1}$. \end{remark}

Before heading to the results of the section, let us generalise the notion of slope (introduced in Section \ref{s:ruledRectangle}) as follows: if $L$ is a horizontal line, and $\W_{\theta}$ is a vertical subgroup, $\theta \in [-\tfrac{\pi}{2},\tfrac{\pi}{2})$, we define the \emph{slope of $L$ relative to $\W_{\theta}$} as
\begin{displaymath} \angle(L,\W_{\theta}) := \angle(R_{\theta}^{-1}L). \end{displaymath}
Similarly, for vertical planes $\V \subset \He$, we define $\angle(\V,\W_{\theta}) := \angle(R^{-1}_{\theta}\V)$.

Here is the main result of this section:

\begin{proposition}\label{prop:bvp} Rickman rugs in $\He$ have BVP. \end{proposition}

We will prove Proposition \ref{prop:bvp} with the aid of the following more technical statement:
\begin{proposition}\label{projProp} For every $M \geq 1$, there exist $H \geq 1$ and $\epsilon > 0$ such that the following holds. Let $f \colon \W \to \He$ be an $M$-bilipschitz map, let $\theta \in [-\tfrac{\pi}{2},\tfrac{\pi}{2})$, and let $Q \in \mathcal{D}$ be a rectangle such that $\beta_{f}(\V;HQ) < \epsilon/H$ for some vertical subgroup $\V$ with $\angle(\V,\W_{\theta}) = 0$. Then
\begin{equation}\label{form117} \calH^{3}(\Pi_{\theta}[f(HQ)]) \geq \ell(Q)^{3}. \end{equation} 
\end{proposition}

The proof of Proposition \ref{projProp} uses winding numbers, so we briefly recall them. There exists a $\Z$-valued function $\gamma \mapsto \mathfrak{w}(\gamma,0)$ defined on the set of loops (that is: closed paths) $\gamma \colon [0,1] \to \R^{2} \, \setminus \, \{0\}$ with the following properties:
\begin{itemize}
\item If $\gamma(t) = e^{2\pi i nt}$, $n \in \Z$, then $\mathfrak{w}(\gamma,0) = n$.
\item If $\gamma_{1}$ and $\gamma_{2}$ are homotopic in $\R^{2} \, \setminus \, \{0\}$, then $\mathfrak{w}(\gamma_{1},0) = \mathfrak{w}(\gamma_{2},0)$.
\end{itemize}
The value $\mathfrak{w}(\gamma,0)$ is the \emph{winding number of $\gamma$ with respect to $0$}. The winding number with respect to an arbitrary point $z \in \R^{2}$ can be defined by $\mathfrak{w}(\gamma,z) = \mathfrak{w}(\gamma - z,0)$, whenever $\gamma$ is a loop in $\R^{2} \, \setminus \, \{z\}$. For us, winding numbers will be useful via the following lemma:
\begin{lemma}\label{windingLemma} Let $Q = [0,a] \times [0,b] \subset \R^{2}$ be a closed rectangle, and let $f \colon Q \to \R^{2}$ be a continuous map. Write $\gamma := f(\partial Q)$, where $\partial Q \colon [0,1] \to \R^{2}$ is parametrised as a loop in any such way that $\partial Q(0) = 0 = \partial Q(1)$. If $z \in \R^{2} \, \setminus \, \gamma$ is a point with $\mathfrak{w}(\gamma,z) \neq 0$, then $z \in f(Q)$. \end{lemma}

\begin{proof} Consider the homotopy $H(s,t) := f((1 - s)\partial Q(t))$, $(s,t) \in [0,1]^{2}$, which deforms $\gamma$ to the constant path $\gamma_{0}(t) \equiv f(0)$. Note that $H$ is well-defined, because $(1 - s)\partial Q(t) \in [0,(1 - s)a] \times [0,(1 - s)b] \subset Q$ for all $(s,t) \in [0,1]^{2}$. If $z \notin f(Q)$, then $\mathrm{Im\,}H \subset \R^{2} \, \setminus \, \{z\}$, so $\gamma$ and $\gamma_{0}$ are homotopic in $\R^{2} \, \setminus \, \{z\}$. Hence $\mathfrak{m}(\gamma,z) = \mathfrak{m}(\gamma_{0},z) = 0$, contrary to the assumption. \end{proof}

We are then armed to prove Proposition \ref{projProp}. Let the reader be warned that the letter "$H$" will refer both to a constant, and occasionally a homotopy; the correct interpretation should always be clear from the context. 

\begin{proof}[Proof of Proposition \ref{projProp}] First, we may assume that $\theta = 0$, since 
\begin{displaymath} \mathcal{H}^{3}(\Pi_{\theta}[f(HQ)]) \stackrel{\eqref{form130}}{=} \mathcal{H}^{3}(R_{\theta}(\Pi[(R_{\theta}^{-1} \circ f)(HQ)])) = \mathcal{H}^{3}(\Pi[(R_{\theta}^{-1} \circ f)(HQ)]), \end{displaymath}
and $f_{\theta} := R_{\theta}^{-1} \circ f$ is an $M$-bilipschitz map with the property $\beta_{f_{\theta}}(R_{\theta}^{-1}(\V),HQ) < \epsilon/H$, where $\angle(R_{\theta}^{-1}(\V)) = \angle(\V,\W_{\theta}) = 0$.

So, assuming $\theta = 0$, there exists a vertical plane $\V \subset \He$ with $\angle(\V) = 0$ such that $\beta_{f}(\V; HQ) \leq \epsilon/H$. In other words, for every horizontal line $\ell \subset \W$, there exists a horizontal line $L \subset \V$ such that 
\begin{equation}\label{form118} f(\ell \cap HQ) \subset N(L,\epsilon\ell(Q)). \end{equation}
By pre- and post-composing $f$ with left translations and dilations, it is easy to reduce the proof to the following situation:
\begin{itemize}
\item $\ell(Q) = 1$ and $HQ = [0,H) \times [0,H^{2})$,
\item $f(0) = 0$.
\end{itemize}
The second point in particular implies, by \eqref{form118}, that $\dist(\W,\V) \leq \dist(f(0),\V) \leq \epsilon$. Therefore, if we replace "$\V$" by "$\W$", the inclusion \eqref{form118} continues to hold with "$2\epsilon$" in place of "$\epsilon$". So, we may additionally assume that
\begin{itemize}
\item $\V = \W = \{(0,y,t) : y,t \in \R\}$.
\end{itemize}
Even though the $(\epsilon,H,0)$-verticality assumption would control the behaviour of $f$ on $HQ$, it is useful to restrict attention to a slightly smaller rectangle, namely
\begin{displaymath} \hat{Q} := [0,H] \times [0,1] \subset \overline{HQ}. \end{displaymath} 
For $j \in \{0,1\}$, let $h_{j} := [0,H] \times \{j\}$ and $v_{j} := \{jH\} \times [0,1]$ be the horizontal and vertical edges of $\hat{Q}$, respectively, and let $\partial \hat{Q} = h_{0}v_{1}h_{1}v_{0}$ be the loop which travels clockwise around $\partial \hat{Q}$, starting and ending at $0$. We parametrise $\partial \hat{Q}$ by the interval $[0,2H + 2]$, so the four corners of $\hat{Q}$ are $\mathcal{C}(\hat{Q}) = \{\partial \hat{Q}(0) = \partial \hat{Q}(2H + 2),\partial \hat{Q}(H),\partial \hat{Q}(H + 1),\partial \hat{Q}(2H + 1)\}$. Then $\pi := \Pi \circ f \circ \partial \hat{Q} \colon [0,2H + 2] \to \W$ is a loop which starts and ends at $\Pi \circ f(0) = 0$.\footnote{The "coordinate projections" $\pi_{1}(y,t) = y$ and $\pi_{2}(y,t) = t$ of $\W$ will not be used in this section, so the notation should not cause confusion.} We plan to show that $\pi$ winds (positively or negatively) around all the points in a large rectangle $R \subset \W$, and then infer from Lemma \ref{windingLemma} that $R \subset \Pi(f(\hat{Q}))$.

We first decompose $\pi = \pi_{1}\pi_{2}\pi_{3}\pi_{4}$ in the natural way, where
\begin{equation}\label{form122} \begin{cases} \pi_{1} = \Pi \circ f \circ h_{0} \colon [0,H] \to \W, \\ \pi_{2} = \Pi \circ f \circ v_{1} \colon [H,H + 1] \to \W, \\ \pi_{3} = \Pi \circ f \circ h_{1} \colon [H + 1,2H + 1] \to \W, \\ \pi_{4} = \Pi \circ f \circ v_{0} \colon [2H + 1,2H + 2] \to \W, \end{cases}  \end{equation} 
and
\begin{equation}\label{form123} \pi_{i}(t) = \pi_{j}(t), \qquad t \in \mathrm{dom\,} \pi_{i} \cap \mathrm{dom\, } \pi_{j}, \, 1 \leq i,j \leq 4. \end{equation}
In the sequel, we abbreviate $I_{j} := \mathrm{dom\,} \pi_{j}$. We will fit $\pi$ (or rather its trace) inside a union of four slabs, that is: neighbourhoods of lines, $S_{1},S_{2},S_{3},S_{4} \subset \W$, two of which are horizontal, and two vertical, see Figure \ref{fig6}.
\begin{figure}[h!]
\begin{center}
\begin{overpic}[scale = 1]{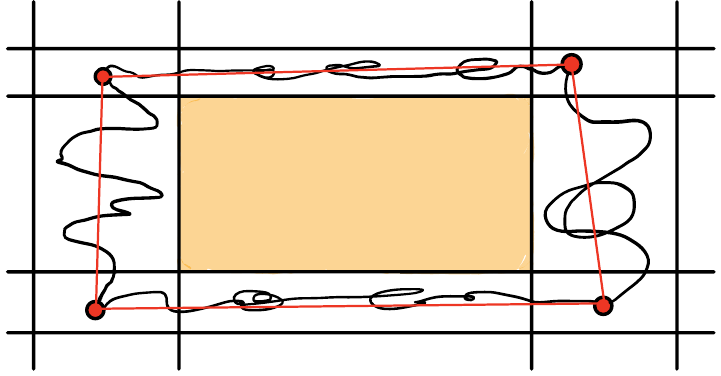}
\put(-2,25){$S_{4}$}
\put(48,0){$S_{1}$}
\put(95,25){$S_{2}$}
\put(48,48){$S_{3}$}
\put(48,25){$R$}
\end{overpic}
\caption{The black wiggly loop is $\pi$, and the red segments form $\eta \cong \pi$.}\label{fig6}
\end{center}
\end{figure}
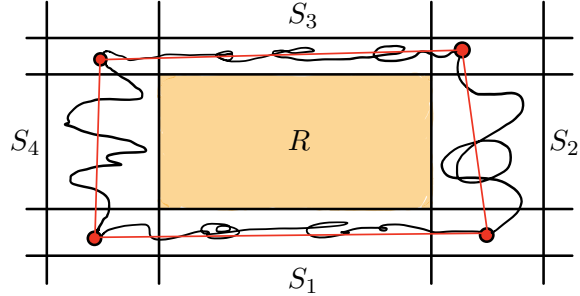
The horizontal slabs $S_{1},S_{3}$ containing $\pi_{1},\pi_{3}$ are neighbourhoods of the horizontal lines (say) $L_{1},L_{3} \subset \W$ which best approximate $f(h_{1})$ and $f(h_{2})$, respectively. Since 
\begin{displaymath} f(h_{j}) \subset N(L_{j},\epsilon) \subset N(\W,\epsilon)  \end{displaymath}
it is not hard to check that
\begin{displaymath} \pi_{j} = \Pi(f(h_{j})) \subset N(L_{j},A\epsilon), \qquad j \in \{1,3\}, \end{displaymath}
for some absolute constant $A \geq 1$, and we set $S_{j} := N(L_{j},A\epsilon)$ for $j \in \{1,3\}$. (For this argument, "$o(\epsilon)$" would work just as well as "$A\epsilon$".) On the other hand, since $f$ is $M$-bilipschitz, the lines $L_{1},L_{3} \subset \W$ are reasonably well separated: 
\begin{displaymath} \dist(L_{1},L_{3}) \geq \dist(f(0,0),f(0,1)) - 2\epsilon \geq \tfrac{1}{M} \cdot d_{\mathrm{par}}((0,0),(0,1)) - 2\epsilon \geq \tfrac{1}{2M},  \end{displaymath}  
choosing $\epsilon < 1/(2M)$. Choosing also $\epsilon \leq 1/(8AM)$, the slabs $S_{1}$ and $S_{3}$ inherit the separation of their central lines:
\begin{equation}\label{form120} \dist(S_{1},S_{3}) \geq \tfrac{1}{4M}. \end{equation}
Figure \ref{fig6} depicts the scenario where $S_{3}$ is "above" $S_{1}$, but it could also happen that $S_{3}$ is "below" $S_{1}$. To define the vertical slabs containing $\pi_{2},\pi_{4}$, first recall that the second coordinate map $w \mapsto f_{2}(w)$ is $M$-Lipschitz. It follows that 
\begin{displaymath} f_{2}(v_{j}) \subset [f_{2}(0,jH) - M, f_{2}(0,jH) + M], \qquad j \in \{0,1\}. \end{displaymath} 
Since $f_{2}$ coincides with the second coordinate of $\Pi \circ f$, we infer that $\pi_{2} = \Pi(f(v_{1}))$ and $\pi_{4} = \Pi(f(v_{0}))$ are contained in vertical slabs $S_{2}$ and $S_{4}$ of width $2M$ around the vertical lines $V_{2} := \{f_{2}(H,0)\} \times \R$ and $V_{4} := \{f_{2}(0,0)\} \times \R$, respectively.

Just like $S_{1},S_{3}$, the vertical slabs $S_{2}$ and $S_{4}$ are also well-separated if the parameter $H$ is chosen larger than $20M^{2}$, and $\epsilon$ is smaller than $H/(AM)$. In fact, we claim that
\begin{equation}\label{form119} |f_{2}(H,0) - f_{2}(0,0)| \geq H/(2M). \end{equation}
To see this, note that $d(f(H,0),f(0,0)) \geq H/M$ by the $M$-bilipschitz property of $f$, and on the other hand $\dist(f(0,0),L_{0}) \leq \epsilon$ and $\dist(f(H,0),L_{0}) \leq \epsilon$. Then, use Lemma \ref{l:Proj_Same_y} to deduce that $f(0,0)$ and $f(H,0)$ lie at distance $\lesssim \epsilon$ from the points $p_{0},p_{H} \in L_{0}$ whose second coordinates match those of $f(0,0)$ and $f(H,0)$. The lower bound \eqref{form119} follows by estimating $|f_{2}(H,0) - f_{2}(0,0)| = d(p_{0},p_{H}) \geq d(f(H,0),f(0,0)) - 2A\epsilon \geq H/(2M)$. 

Taking $H := 20M^{2}$, it follows from the definitions of $S_{2},S_{4}$, and the lower bound \eqref{form119}, that
\begin{equation}\label{form121} \dist(S_{2},S_{4}) \geq H/(4M). \end{equation}
Of course $S_{4}$ can either be to the "left" or "right" from $S_{2}$; Figure \ref{fig6} shows the latter scenario. Combining \eqref{form120} and \eqref{form121}, we deduce that the area of the open rectangle "$R$" bounded by the four slabs $S_{1},S_{2},S_{3},S_{4}$ is $\geq H/(16M^{2}) \geq 1$. So, to conclude the proof of \eqref{form117}, hence the proposition, it remains to show that $R \subset \Pi(f(\hat{Q}))$.

By Lemma \ref{windingLemma}, it suffices to show that $\mathfrak{w}(\pi,z) \neq 0$ for all $z \in R$. This hopefully seems quite obvious: $R$ is "enclosed" by the union $S_{1} \cup \ldots \cup S_{4}$, and $\pi = \pi_{1}\pi_{2}\pi_{3}\pi_{4}$ is a loop with $\pi_{j} \subset S_{j}$ for all $1 \leq j \leq 4$. The details are not much harder: we will deform $\pi$ inside $S_{1} \cup \ldots \cup S_{4}$ into a polygonal loop $\eta = \eta_{1}\eta_{2}\eta_{3}\eta_{4}$ with $4$ linearly parametrised edges $\eta_{j} = [p_{j},p_{j + 1}] \subset S_{j}$. Finally, we observe that if $z \in R$, then $z$ in particular lies in the quadrilateral bounded by $\eta$, see Figure \ref{fig6}. It follows that $\mathfrak{w}(\pi,z) = \mathfrak{w}(\eta,z) \in \{-1,1\}$, since $\eta$ is evidently homotopic to either $z + e^{2\pi it}$ or $z + e^{-2\pi i t}$ in $\R^{2} \, \setminus \, \{z\}$. 

Let us fill in the details. We define $\eta_{1},\ldots,\eta_{4}$ by the formulae
\begin{displaymath} \begin{cases} \eta_{1} = [\pi_{1}(0),\pi_{1}(H)]_{I_{1}} \subset S_{1}, \\ \eta_{2} = [\pi_{2}(H),\pi_{2}(H + 1)]_{I_{2}} \subset S_{2}, \\ \eta_{3} = [\pi_{3}(H + 1),\pi_{3}(2H + 1)]_{I_{3}} \subset S_{3}, \\ \eta_{4} = [\pi_{4}(2H + 1),\pi_{4}(2H + 2)]_{I_{4}} \subset S_{4}, \end{cases} \end{displaymath}
where $I_{j} = \mathrm{dom\,} \pi_{j} = \mathrm{dom\,} \eta_{j}$, and $[p_{1},p_{2}]_{[a,b]}$ refers, in general, to the linear parametrisation of the segment $[p_{1},p_{2}] \subset \W$ by the interval $[a,b] \subset \R$, namely
\begin{displaymath} [p_{1},p_{2}]_{[a,b]}(t) := (1 - \tfrac{t - a}{b - a}) \cdot p_{1} + \tfrac{t - a}{b - a} \cdot p_{2}, \qquad t \in [a,b]. \end{displaymath} 
Then, we define the polygonal path $\eta \colon [0,2H + 2] \to S_{1} \cup \ldots \cup S_{4}$ by setting $\eta(t) := \eta_{j}(t)$ for $t \in I_{j}$. This definition is consistent by \eqref{form123}, and $\eta$ is continuous. 

Next, let $H_{j} \colon [0,1] \times I_{j} \to S_{j}$ be the convex homotopy between $\pi_{j}$ and $\eta_{j}$, namely $H_{j}(s,t) := (1 - s)\pi_{j}(t) + s\eta_{j}(t)$ for $(s,t) \in [0,1] \times I_{j}$. Then $H_{j}|_{\{0\} \times I_{j}} = \pi_{j}$, $H_{j}|_{\{1\} \times I_{j}} = \eta_{j}$, and $H_{j}(s,\partial I_{j}) \equiv \pi_{j}(\partial I_{j}) \equiv \eta_{j}(\partial I_{j})$ for all $s \in [0,1]$.

Finally, we define the homotopy $H \colon [0,1] \times [0,2H + 2] \to S_{1} \cup \ldots \cup S_{4}$ as a concatenation of the homotopies $H_{j}$, namely
\begin{displaymath} H(s,t) := H_{j}(s,t), \qquad s \in [0,1], \, t \in I_{j}, \, 1 \leq j \leq 4. \end{displaymath}
This definition is well-posed, and $H$ is continuous, since 
\begin{displaymath} H_{i}(s,t) = H_{j}(s,t), \qquad 1 \leq i,j \leq 3, \, s \in [0,1], \, t \in I_{i} \cap I_{j}, \end{displaymath} 
by \eqref{form123} (for $t \in I_{j} \cap I_{j}$ simply $H_{j}(s,t) = \pi(t) = \eta(t)$ for all $s \in [0,1]$). Since $\mathrm{Im\,} H \subset S_{1} \cup \ldots \cup S_{4} \subset \W \, \setminus \, R$, the winding numbers of $\pi$ and $\eta$ with respect to every point in $R$ coincide. Clearly $\mathfrak{w}(\eta,z) \in \{-1,1\}$ for all $z \in R$, so we conclude that $\mathfrak{w}(\pi,z) \neq 0$ for all $z \in R$. Lemma \ref{windingLemma} now implies that $R \subset \Pi(f(\hat{Q}))$, and the proof of the proposition is complete. \end{proof}

Proposition \ref{projProp} also yields information about the projections of $f(Q)$ instead of $f(HQ)$:
\begin{cor}\label{cor8} For every $M \geq 1$, there exist $\delta,\epsilon > 0$, depending only on $M$, such that that the following holds. Let $f \colon \W \to \He$ be an $M$-bilipschitz map, $M \geq 1$, let $\theta \in [-\tfrac{\pi}{2},\tfrac{\pi}{2})$, and let $Q \in \mathcal{D}$ be a rectangle with $\beta_{f}(\V;Q) < \epsilon$ for some vertical plane $\V \subset \He$ with $\angle(\V,\W_{\theta}) = 0$. Then,
\begin{equation}\label{form131} \mathcal{H}^{3}(\Pi_{\theta}(f(Q))) \geq \delta \mathcal{H}^{3}(Q). \end{equation}
\end{cor}
\begin{proof} Choose a sub-rectangle $Q' \subset Q$ such that $HQ' \subset Q$ and $\ell(Q') \geq \ell(Q)/(AH)$, where $H \sim_{M} 1$ is the parameter from Proposition \ref{projProp}, and $A \geq 1$ is an absolute constant. If $\ell \subset \W$ is a horizontal line, then by the assumption $\beta_{f}(\V;Q) < \epsilon$, there exists a horizontal line $L \subset \V$ such that
\begin{displaymath} f(\ell \cap HQ') \subset f(\ell \cap Q) \subset N(L,\epsilon \ell(Q)) = N(L,AH\epsilon \cdot \ell(Q')). \end{displaymath} 
In other words, $\beta_{f}(\V;HQ') \leq A\epsilon$. Now, if $\epsilon > 0$ is chosen sufficiently small, depending on $H \sim_{M} 1$, Proposition \ref{projProp} shows that
\begin{displaymath} \mathcal{H}^{3}(\Pi_{\theta}[f(Q)]) \geq \mathcal{H}^{3}(\Pi_{\theta}[f(HQ')]) \geq \ell(Q')^{3} \sim_{M} \ell(Q)^{3}. \end{displaymath}
Thus \eqref{form131} is satisfied for $\delta = \delta(M) > 0$ small enough. \end{proof}

We then complete the proof of Proposition \ref{prop:bvp}.
\begin{proof}[Proof of Proposition \ref{prop:bvp}] The BVP property clearly follows from its dyadic variant: for every $Q \in \mathcal{D}$, there exists $\theta \in [-\tfrac{\pi}{2},\tfrac{\pi}{2})$ such that $\mathcal{H}^{3}(\Pi_{\theta}(f(Q))) \gtrsim_{M} \ell(Q)^{3}$. To prove this, fix $Q \in \mathcal{D}$, and let $Q' \subset Q$ be the largest (or one of the largest) sub-rectangle(s) of $Q$ with $\beta_{f}(Q') < \epsilon$, where $\epsilon = \epsilon(M) > 0$ is the constant from Corollary \ref{cor8}. Since the rectangles $Q' \subset Q$ with $\beta_{f}(Q') \geq \epsilon$ satisfy a Carleson packing condition by Corollary \ref{wgl}, we have $\ell(Q') \sim_{M} \ell(Q)$. Now, Corollary \ref{cor8} yields $\theta \in [-\pi,\pi)$ such that
\begin{displaymath} \mathcal{H}^{3}(\Pi_{\theta}[f(Q)]) \geq \mathcal{H}^{3}(\Pi_{\theta}[f(Q')]) \geq \delta \ell(Q')^{3} \sim_{M} \ell(Q)^{3}, \end{displaymath}
as desired. \end{proof}

%%%%%%%%%%%%%%%%%%%%%%%%%%

\section{Construction of trees}\label{s:treeSelection}

In this section, we construct the decomposition $\mathcal{D} = \mathcal{B} \cup \mathcal{G}$ and the trees $\mathcal{T}_{1},\mathcal{T}_{2},\ldots$ which are required by the corona decomposition, Theorem \ref{main}, recall also Section \ref{s:corona}. This is not too difficult, and the most involved part of the paper only begins in the next section, where each tree $\mathcal{T}_{j}$ is associated with an intrinsic bilipschitz graph $\Gamma_{j} \subset \He$.

Let $f \colon \W \to \He$ be an $M$-bilipschitz map, $M \geq 1$. The construction of $(\mathcal{B},\mathcal{G},\mathcal{T}_{j})$ is very similar to the one described in \cite[\S7]{DS1}: for the reader familiar with \cite{DS1}, the only significant difference lies in checking the Carleson packing condition for the top rectangles, see \nref{D4} below. As in \cite{DS1}, this task requires controlling the total "turning" of certain best-approximating planes; whereas in \cite{DS1} the amount of "turning" is controlled by hypothesis on either singular integrals or $\beta$-numbers, the present argument rather relies on the "big projections" property, Proposition \ref{prop:bvp}.

We claim that for every $H \geq 1$, and for every $\epsilon > 0$ which is sufficiently small in a manner depending only on $M$, there exist constants $\delta > 0$ (depending on $M$), $\Sigma \geq 1$ (depending on $\delta$), and a partition $\mathcal{D} = \mathcal{B} \cup \mathcal{G}$ with the following properties:
\begin{itemize}
\item[(D1) \phantomsection\label{D1}] The rectangles in $\mathcal{B}$ satisfy a Carleson packing condition, that is,
\begin{displaymath} \mathop{\sum_{Q \in \mathcal{B}}}_{Q \subset Q_{0}} |Q| \lesssim_{\epsilon,H,M} |Q_{0}|, \qquad Q_{0} \in \mathcal{D}. \end{displaymath}
\item[(D2) \phantomsection\label{D2}] The rectangles in $\mathcal{G}$ are partitioned into countably many trees "$\mathcal{T}$" with top rectangles $Q(\mathcal{T})$. For each tree $\mathcal{T}$, there exists a vertical subgroup $\W_{\mathcal{T}} = \W_{\theta(\mathcal{T})} \subset \He$ such that 
\begin{equation}\label{form142} \mathcal{H}^{3}(\Pi_{\theta(\mathcal{T})}[f(Q(\mathcal{T}))]) \geq  \delta \mathcal{H}^{3}(Q(\mathcal{T})). \end{equation}
Further, for every $Q \in \mathcal{T}$, there exists a vertical plane $\V_{Q} \subset \He$ with $\angle(\V_{Q},\W_{\mathcal{T}}) \leq \Sigma$ with the property $\beta_{f}(\V;HQ) < H^{-1}A^{-2}(1 + \Sigma)^{-2}\epsilon$. If $\W_{\mathcal{T}} = \W$, these properties imply that $\mathcal{T}$ is an $(\epsilon,H,\Sigma)$-vertical tree.
\item[(D3) \phantomsection\label{D3}] $Q \in \mathbf{Leaves}(\mathcal{T})$ if and only if either $\angle(\V_{Q},\W_{\mathcal{T}}) \geq \Sigma/2$, or then at least one child $Q' \in \mathbf{ch}(Q)$ has $\beta_{f}(HQ') \geq H^{-1}A^{-2}(1 + \Sigma)^{-2}\epsilon$.
\item[(D4) \phantomsection\label{D4}] The top rectangles of the trees $\mathcal{T}_{1},\mathcal{T}_{2},\ldots$ satisfy a Carleson packing condition with constants depending only on $\epsilon,H,M$.
\end{itemize}
The conditions \nref{D1} and \nref{D4} are exactly the Carleson packing conditions required by \nref{C1} and \nref{C3} in the definition of corona decompositions -- at least if eventually $\epsilon,H \sim_{M,\eta} 1$, and this will be the case. The condition \nref{D3} is merely an auxiliary condition which helps verify \nref{D4} below. The condition \nref{D2} is the one to remember: it contains the properties of individual trees which are needed, starting in the next section, to build the $\eta$-approximating intrinsic bilipschitz graphs required by condition \nref{C2}. In this construction, we will want to choose $H \geq 1$ large enough, and $\epsilon > 0$ small enough, in a manner depending on $M,\eta$, so the decompositions $\mathcal{D} = \mathcal{B} \cup \mathcal{G}$ and $\mathcal{G} = \mathcal{T}_{1} \cup \mathcal{T}_{2} \cup \ldots$ will have to be performed with such parameters.

It will soon be apparent that $M$ determines $\delta$, then a small "$\delta$" forces us to choose "$\Sigma$" large, and eventually small "$\delta$" and large "$\Sigma$" force us to choose "$\epsilon$" small. So, the decomposition $\mathcal{D} = \mathcal{B} \cup \mathcal{G}$ can be performed with any $\epsilon > 0$ sufficiently small to meet these requirements, which eventually just depend on $M$.

Before attempting a "global" decomposition of $\mathcal{D} = \mathcal{D} \cup \mathcal{B}$, we perform a localised version. Fix $\bar{Q} \in \mathcal{D}$ arbitrary, and write $\mathcal{D}(\bar{Q}) := \{Q \in \mathcal{D} : Q \subset \bar{Q}\}$. We claim that $\mathcal{D}(\bar{Q})$ can be partitioned to collections $\mathcal{B}(\bar{Q})$ and $\mathcal{G}(\bar{Q})$ with the properties \nref{D1}-\nref{D4}.

We simply let $\mathcal{B}(\bar{Q})$ consist of all the rectangles $Q \in \mathcal{D}(\bar{Q})$ with 
\begin{equation}\label{form135} \beta_{f}(HQ) \geq H^{-1}A^{-2}(1 + \Sigma)^{-2}\epsilon. \end{equation}
Let $\mathcal{G}(\bar{Q}) := \mathcal{D}(\bar{Q}) \, \setminus \, \mathcal{B}(\bar{Q})$. For $Q \in \mathcal{G}(\bar{Q})$, let $\V_{Q}$ be the best approximating plane for $f$ in $HQ$, so  $\beta_{f}(\V_{Q};HQ) < H^{-1}A^{-1}(1 + \Sigma)^{-2}\epsilon$. In particular, $\beta_{f}(\V;Q) < \epsilon$, so if $\theta \in [-\tfrac{\pi}{2},\tfrac{\pi}{2})$ is the unique parameter with $\angle(\V_{Q},\W_{\theta}) = 0$, and $\epsilon,\delta > 0$ are small enough, depending only on $M$, Corollary \ref{cor8} implies $\mathcal{H}^{3}(\Pi_{\theta}[f(Q))] \geq \delta \calH^{3}(Q)$. We also define $\W_{Q} := \W_{\theta}$ for $Q \in \mathcal{G}(\bar{Q})$, so
\begin{equation}\label{form136} \angle(\V_{Q},\W_{Q}) = 0, \qquad Q \in \mathcal{G}(\bar{Q}). \end{equation}
We now perform the tree decomposition of $\mathcal{G}(\bar{Q})$ in the only possible way. Pick some maximal rectangle $Q \in \mathcal{G}(\bar{Q})$ which has not yet been assigned to any tree. If either the $\mathcal{D}$-parent "$\hat{Q}$" of $Q$ lies in $\mathcal{D} \, \setminus \, \mathcal{G}(\bar{Q})$, or $\hat{Q} \in \mathbf{Leaves}(\hat{\mathcal{T}})$ for some existing tree $\hat{\mathcal{T}}$, then $Q$ becomes the top rectangle of a new tree $\mathcal{T}$, and we set $\W_{\mathcal{T}} := \W_{Q} = \W_{Q(\mathcal{T})}$. Otherwise, $\hat{Q} \in \hat{\mathcal{T}} \, \setminus \, \mathbf{Leaves}(\hat{\mathcal{T}})$ for some existing tree $\hat{\mathcal{T}}$, and we add $Q$ to $\hat{\mathcal{T}}$. Now that $Q$ has been assigned to some tree $\mathcal{T}$, we declare when, exactly, $Q$ is a minimal rectangle in $\mathcal{T}$:
\begin{displaymath} Q \in \mathbf{Leaves}(\mathcal{T}) \quad \Longleftrightarrow \quad \mathbf{ch}(Q) \cap \mathcal{B}(\bar{Q}) \neq \emptyset \quad \text{or} \quad \angle(\V_{Q},\W_{\mathcal{T}}) \geq \Sigma/2. \end{displaymath}
This completes the construction of the trees. Note that, for every tree $\mathcal{T}$, the family $\mathbf{Leaves}(\mathcal{T})$ is a union of the collections
\begin{displaymath} \mathbf{Leaves}_{1}(\mathcal{T}) := \{Q \in \mathbf{Leaves}(\mathcal{T}) : \mathbf{ch}(Q) \cap \mathcal{B}(\bar{Q}) \neq \emptyset\} \end{displaymath}
and
\begin{displaymath} \mathbf{Leaves}_{2}(\mathcal{T}) := \{Q \in \mathbf{Leaves}(\mathcal{T}) : \angle(\V_{Q},\W_{\mathcal{T}}) \geq \Sigma/2\}. \end{displaymath}
Let us mention that, while it is possible that $Q(\mathcal{T}) \in \mathbf{Leaves}_{1}(\mathcal{T})$, we always have $Q(\mathcal{T}) \in \mathcal{T} \, \setminus \, \mathbf{Leaves}_{2}(\mathcal{T})$, because $\angle(\V_{Q(\mathcal{T})},\W_{\mathcal{T}}) = 0$ by \eqref{form136} and the choice $\W_{\mathcal{T}} := \W_{Q(\mathcal{T})}$.

Which conditions among \nref{D1}-\nref{D4} are satisfied? The Carleson packing condition for $\mathcal{B}(\bar{Q})$ required in \nref{D1} is satisfied by Corollary \ref{wgl} (also because $\Sigma \sim_{M} 1$, as we will soon see). The property \nref{D3} of $\mathbf{Leaves}(\mathcal{T})$ is evidently satisfied by construction. The projection condition \eqref{form142} and the strong $\beta$-number bound in \nref{D2} are also satisfied by construction (and because $\mathcal{T} \subset \mathcal{G}(\bar{Q})$), so the only part of \nref{D2} in doubt is the slope bound $\angle(\V_{Q},\W_{\mathcal{T}}) \leq \Sigma$ for $Q \in \mathcal{T}$. We handle this in the next lemma before turning to the most difficult property \nref{D4}.
\begin{lemma} Let $\mathcal{T}$ be one of the trees constructed above. If the constant $\epsilon > 0$ is chosen small enough, depending on $M,\Sigma$, then $\angle(\V_{Q},\W_{\mathcal{T}}) \leq \Sigma$ for all $Q \in \mathcal{T}$. \end{lemma}
\begin{proof} The inequality is only in doubt for $Q \in \mathbf{Leaves}_{2}(\mathcal{T})$, so fix $Q \in \mathbf{Leaves}_{2}(\mathcal{T})$. Note that the $\mathcal{D}$-parent $\hat{Q}$ of $Q$ lies in $\mathcal{T} \, \setminus \, \mathbf{Leaves}_{2}(\calT)$, hence satisfies $\angle(\V_{\hat{Q}},\W_{\mathcal{T}}) < \Sigma/2$, because $Q(\mathcal{T}) \notin \mathbf{Leaves}_{2}(\mathcal{T})$.
 
To show that $\angle(\V_{Q},\W_{\mathcal{T}}) \leq \Sigma$, pick two points $w_{1},w_{2} \in Q \subset \hat{Q}$ which lie on a common horizontal line and satisfy $d_{\mathrm{par}}(w_{1},w_{2}) = \ell(Q)$. Since $\beta_{f}(Q) < \epsilon$ and $\beta_{f}(\hat{Q}) < \epsilon$, there exist horizontal lines $L_{Q} \subset \V_{Q}$ and $L_{\hat{Q}} \subset \V_{\hat{Q}}$ with the property
\begin{displaymath} \{f(w_{1}),f(w_{2})\} \subset N(L_{Q},\epsilon \ell(Q)) \cap N(L_{\hat{Q}},2\epsilon \ell(Q)), \end{displaymath}
and moreover $\angle(L_{\hat{Q}},\W_{\mathcal{T}}) =: \angle(\V_{\hat{Q}},\W_{\mathcal{T}}) < \Sigma/2$. So, it now suffices to prove the following geometric statement for $r := 2\ell(Q) > 0$. Assume that $L_{1},L_{2} \subset \He$ are horizontal lines, $\W'$ is a vertical subgroup, and $p_{1},p_{2} \in \He$ are points satisfying
\begin{itemize}
\item[(a)] $\angle(L_{1},\W') < \Sigma/2$,
\item[(b)] $r/M \leq d(p_{1},p_{2}) \leq Mr$,
\item[(c)] $\{p_{1},p_{2}\} \subset N(L_{1},\epsilon r) \cap N(L_{2},\epsilon r)$.
\end{itemize}
Then $\angle(L_{2},\W') \leq \Sigma$ if $\epsilon > 0$ is chosen small enough, depending on $M,\Sigma$. After rotating $\W'$ to $\W$, if necessary, this statement follows directly from Lemma \ref{lemma0} (more precisely the first part about "$|a_{2}| \leq 2\Sigma$"). \end{proof}

We finally onsider condition \nref{D4}. Let $\kappa \sim_{M} \delta \sim_{M} 1$ be a suitable small constant to be determined soon, and classify the trees according to the following three types:
\begin{enumerate}
\item[(i)] Those with $|Q(\mathcal{T}) \, \setminus \, \cup \mathbf{Leaves}(\mathcal{T})| \geq \kappa \cdot |Q(\mathcal{T})|$.
\item[(ii)] Those with $|\cup \mathbf{Leaves}_{1}(\mathcal{T})| \geq \kappa \cdot |Q(\mathcal{T})|$.
\item[(iii)] Those with $|\cup \mathbf{Leaves}_{2}(\mathcal{T})| \geq (1 - 2\kappa) \cdot |Q(\mathcal{T})|$.
\end{enumerate}
Every tree has one of the types (i)-(iii), because
\begin{displaymath} Q(\mathcal{T}) \subset [Q(\mathcal{T}) \, \setminus \, \cup \mathbf{Leaves}(\mathcal{T})] \cup [\cup \mathbf{Leaves}_{1}(\mathcal{T})] \cup [\cup \mathbf{Leaves}_{2}(\mathcal{T})]. \end{displaymath}
We next claim that \emph{trees of type (iii) do not exist} if the "approximation parameter" $\epsilon > 0$ is chosen small enough, and and the "slope parameter" $\Sigma \geq 1$ is chosen large enough, depending only on $\delta \sim_{M} 1$. To see this, let $\mathcal{T}$ be a hypothetical tree of type (iii), and recall from \eqref{form142} that
\begin{equation}\label{form132} \mathcal{H}^{3}(\Pi_{\theta}[f(Q(\mathcal{T}))]) \geq \delta \mathcal{H}^{3}(Q(\mathcal{T})) \end{equation}
for the unique parameter $\theta \in [-\tfrac{\pi}{2},\tfrac{\pi}{2})$ such that $\W_{\mathcal{T}} = \W_{\theta}$. We claim that
\begin{equation}\label{form133} \mathcal{H}^{3}(\Pi_{\theta}[f(Q)]) = o(\epsilon,\Sigma) \cdot \mathcal{H}^{3}(Q), \qquad Q \in \mathbf{Leaves}_{2}(\mathcal{T}), \end{equation}
where $o(\epsilon,\Sigma)$ denotes a constant which tends to $0$ as $\epsilon \to 0$ and $\Sigma \to \infty$. Let us complete the non-existence proof of trees of type (iii), assuming that \eqref{form133} has already been established. In addition to \eqref{form133}, we will need the following estimate:
\begin{align} \mathcal{H}^{3}(\Pi_{\theta}[f(Q(\mathcal{T}) \, \setminus \, \cup \mathbf{Leaves}_{2}(\mathcal{T}))]) & \lesssim \mathcal{H}^{3}(f(Q(\mathcal{T}) \, \setminus \, \cup \mathbf{Leaves}_{2}(\mathcal{T}))) \notag\\
&\label{form134} \lesssim_{M} \kappa \cdot |Q(\mathcal{T})|. \end{align}
The first inequality follows from the well-known fact that vertical projections do not increase $3$-dimensional Hausdorff measure by more than a constant, see \cite[Lemma 3.6]{CFO}. The second inequality follows from the the $M$-bilipschitz property of $f$, and the hypothesis that $\mathcal{T}$ has type (iii). By \eqref{form134}, we may choose $\kappa \sim_{M} \delta$ so small that $\mathcal{H}^{3}(\Pi_{\theta}[f(Q(\mathcal{T}) \, \setminus \, \cup \mathbf{Leaves}_{2}(\mathcal{T}))]) < (\delta/2) \cdot \mathcal{H}^{3}(Q(\mathcal{T}))$. Using this estimate, \eqref{form132}, and the disjointness of the rectangles in $\mathbf{Leaves}(\mathcal{T})$, we have
\begin{align*} \delta \mathcal{H}^{3}(Q(\mathcal{T})) & \leq \mathcal{H}^{3}(\Pi_{\theta}[f(\cup \mathbf{Leaves}_{2}(\mathcal{T}))]) + \mathcal{H}^{3}(\Pi_{\theta}[f(Q(\mathcal{T}) \, \setminus \, \cup \mathbf{Leaves}_{2}(\mathcal{T}))])\\
& \leq \sum_{Q \in \mathbf{Leaves}_{2}(\mathcal{T})} \mathcal{H}^{3}(\Pi_{\theta}[f(Q)]) + \frac{\delta}{2} \cdot |Q(\mathcal{T})|\\
& \leq o(\epsilon,\Sigma) \sum_{Q \in \mathbf{Leaves}(\mathcal{T})} \mathcal{H}^{3}(Q) + \frac{\delta}{2} \cdot |Q(\mathcal{T})| < \delta |Q(\mathcal{T})|, \end{align*} 
and choosing $\epsilon > 0$ so small and $\Sigma \geq 1$ so large, depending on $\delta \sim_{M} 1$, that $o(\epsilon,\Sigma) < \delta/2$. This contradiction will show that trees of type (iii) exist, once \eqref{form133} has been established.

To prove \eqref{form133}, fix $Q \in \mathbf{Leaves}_{2}(\mathcal{T})$. Applying the rotation $R_{\theta}^{-1}$, we may assume $\theta = 0$, so $\Pi_{\theta} = \Pi$, $\W_{\mathcal{T}} = \W$, and $\angle(\V_{Q}) = \angle(\V_{Q},\W_{\mathcal{T}}) \geq \Sigma/2$, where $\V_{Q}$ is a vertical plane with $f(Q) \subset N(\V,\epsilon \ell(Q))$. Of course we also have $f(Q) \subset B(f(c_{Q}),M\ell(Q))$. Putting this information together, \eqref{form133} will follow from a geometric estimate of the form
\begin{displaymath} \mathcal{H}^{3}(\Pi[B(p,r) \cap N(\V,\epsilon r)]) = o(\epsilon,\Sigma) \cdot r^{3}, \qquad \angle(\V) \geq \Sigma/2, \, p \in \He, \, r > 0. \end{displaymath}
The claim is invariant under scaling and left translation, so we may assume that $\V$ is a vertical subgroup, and $r = 1$. Also, since we only need to consider the situation where $B(p,r) \cap N(\V,\epsilon r) \neq \emptyset$, we may assume that $p \in \V$, replacing "$r$" by "$2r$" if necessary. After this reduction, we may apply one more left translation which keeps the subgroup $\V$ fixed, but sends $p$ to $0$; in other words, it suffices to show
\begin{displaymath} \mathcal{H}^{3}(\Pi[B(0,1) \cap N(\V,\epsilon)]) = o(\epsilon,\Sigma), \qquad \angle(\V) \geq \Sigma/2. \end{displaymath}
This is so obvious that we leave further calculations to the reader (if $\epsilon = 0$ and $\angle(\V) = \infty$, the projection $\Pi(B(0,1) \cap \V)$ is just a segment of length $\sim 1$ on the $t$-axis).

We have now proven that trees of type (iii) do not exist. So, to verify the Carleson packing condition for the tops $Q(\mathcal{T})$ in property \nref{D4}, it suffices to consider trees of type (i)-(ii). First, the sets $Q(\mathcal{T}) \, \setminus \, \cup \mathbf{Leaves}(\mathcal{T}) \subset Q(\mathcal{T})$ are disjoint for distinct trees, so 
\begin{displaymath} \mathop{\sum_{\mathcal{T} \text{ type (i)}}}_{Q(\mathcal{T}) \subset Q_{0}} |Q(\mathcal{T})| \leq \frac{1}{\kappa} \mathop{\sum_{\mathcal{T} \text{ type (i)}}}_{Q(\mathcal{T}) \subset Q_{0}} |Q(\mathcal{T}) \, \setminus \, \cup \mathbf{Leaves}(\mathcal{T})| \leq \frac{1}{\kappa} \cdot |Q_{0}| \sim_{M} |Q_{0}|, \quad Q_{0} \in \mathcal{D}. \end{displaymath} 
To handle trees of type (ii), we use the Carleson packing condition for the strong vertical $\beta$-numbers in Corollary \ref{wgl}, recalling that every leaf $Q \in \mathbf{Leaves}_{1}(\mathcal{T})$ has a child $Q' \in \mathbf{ch}(Q)$ with $\beta_{f}(HQ') \gtrsim_{M} \epsilon$:
\begin{displaymath} \mathop{\sum_{\mathcal{T} \text{ type (ii)}}}_{Q(\mathcal{T}) \subset Q_{0}} |Q(\mathcal{T})| \leq \frac{1}{\kappa} \mathop{\sum_{\mathcal{T} \text{ type (ii)}}}_{Q(\mathcal{T}) \subset Q_{0}} \sum_{Q \in \mathbf{Leaves}_{1}(\mathcal{T})} |Q| \lesssim \frac{1}{\kappa} \mathop{\sum_{Q \subset Q_{0}}}_{\beta(HQ) \gtrsim_{M} \epsilon} |Q| \lesssim_{\epsilon,H,M} |Q_{0}|, \quad Q_{0} \in \mathcal{D}. \end{displaymath} 
This completes the proofs of \nref{D1}-\nref{D4} for the decomposition $\mathcal{D}(\bar{Q}) = \mathcal{B}(\bar{Q}) \cup \mathcal{G}(\bar{Q})$, where $\bar{Q} \in \mathcal{D}$ was an arbitrary "starting rectangle".

To infer, from this special case, a decomposition $\mathcal{D} = \mathcal{B} \cup \mathcal{G}$ satisfying properties \nref{D1}-\nref{D4}, one follows \cite[p. 38]{DS1} verbatim. Choose a  sequence of disjoint "starting rectangles" $\bar{Q}_{1},\bar{Q}_{2},\ldots$ with the following property: every rectangle in $\mathcal{D}_{n}$ is contained in $\bigcup \bar{Q}_{j}$ for $n \geq 0$, and for $n < 0$, the same is still true for every rectangle in $\mathcal{D}_{n} \, \setminus \, \mathcal{E}_{n}$, where $\card \mathcal{E}_{n} \lesssim 1$. Such a sequence $\bar{Q}_{1},\bar{Q}_{2}\ldots $ is easy to pick by hand, or see \cite[p. 38]{DS1} for a more abstract construction; the idea is roughly to cover $B(0,1)$ by disjoint rectangles of side-length $1$, then $B(0,2) \, \setminus \, B(0,1)$ by disjoint rectangles of side-length $2$ (without overlapping the unit rectangles), and so on. Once this is done properly, all the rectangles in $Q \in \mathcal{D}_{n}$ with $\dist(Q,0) \gg 2^{-n}$ will be contained in the union $\bigcup \bar{Q}_{j}$, and of course $\card \{Q \in \mathcal{D}_{n} : \dist(Q,0) \lesssim 2^{-n}\} \lesssim 1$.

Now, we define $\mathcal{G} := \cup \mathcal{G}(\bar{Q}_{j})$ and $\mathcal{B} := \mathcal{D} \, \setminus \, \mathcal{G}$. The tree partition of $\mathcal{G}$ is given as the union of the tree partitions of the families $\mathcal{G}(\bar{Q}_{j}) \subset \mathcal{D}(\bar{Q}_{j})$. The conditions \nref{D2}-\nref{D3} concerning individual trees are valid without problems, and the Carleson packing condition \nref{D4} follows from the disjointness of the rectangles $\bar{Q}_{1},\bar{Q}_{2},\ldots$ The Carleson packing condition for $\mathcal{B}$ postulated in \nref{D1} is also valid, because
\begin{displaymath} \mathcal{B} = \bigcup_{j \geq 1} \mathcal{B}(\bar{Q}_{j}) \cup \bigcup_{n \in \Z} \mathcal{E}_{n}. \end{displaymath}
This completes the construction of the trees satisfying conditions \nref{D1}-\nref{D4}.

%%%%%%%%%%%%%%%%%%%%%%%%%%

\section{Constructing intrinsic bilipschitz graphs}\label{s:bilipschitzGraphs}

Recalling the the decomposition $\mathcal{D} = \mathcal{B} \cup \mathcal{G}$ achieved in the previous section, the proof of Theorem \ref{main} is only one component short: to every tree $\mathcal{T} := \mathcal{T}_{j}$, satisfying the properties stated in \nref{D2}, we need to associate an intrinsic bilipschitz graph $\Gamma_{\mathcal{T}}$ which $\eta$-approximates $f$ in the sense specified in \eqref{form143}. Performing a rotation around the $t$-axis if necessary, we may assume that the tree $\mathcal{T}$ is $(\epsilon,H,\Sigma)$-vertical, where $\epsilon,H$ may be chosen freely (as long as $\epsilon > 0$ was taken small enough, depending on $M$), and $\Sigma$ is a parameter determined by $M$. The existence of $\Gamma_{\mathcal{T}}$ is the content of the following proposition:
\begin{proposition}\label{mainProp} For every $M \geq 1$ and $\eta,\Sigma > 0$, there exist constants $\epsilon > 0$ and $H \geq 10$ such that the following holds. Let $f \colon \W \to \He$ be an $M$-bilipschitz map, and let $\mathcal{T} \subset \mathcal{D}$ be an $(\epsilon,H,\Sigma)$-vertical tree associated with $f$. Then, there exists an intrinsic Lipschitz graph $\Gamma_{\mathcal{T}} \subset \He$ over the plane $\W$ such that
\begin{equation}\label{form103} \dist(f(w),\Gamma_{\mathcal{T}}) \leq \eta \cdot \ell(Q), \qquad w \in 2Q, \, Q \in \mathcal{T}. \end{equation} 
The intrinsic Lipschitz constant of $\Gamma_{\mathcal{T}}$ only depends on $M$ and $\Sigma$. The graph $\Gamma_{\mathcal{T}}$ can be chosen so that it admits a $C$-bilipschitz parametrisation by $\W$, where $C$ also only depends on $M$ and $\Sigma$.
\end{proposition}

The intrinsic graph $\Gamma_{\mathcal{T}}$ will be constructed by specifying its \emph{characteristic curves}, but any general theory of these curves is not required to understand the following proposition: 
\begin{proposition}\label{mainProp2} Let $M \geq 1$, $\Sigma,\eta > 0$, $f \colon \W \to \He$, and $\mathcal{T}$ have the same meaning as in Proposition \ref{mainProp}. Then, the following holds if $H,K \geq 1$ are chosen large enough depending on $M,\Sigma,\eta$, and $\epsilon \in (0,1]$ is chosen small enough depending on all the previous parameters. There exists a map $\tau = \tau_{\mathcal{T}} \colon \W \to \R^{2}$ with the following properties.
\begin{itemize}
\item[($\tau_{1}$)] For every $t \in \R$ fixed, $y \mapsto \tau(y,t) \in C^{1,1}(\R)$, and \emph{a fortiori} $\partial_{y} \tau(y,t)$ is $2\Sigma$-Lipschitz:
\begin{displaymath} |\partial_{y} \tau(y_{2},t) - \partial_{y} \tau(y_{2},t)| \leq 2\Sigma |y_{2} - y_{1}|, \qquad y_{1},y_{2} \in \R. \end{displaymath}
\item[($\tau_{2}$)]  There exists a constant $\theta > 0$, depending only on $M,\Sigma$, such that
\begin{displaymath} \theta \cdot |t_{2} - t_{1}| \leq |\tau(y,t_{2}) - \tau(y,t_{1})| \leq \theta^{-1} \cdot |t_{2} - t_{1}|, \qquad t_{1},t_{2},y \in \R. \end{displaymath}
In particular, the graphs $\{(y,\tau(y,t)) : y \in \R\}$, $t \in \R$, foliate $\R^{2}$. 
\item[($\tau_{3}$)] For every $y \in \R$ fixed, the map $t \mapsto \partial_{y} \tau(y,t)$ is $\tfrac{1}{2}$-H\"older continuous with constant $\eta$:
\begin{displaymath} |\partial_{y} \tau(y,t_{2}) - \partial_{y} \tau(y,t_{1})| \leq \eta \sqrt{|t_{2} - t_{1}|}, \qquad t_{1},t_{2} \in \R. \end{displaymath}
\item[($\tau_{4}$)] Let $Q \in \mathcal{T}$ be a rectangle with corner $w := (x,t) \in \mathcal{C}(Q)$, and let $\ell \subset \W$ be the unique horizontal line containing $w$. Let $q := q_{Q,\ell}$ be the $(Q,\ell)$-approximate quadric introduced in Section \ref{s:appQuad}. Then,
\begin{displaymath} |\tau(y,t) - q(y)| \leq \eta \cdot \ell(Q)^{2} \quad \text{and} \quad |\dot{\tau}(y,t) - \dot{q}(y)| \leq \eta \cdot \ell(Q) \end{displaymath}
for all $y \in B(f_{2}(w),K\ell(Q))$, where $f_{2}$ is the second component of $f$.
\end{itemize}
 \end{proposition}
 
\begin{remark}\label{rem3} Property ($\tau_{4}$) captures, roughly speaking, that the values of $\tau$ are dictated by approximate quadrics associated with rectangles $Q \in \mathcal{T}$. This statement encodes the relation between $\tau$ and $f$, because, by Corollary \ref{cor4}, the values of the approximate quadrics evaluated at $f_{2}(w)$ are nearly determined by $f(w)$. The property ($\tau_{4}$) is eventually used to verify the $\eta$-approximation in \eqref{form103}, and for this purpose, we will need to choose $K \geq A(1 + \Sigma)M$ for a sufficiently large absolute constant $A \geq 1$.  \end{remark}

During the remainder of this section, we will reduce the proof of Proposition \ref{mainProp} (hence Theorem \ref{main}) to Proposition \ref{mainProp2}, which we now take for granted. The proof of Proposition \ref{mainProp2} is contained in the next section. To prove Proposition \ref{mainProp}, we define a function $\phi_{\mathcal{T}}$ implicitly by
\begin{displaymath} \phi_{\mathcal{T}}(0,y,\tau(y,t)) := \dot{\tau}(y,t), \qquad (y,t) \in \W, \end{displaymath} 
where $\dot{\tau}(y,t) := \partial_{y}\tau(y,t)$, here and always. The definition of $\phi_{\mathcal{T}}$ is well-posed by properties ($\tau_{1}$)-($\tau_{2}$). We also define the \emph{intrinsic graph map}
\begin{align*} \Phi_{\mathcal{T}}(0,y,\tau(y,t)) & := (0,y,\tau(y,t)) \cdot (\dot{\tau}(y,t),0,0), \qquad (y,t) \in \W, \end{align*}
and the intrinsic graph $\Gamma_{\mathcal{T}} := \{\Phi_{\mathcal{T}}(0,y,t) : (0,y,t) \in \W\}$. 

\begin{remark} I explain, briefly and informally, how the properties ($\tau_{1}$)-($\tau_{4}$) of the map "$\tau$" translate into geometric properties of $\Gamma_{\mathcal{T}}$. The property ($\tau_{1}$) says nothing beyond "$\phi_{\mathcal{T}}$ is an intrinsic Lipschitz function". Property ($\tau_{2}$) is the deepest property, and allows us to find a bilipschitz parametrisation for $\Gamma_{\mathcal{T}}$; more on this very soon. Property $(\tau_{3})$ means that the graph $\Gamma_{\mathcal{T}}$ looks everywhere, and at every scale, like a "vertical Lipschitz flag", up to an error quantified by the parameter $\eta$. Namely, if $\eta = 0$, then $\dot{\tau}(y,t)$ is independent of $t$, and therefore also $(0,y,t) \mapsto \phi_{\mathcal{T}}(0,y,t)$ is independent of $t$. And the intrinsic graph of an intrinsic Lipschitz function, which does not depend on the $t$-variable, has the form $\Gamma_{\R^{2}} \times \R \subset \He$, where $\Gamma_{\R_{2}} \subset \R^{2}$ is a Lipschitz graph. 

Finally, condition ($\tau_{4}$) means that the "Lipschitz flag" mentioned above looks like a vertical plane near all points of the form $\Phi_{\mathcal{T}}(0,f_{2}(w),t)$, where $w = (x,t) \in \mathcal{C}(Q)$ and $Q \in \mathcal{T}$; to see this, try replacing "$\tau$" and "$\dot{\tau}$" by "$q$" and "$\dot{q}$" in the definition of $\Phi_{\mathcal{T}}$, and notice (or check form \eqref{form105}) that the image of $\Phi_{\mathcal{T}}$ is a vertical plane in this case. \end{remark}

The map $\Phi_{\mathcal{T}}$ is the "standard" intrinsic graph map of $\phi_{\mathcal{T}}$, and is almost never a bilipschitz map $\W \to \Gamma_{\mathcal{T}}$, as explained in Section \ref{s:iLip}. To obtain a \textbf{bi}lipschitz parametrisation of $\Gamma_{\mathcal{T}}$, one instead needs to consider $\Psi \colon \W \to \He$ defined by 
\begin{equation}\label{form104} \Psi(y,t) := \Phi_{\mathcal{T}}(0,y,\tau(y,t)), \qquad (y,t) \in \W. \end{equation}
It is not a new idea that the map "$\Psi$" above is a good candidate for bilipschitz parametrising $\Gamma_{\mathcal{T}}$ by $\W$; the same map was used earlier in, at least, \cite{MR2247905} and \cite{MR2603594} in a scenario where $\phi_{\mathcal{T}}$ is assumed \emph{a priori} Euclidean Lipschitz regular. This need not be true in our setting, but condition ($\tau_{2}$) is an adequate substitute.

\begin{proposition}\label{bilipProp} The function $\phi_{\mathcal{T}}$ is an intrinsic Lipschitz function. Further, the map $\Psi$ is a bilipschitz homeomorphism $(\W,d_{\mathrm{par}}) \to \Gamma_{\mathcal{T}}$. The intrinsic Lipschitz, and bilipschitz constants, here, depend only on $\eta,\Sigma,\theta$ in the properties $(\tau_{1})$-$(\tau_{3})$.  \end{proposition} 

Recalling that $\theta \sim_{M,\Sigma} 1$, and choosing $0 < \eta < \min\{\Sigma,\theta\}$, Proposition \ref{bilipProp} implies the claims in Proposition \ref{mainProp} about the intrinsic Lipschitz constant of $\Gamma_{\mathcal{T}}$, and the bilipschitz constant of its $\W$-parametrisation $\Psi$.

\begin{proof}[Proof of Proposition \ref{bilipProp}] Fix $y_{1},y_{2},t_{1},t_{2} \in \R$, and write $p_{j} := \Psi(y_{j},t_{j})$ for $j \in \{1,2\}$. These are two arbitrary points on $\Gamma_{\mathcal{T}}$. We record immediately that $p^{-1} \cdot p_{2} = \Psi(y_{1},t_{1})^{-1} \cdot \Psi(y_{2},t_{2}) = (a,b,c)$, where
\begin{equation}\label{form97} \begin{cases} a = \dot{\tau}(y_{2},t_{2}) - \dot{\tau}(y_{1},t_{1}), \\ b = y_{2} - y_{1}, \\ c = \tau(y_{2},t_{2}) - \tau(y_{1},t_{1}) - \tfrac{1}{2}(y_{2} - y_{1}) \cdot [\dot{\tau}(y_{1},t_{1}) + \dot{\tau}(y_{2},t_{2})]. \end{cases} \end{equation}
To prove that $\phi_{\mathcal{T}}$ is intrinsic Lipschitz, we need to show (recall \eqref{coneCondition2}) that
\begin{equation}\label{form115} |\dot{\tau}(y_{2},t_{2}) - \dot{\tau}(y_{1},t_{1})| \lesssim_{\eta,\Sigma,\theta} \|\Pi(p_{1}^{-1} \cdot p_{2})\|, \end{equation}
and to prove that $\Psi$ is bilipschitz, we need to show that
\begin{equation}\label{form116} d(p_{1},p_{2}) = \|p_{1}^{-1} \cdot p_{2}\| \sim_{\eta,\Sigma,\theta} \max\{|y_{1} - y_{2}|, \sqrt{|t_{1} - t_{2}|}\}. \end{equation}
We start with the latter task. The third component "$c$" in \eqref{form97} simplifies if either $y_{1} = y_{2}$ or $t_{1} = t_{2}$. First, if $y_{1} = y = y_{2}$, then 
\begin{equation}\label{form98} |c| = |\tau(y,t_{2}) - \tau(y,t_{1})| \in [\theta |t_{2} - t_{1}|,\theta^{-1}|t_{2} - t_{1}|] \end{equation}
by condition ($\tau_{2}$). Second, if $t_{1} = t = t_{2}$, then
\begin{equation}\label{form99} |c| = \left| \int_{y_{1}}^{y_{2}} \frac{2\dot{\tau}(y,t) - [\dot{\tau}(y_{1},t) + \dot{\tau}(y_{2},t)]}{2} \, dy \right| \lesssim \Sigma |y_{2} - y_{1}|^{2} \end{equation} 
by condition ($\tau_{1}$). Since moreover, by conditions ($\tau_{1})$ and ($\tau_{3}$),
\begin{displaymath} |a| \leq |\dot{\tau}(y_{2},t_{2}) - \dot{\tau}(y_{1},t_{2})| + |\dot{\tau}(y_{1},t_{2}) - \dot{\tau}(y_{1},t_{1})| \leq 2\Sigma |y_{2} - y_{1}| + \eta \sqrt{|t_{2} - t_{1}|}, \end{displaymath}
we can now use the triangle inequality to conclude from \eqref{form98}-\eqref{form99} that
\begin{displaymath} d(p_{1},p_{2}) \lesssim \max\{\Sigma,\sqrt{\Sigma},\eta,\theta^{-1/2}\} \cdot d_{\mathrm{par}}((y_{1},t_{1}),(y_{2},t_{2})). \end{displaymath} 
Hence $\Psi \colon \W \to \Gamma_{\mathcal{T}}$ is Lipschitz. To prove a lower bound, observe first that
\begin{displaymath} d(\Psi(y_{1},t_{1}),\Psi(y_{2},t_{2})) = d(p_{1},p_{2}) \geq |b| = |y_{2} - y_{1}|. \end{displaymath} 
So, if $\sqrt{|t_{2} - t_{1}|} \lesssim_{\Sigma,\theta} |y_{2} - y_{1}|$, the lower bound is clear. If, on the other hand $|y_{2} - y_{1}| \ll \sqrt{|t_{1} - t_{2}|}$, or more precisely $(1 + 10\max\{\Sigma,\sqrt{\Sigma}\})|y_{2} - y_{1}| \leq \tfrac{1}{2}\sqrt{\theta}\sqrt{|t_{2} - t_{1}|}$, then we use \eqref{form98}-\eqref{form99}, and the upper bound for $|a|$ (with $t_{1} = t_{2}$) to deduce
\begin{align*} d(p_{1},p_{2}) & \geq d(\Psi(y_{1},t_{1}),\Psi(y_{1},t_{2})) - d(\Psi(y_{1},t_{2}),(y_{2},t_{2}))\\
& \geq \sqrt{\theta}\sqrt{|t_{2} - t_{1}|} - (1 + 10\max\{\Sigma,\sqrt{\Sigma}\})|y_{2} - y_{1}|\\
& \geq \tfrac{1}{2}\sqrt{\theta}\sqrt{|t_{2} - t_{1}|} \gtrsim \sqrt{\theta} \cdot d_{\mathrm{par}}((y_{1},t_{1}),(y_{2},t_{2})). \end{align*}  
This completes the proof of the bilipschitz property \eqref{form116}. 

We then turn to the proof of the intrinsic Lipschitz condition \eqref{form115}. First, we spell out
\begin{displaymath} \Pi(p_{1}^{-1} \cdot p_{2}) = (0,y_{2} - y_{1},\tau(y_{2},t_{2}) - \tau(y_{1},t_{1}) - (y_{2} - y_{1}) \dot{\tau}(y_{1},t_{1})). \end{displaymath} 
Denote the third component "$h$", and estimate it from below as follows:
\begin{align*} |h| & \geq |\tau(y_{2},t_{2}) - \tau(y_{2},t_{1})| - |\tau(y_{2},t_{1}) - \tau(y_{1},t_{1}) - (y_{2} - y_{2})\dot{\tau}(y_{1},t_{1})|\\
& \geq \theta \cdot |t_{2} - t_{1}| - \left| \int_{y_{1}}^{y_{2}} \dot{\tau}(y,t_{1}) - \dot{\tau}(y_{1},t_{1}) \, dy \right| \geq \theta \cdot |t_{2} - y_{1}| - 2\Sigma \cdot |y_{2} - y_{1}|^{2}. \end{align*}
We used here properties ($\tau_{1}$)-($\tau_{2}$). Consequently, if $2\Sigma \cdot |y_{1} - y_{2}|^{2} \leq \tfrac{\theta}{2} \cdot |t_{2} - t_{1}|$, then $\|\Pi(p_{1}^{-1} \cdot p_{2})\| \geq \sqrt{|h|} \gtrsim_{\theta} \sqrt{|t_{2} - t_{1}|}$. In the opposite case,
\begin{displaymath} \|\Pi(p^{-1} \cdot p_{2})\| \geq |y_{2} - y_{1}| \gtrsim_{\Sigma,\theta} \sqrt{|t_{2} - t_{1}|}. \end{displaymath}
In both cases, $\|\Pi(p^{-1} \cdot p_{2})\| \gtrsim_{\Sigma,\theta} |y_{2} - y_{1}| + \sqrt{|t_{2} - t_{1}|}$.

Next, we use properties ($\tau_{2}$)-($\tau_{3}$):
\begin{align*} |\dot{\tau}(y_{2},t_{2}) - \dot{\tau}(y_{1},t_{1})| & \leq |\dot{\tau}(y_{2},t_{2}) - \dot{\tau}(y_{1},t_{2})| + |\dot{\tau}(y_{1},t_{2}) - \dot{\tau}(y_{1},t_{1})|\\
& \leq 2\Sigma \cdot |y_{2} - y_{1}| + \eta \cdot \sqrt{|t_{2} - t_{1}|}. \end{align*}
Comparing this upper bound with the lower bound $\|\Pi(p^{-1} \cdot p_{2})\| \gtrsim_{\Sigma,\theta} |y_{2} - y_{1}| + \sqrt{|t_{2} - t_{1}|}$ completes the proof of \eqref{form115}, and the proof of the proposition. \end{proof}

The next proposition shows that the intrinsic Lipschitz graph $\Gamma_{\mathcal{T}}$ satisfies the main requirement of Proposition \ref{mainProp}, namely \eqref{form103}. This will be based on properties ($\tau_{3}$)-($\tau_{4}$).

\begin{proposition}\label{approxProp} Let $Q \in \mathcal{T}$, and assume that the constant $K$ in ($\tau_{4}$) satisfies $K \geq A(1 + \Sigma)M$. Then, $\dist(f(w),\Gamma_{\mathcal{T}}) = o(\eta) \cdot \ell(Q)$ for all $w \in 2Q$.
\end{proposition}

The notation $o(\eta)$ here refers to a constant which can be chosen arbitrarily small by taking $\eta > 0$ in Proposition \ref{mainProp2} sufficiently small, \textbf{depending only on $M$ and $\Sigma$}. Hence, Proposition \ref{approxProp} implies \eqref{form103}, modulo the effort of re-naming small constants.

\begin{proof}[Proof of Proposition \ref{approxProp}] The idea is simple: since $Q \in \mathcal{T}$ is an $(\epsilon,H,\Sigma)$-vertical rectangle, we know that $f(2Q)$ is approximated by the vertical plane $\V := \V_{Q}$ up to an error $\epsilon \ell(Q)$. So, choosing $0 < \epsilon \leq \eta$, it suffices to show that $\V$ is further approximated by $\Gamma_{\mathcal{T}}$ in a sufficiently large neighbourhood of $f(Q)$, up to an error $o(\eta) \cdot \ell(Q)$.

For this purpose, let $w_{0} \in \mathcal{C}(Q)$ be an arbitrary corner. Then, $2Q \subset B(w_{0},5\ell(Q))$, so 
\begin{displaymath} f(2Q) \subset B(f(w_{0}),5M\ell(Q)) \cap N(\mathbb{V},\epsilon \ell(Q)) \subset N(\V \cap B(f(w_{0}),6M\ell(Q)),\epsilon \ell(Q)), \end{displaymath} 
using $M \geq 1 \geq \epsilon$ in the last inclusion. Abbreviate $B(f(w_{0}),6M\ell(Q)) =: B_{0}$. It remains to prove
\begin{equation}\label{form100} \V \cap B_{0} \subset N(\Gamma_{\mathcal{T}},o(\eta) \cdot \ell(Q)), \end{equation} 
because then $f(2Q) \subset N(\Gamma_{\mathcal{T}},o(\eta) \cdot \ell(Q))$.

To prove \eqref{form100}, let $\ell = \{(y,t_{0}) : y \in \R\} \subset \W$ be the unique horizontal line containing $w_{0}$, and let $q := q_{Q,\ell}$ be the $(Q,\ell)$-approximate quadric. Thus $q(y) = \tfrac{1}{2}ay^{2} + by + c$ for some $|a| \leq \Sigma$ and $b,c \in \R$. With this notation, we parametrise $\V$ with the map 
\begin{equation}\label{form105} V(y,t) = (0,y,q(y) + t) \cdot (\dot{q}(y),0,0) = (ay + b,y,\tfrac{1}{2}by + c + t), \end{equation}
which has the special property $V(y,t) = V(y,0) \cdot (0,0,t)$. The map $V$ is an analogue of the bilipschitz parametrisation "$\Psi$" for $\V$; replace $\tau(y,t)$ by $q(y) + t$ in \eqref{form104} to see the connection). It is easy to see that $V$ is a bilipschitz embedding with constant $\lesssim 1 + \Sigma$; one can either use the formula \eqref{form105} directly, or repeat the proof of Proposition \ref{bilipProp}, noting that $t \mapsto q(y) + t$ is $1$-bilipschitz, and $t \mapsto \partial_{y} [q(y) + t] = \dot{q}(y)$ is independent of $t$.

We next wish to parametrise $\V \cap B_{0}$ by $V$, and we claim the following: there exists an absolute constant $A \geq 1$ such that 
\begin{equation}\label{form106} \V \cap B_{0} \subset V(B[(f_{2}(w_{0}),0),AM(1 + \Sigma)\ell(Q)]), \end{equation}
The proof of \eqref{form106} is based on recalling that $q = q_{Q,\ell}$ is the approximate quadric associated with the horizontal line $\ell \ni w_{0}$. Recalling the definition of the map $\iota_{Q}$ from \eqref{iota}, we in fact have $V(f_{2}(w_{0}),0) = \iota_{Q}(w_{0})$ (use the formula \eqref{form105} to verify this), and hence
\begin{equation}\label{form107} d(f(w_{0}),V(f_{2}(w_{0}),0)) = d(f(w_{0}),\iota_{Q}(w_{0})) \lesssim (1 + \Sigma) \epsilon \ell(Q) \end{equation}
by Lemma \ref{l:Proj_Same_y} (or see the discussion below \eqref{iota}). Since $V$ is an $A(1 + \Sigma)$-bilipschitz parametrisation of $\V$, and $f(w_{0})$ is the centre of $B_{0}$, \eqref{form106} follows from \eqref{form107}.

With \eqref{form106} in hand, \eqref{form100} will evidently follow from
\begin{equation}\label{form108} V(B[(f_{2}(w_{0}),0),AM(1 + \Sigma)\ell(Q)]) \subset N(\Gamma_{\mathcal{T}},o(\eta) \cdot \ell(Q)). \end{equation}
To prove \eqref{form108}, pick $(y,h) \in B[(f_{2}(w_{0}),0),AM(1 + \Sigma)\ell(Q)]$, hence
\begin{equation}\label{form109} |y - f_{2}(w_{0})| \leq A(1 + \Sigma)M \ell(Q) \quad \text{and} \quad \sqrt{|h|} \leq A(1 + \Sigma)M \ell(Q). \end{equation}
We need to prove that $\dist(V(y,t),\Gamma_{\mathcal{T}}) = o(\eta) \cdot \ell(Q)$. We start by verifying
\begin{equation}\label{form102} d(V(y,0),\Psi(y,t_{0})) = o(\eta) \cdot \ell(Q). \end{equation} 
Recall that $t_{0}$ is second coordinate of $w_{0}$. To prove \eqref{form102}, compute explicitly
\begin{displaymath} \Psi(y,t_{0})^{-1} \cdot V(y,0) = (\dot{q}(y) - \dot{\tau}(y,t_{0}),0,q(y) - \tau(y,t_{0})).  \end{displaymath} 
Property ($\tau_{4}$), combined with the first part of \eqref{form109} and $K \geq A(1 + \Sigma)M$, now implies
\begin{displaymath} |\tau(y,t_{0}) - q(y)| \leq \eta \cdot \ell(Q)^{2} \quad \text{and} \quad |\dot{\tau}(y,t_{0}) - \dot{q}(y)| \leq \eta \cdot \ell(Q), \end{displaymath}
and \eqref{form102} follows. As a consequence,
\begin{displaymath}  d(V(y,h),\Psi(y,t_{0}) \cdot (0,0,h)) = d(V(y,0) \cdot (0,0,h),\Psi(y,t_{0}) \cdot (0,0,h)) \leq o(\eta) \cdot \ell(Q). \end{displaymath}
Now $\dist(V(y,h),\Gamma_{\mathcal{T}}) = o(\eta) \cdot \ell(Q)$, and hence \eqref{form108}, will follow as soon as we prove
\begin{equation}\label{form110} \dist(\Psi(y,t_{0}) \cdot (0,0,h),\Gamma_{\mathcal{T}}) = o(\eta) \cdot \ell(Q). \end{equation}
To see this, let us spell out the definition:
\begin{displaymath} \Psi(y,t_{0}) \cdot (0,0,t) = (\dot{\tau}(y,t_{0}),y,\tau(y,t_{0}) + t + \tfrac{1}{2}y \cdot \dot{\tau}(y,t_{0})). \end{displaymath}
By property ($\tau_{2}$), the map $t \mapsto \tau(y,t)$ is a $\theta^{-1}$-bilipschitz homeomorphism $\R \to \R$, for any fixed $y \in \R$, where $\theta \sim_{M,\Sigma} 1$. Recalling from \eqref{form109} that $|h| \lesssim_{M,\Sigma} \ell(Q)^{2}$, there consequently exists $t_{h} \in \R$ with $|t_{h} - t_{0}| \lesssim_{M,\Sigma} \ell(Q)^{2}$ such that $\tau(y,t_{h}) = \tau(y,t_{0}) + h$. To complete the proof of \eqref{form110}, and the proposition, we claim that
\begin{displaymath} d(\Psi(y,t_{0}) \cdot (0,0,h),\Psi(y,t_{h})) \leq_{M,\Sigma} \eta \cdot \ell(Q). \end{displaymath} 
This follows immediately from computing explicitly
\begin{displaymath} [\Psi(y,t_{0}) \cdot (0,0,h)]^{-1} \cdot \Psi(y,t_{h}) = (\dot{\tau}(y,t_{h}) - \dot{\tau}(y,t_{0}),0,0), \end{displaymath}
recalling that  $|t_{h} - t_{0}| \lesssim_{M,\Sigma} \ell(Q)^{2}$, and using the $\tfrac{1}{2}$-H\"older continuity of $t \mapsto \dot{\tau}(y,t)$ with constant $\eta$, recall property ($\tau_{3}$). The proof of Proposition \ref{approxProp} is complete. \end{proof}

\section{Reduction to an inductive construction} 

The purpose of this section is to reduce Proposition \ref{mainProp2} to yet another result, Proposition \ref{t:construction} below. We first describe this proposition, and the details of the reduction are contained in Section \ref{s:reduction}. With no loss of generality, we assume that the top rectangle of the given $(\epsilon,H,\Sigma)$-vertical tree $\mathcal{T}$ is $Q(\mathcal{T}) = [0,1)^{2} \subset \W$. Also, with no loss of generality, we assume that $\mathcal{T}$ is a positive vertical tree, that is, the signature of $\mathcal{T}$ is $+$. The case of negative signature could be reduced to the positive case, but in fact it would be easier to make minor changes in the argument whenever the positive signature is used. 

Instead of constructing $\tau \colon \W \to \R^{2}$ directly, we will prove, inductively, the existence of a sequence of maps $\tau_{n}$, $n \geq 0$, defined on the dyadic horizontal lines of generation $n$:
\begin{displaymath} \mathcal{L}_{n} = \{\R \times \{k \cdot 2^{-n}\} : k \in \Z\}. \end{displaymath}
The lines in $\mathcal{L}_{n}$ will be in one-to-one correspondence with \emph{characteristics of generation $n$}, denoted $\mathcal{X}_{n}$, each of which is a graph of a $C^{2}$-function, with $2\Sigma$-Lipschitz first derivative. For $n \geq 0$, we will construct a map $\tau_{n} \colon \mathcal{L}_{n} \to C^{2}(\R)$. Thus, the $\tau_{n}$-image of every line in $\mathcal{L}_{n}$ will be a real-valued $C^{2}$-function defined on $\R$. The relationship between $\tau_{n}$ and $\tau$ will, eventually, be that if $\ell = \R \times \{k \cdot 2^{-n}\} \in \mathcal{L}_{n}$, then 
\begin{displaymath} \tau(y,k \cdot 2^{-n}) \approx \tau_{n}(\ell)(y), \qquad y \in \R, \end{displaymath}
where the error tends geometrically to $0$ as $n \to \infty$. 

The construction of the maps $\tau_{n}$ involves a number of constant parameters; we gather here some reminders and information. We have already encountered the bilipschitz constant $M$, the "slope" constant $\Sigma$, and the constants $B \geq 5$ (such that $F_{Q}$ is vertically $\epsilon'\ell(Q)$-increasing on $BQ$ by Lemma \ref{lemma4}), $H \geq A(1 + \Sigma)M^{2}B$, and $N = A(1 + \Sigma)M$. Also, the constant $K \geq 1$ already appeared in condition ($\tau_{4}$) of Proposition \ref{mainProp2}, and we mentioned in Remark \ref{rem3} that we need to choose $K \geq A(1 + \Sigma)M$ for the purpose of verifying \eqref{form103}. Given such a constant "$K$", we will below need to assume
\begin{equation}\label{form146} H \geq A\max\{N^{2},KN\}. \end{equation}
The letter "$A$" will continue to denote a generic large absolute constant, whose value may change without separate mention. In addition to all the previous constants, there will be a new small constant $\delta > 0$, which may be specified by the "user", but must be sufficiently small in a manner depending on $M,\Sigma$, and $\eta$ (as in property ($\tau_{3}$) of Proposition \ref{mainProp2}). The constant "$\delta$" here has nothing to do with the "projection constant" which appeared in \eqref{form142}; this previous "$\delta$" will no longer be used.

Finally, the constant $\epsilon > 0$ in the definition of $(\epsilon,H,\Sigma)$-vertical rectangles needs to be chosen small in a manner depending on all the previous constants. All in all, the relations between the constants are given by the following informal sequence:
\begin{equation}\label{allConstants} 1 \lesssim_{\delta,K,M,\Sigma} \epsilon \ll \{\delta,M,\Sigma\} \ll K \ll H \lesssim_{\delta,K,M,\Sigma} 1. \end{equation}

After this introduction, here is the statement containing the properties of the maps $\tau_{n}$:
\begin{proposition}\label{t:construction} Let $M \geq 1$, $\Sigma > 0$, let $\delta > 0$ be sufficiently small relative to $M,\Sigma$, let $K \geq 1$ be sufficiently large depending on $\delta,M,\Sigma$, let $H \geq 1$ be sufficiently large depending $K,M,\Sigma$, and let $\epsilon > 0$ be sufficiently small depending on $\delta,H,K,M,\Sigma$. Let $\mathcal{T}$ be an $(\epsilon,H,\Sigma)$-vertical tree with top $Q(\mathcal{T}) = [0,1)^{2}$. Then, for every $n \geq 0$, there exists a map $\tau_{n} \colon \mathcal{L}_{n} \to C^{2}(\R)$ satisfying the following axioms \eqref{X1}-\eqref{X4}.  \end{proposition}

The axioms \eqref{X1}-\eqref{X4} closely mirror the properties ($\tau_{1}$)-($\tau_{4}$) listed in Proposition \ref{mainProp2}; in fact, \eqref{X1}-\eqref{X5} correspond directly to ($\tau_{1}$)-($\tau_{4}$), and \eqref{X4} is an additional axiom which ensures that the maps $\tau_{n}$ "form a Cauchy sequence in $C^{1}$" -- a deliberately vague statement to be clarified in Section \ref{s:reduction}. We move to the axioms \eqref{X1}-\eqref{X4}.

The first axiom requires that $\dot{\tau}_{n}(\ell)$ is $2\Sigma$-Lipschitz for all $\ell \in \mathcal{L}_{n}$:
\begin{equation}\label{X1} |\dot{\tau}_{n}(\ell)(y_{1}) - \dot{\tau}_{n}(\ell)(y_{2})| \leq 2\Sigma |y_{1} - y_{2}|, \qquad y_{1},y_{2} \in \R, \, n \geq 0, \, \ell \in \mathcal{L}_{n}. \tag{X1} \end{equation}
The characteristics will have the same "vertical ordering" as the lines in $\mathcal{L}_{n}$. Since $\pi_{2}(\ell) \in \R$ is a singleton for all $\ell \in \mathcal{L}_{n}$, we may define $\ell_{2} - \ell_{1} := \pi(\ell_{2}) - \pi(\ell_{1}) \in \R$. We write $\ell_{2} > \ell_{1}$ if $\ell_{2} - \ell_{1} > 0$. If $\ell_{1},\ell_{2} \in \mathcal{L}_{n}$ with $\ell_{2} > \ell_{1}$, the second axiom states that there exists a constant $\theta = \theta(M,\Sigma) > 0$ such that 
\begin{equation}\label{X2} \theta (\ell_{2} - \ell_{1}) \leq \tau_{n}(\ell_{2})(y) - \tau_{n}(\ell_{1})(y) \leq \theta^{-1} (\ell_{2} - \ell_{1}), \qquad y \in \R. \tag{X2} \end{equation}
This is the only axiom where the positive signature of $\mathcal{T}$ is visible. If the signature were negative, one would simply replace $\tau_{n}(\ell_{2})(y) - \tau_{n}(\ell_{1})(y)$ by $\tau_{n}(\ell_{1})(y) - \tau_{n}(\ell_{2})(y)$. 

The slopes of neighbouring characteristics do not differ much: the third axiom states that if $\ell \in \mathcal{L}_{n}$, and $\bar{\ell} \in \{\ell,\ell^{\uparrow},\ell_{\downarrow}\}$, then
\begin{equation}\label{X3} |\dot{\tau}_{n}(\ell)(y) - \dot{\tau}_{n}(\bar{\ell})(y)| \leq \delta \cdot 2^{-n/2}, \qquad y \in \R. \tag{X3} \end{equation}
Here $\ell^{\uparrow}$ and $\ell_{\downarrow}$ refer to the \emph{$\mathcal{L}_{n}$-neighbours of $\ell$}: by definition, $\ell^{\uparrow} := \ell + 2^{-n}$ and $\ell_{\downarrow} := \ell - 2^{-n}$.

The fourth axiom nails the connection between the characteristics in $\mathcal{X}_{n}$ and the map $f = (f_{1},f_{2},f_{3}) \colon \W \to \He$ via the $(Q,\ell)$-approximate quadrics, recall Section \ref{s:appQuad}. For $Q \in \mathcal{D}$, recall the notation $\mathcal{C}(Q)$ for the four corners of $Q$.
\begin{definition}[$\mathcal{C}$ and $\mathcal{C}_{m}$] For $m \in \Z$, we write
\begin{displaymath} \mathcal{C}_{m} := \bigcup_{Q \in \mathcal{T}_{m}} \mathcal{C}(Q) \quad \text{and} \quad \mathcal{C} := \bigcup_{m \in \Z} \mathcal{C}_{m}. \end{displaymath} 
For a horizontal line $\ell \subset \W$, we also write $\mathcal{C}_{m}(\ell) := \mathcal{C}_{m} \cap \ell$. \end{definition}
So, the set $\mathcal{C}$ consists of the corners of rectangles \textbf{only in $\mathcal{T}$}. Since we assume that the top of $\mathcal{T}$ is $[0,1)^{2}$, we have $\mathcal{T}_{m} = \emptyset$, hence $\mathcal{C}_{m} = \emptyset$, for all $m < 0$, and $\mathcal{C}_{0} = \{(0,0),(1,0),(0,1),(1,1)\}$. The axiom \eqref{X5}, below, together with the "stability" axiom \eqref{X4}, implies that the characteristics in $\mathcal{X}_{n}$ resemble $(Q,\ell)$-approximate quadrics in large neighbourhoods of points of the form $f_{2}(w)$, where $w \in \mathcal{C}_{\ceil{n/2}}(\ell)$ is a corner of $Q \in \mathcal{T}_{\ceil{n/2}}$. 

Fix $\ell \in \mathcal{L}_{n}$ and $y \in \R$. Let
\begin{displaymath} w \in \mathcal{C}_{\ceil{n/2}}(\ell) \cup \mathcal{C}_{\ceil{n/2}}(\ell^{\uparrow}) \cup \mathcal{C}_{\ceil{n/2}}(\ell_{\downarrow}) \end{displaymath}
be a corner point such that $y \in B(f_{2}(w),K2^{-n/2})$. Let $Q \in \mathcal{T}_{\ceil{n/2}}$ be any rectangle with corner $w$, and let $q_{Q,\ell}$ be the $(Q,\ell)$-approximate quadric. We require that
\begin{equation}\label{X5} \begin{cases} |\tau_{n}(\ell)(y) - q_{Q,\ell}(y)| = o(\epsilon) \cdot 2^{-n}, \\ |\dot{\tau}_{n}(\ell)(y) - \dot{q}_{Q,\ell}(y)| = o(\epsilon) \cdot 2^{-n/2}, \\ |\ddot{\tau}_{n}(\ell)(y) - \ddot{q}_{Q,\ell}(y)| = o(\epsilon). \end{cases} \tag{X4} \end{equation}
The $o(\epsilon)$-notation here means the following: we write $|g(y) - h(y)| = o(\epsilon) \cdot \mathrm{const}$ if for all $\bar{\epsilon} > 0$, there exists $\epsilon > 0$, depending only on $K,M,\Sigma$, such that $|g(y) - h(y)| \leq \bar{\epsilon} \cdot \mathrm{const}$. The "triple control" in \eqref{X5} is such a common notion in the proof below that we give it a separate abbreviation: for $g,h \in C^{2}(\R)$, and $y \in \R$, we write $g(y) \approx_{n} h(y)$ if 
\begin{equation}\label{approxNotation} \begin{cases} |g(y) - h(y)| = o(\epsilon) \cdot 2^{-n}, \\ |\dot{g}(y) - \dot{h}(y)| = o(\epsilon) \cdot 2^{-n/2}, \\ |\ddot{g}(y) - \ddot{h}(y)| = o(\epsilon). \end{cases} \end{equation}
Thus, \eqref{X5} is abbreviated to $\tau_{n}(y) \approx_{n} q_{Q,\ell}(y)$, for appropriate $Q,\ell,y$.

Characteristics of generation $n - 1$ are generally not characteristics of generation $n$, for $n \geq 1$, but this is almost true: the final axiom states that 
\begin{equation}\label{X4} |\tau_{n}(\ell)(y) - \tau_{n - 1}(\ell)(y)| \leq \delta \cdot 2^{-n} \quad \text{and} \quad |\dot{\tau}_{n}(\ell)(y) - \dot{\tau}_{n - 1}(\ell)(y)| \leq \delta \cdot 2^{-n/2} \tag{X5} \end{equation}
for all $\ell \in \mathcal{L}_{n - 1} \subset \mathcal{L}_{n}$ and all $y \in \R$. 

\begin{remark} The appearance of $\ceil{n/2}$ in axiom \eqref{X5} may appear arbitrary and awkward, but there is a reason for the numerology: for $n \geq 0$, we, informally speaking, want $\tau_{n} \colon \mathcal{L}_{n} \to C^{2}(\R)$ to keep records of all the data of the tree $\mathcal{T}$ down to a certain depth $d(n)$. We choose $d(n)$ to be the \emph{smallest integer such that that all the lines in $\mathcal{L}_{n}$ contain corner points of rectangles in $\mathcal{D}_{d(n)}$}. This means that we have to take $d(n) := \ceil{n/2}$; see Figure \ref{fig5} for an illustration of the case $n = 2$. Note, for example, that the horizontal line $\ell$ in the middle lies in $\mathcal{L}_{1}$, and the smallest generation of rectangles with corners on $\ell$ is $\ceil{1/2} = 1$. Most likely any fixed choice $d(n) \sim n/2$ would work equally well!
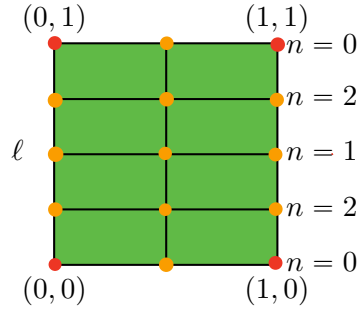
\begin{figure}[h!]
\begin{center}
\begin{overpic}[scale = 0.8]{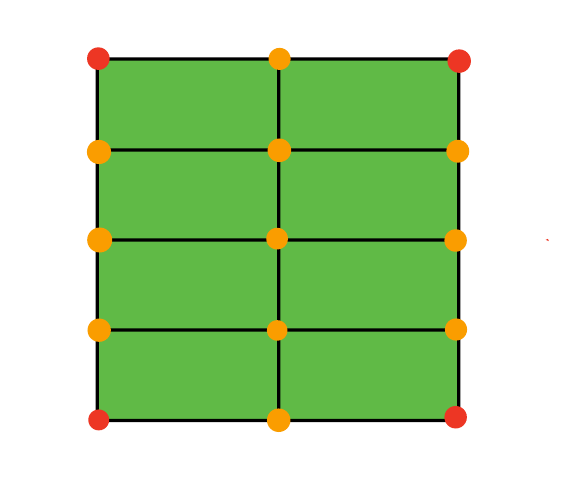}
\put(8,1){$(0,0)$}
\put(72,1){$(1,0)$}
\put(8,80){$(0,1)$}
\put(72,80){$(1,1)$}
\put(84,9){$n = 0$}
\put(84,72){$n = 0$}
\put(84,41){$n = 1$}
\put(84,25){$n = 2$}
\put(84,57){$n = 2$}
\put(5,41){$\ell$}
\end{overpic}
\caption{The points in $\mathcal{C}_{0}$ depicted are in red, and the points in $\mathcal{C}_{1} \, \setminus \, \mathcal{C}_{0}$ are depicted in orange. The generations of the horizontal lines are indicated with "$n = j$" for $j \in \{0,1,2\}$.}\label{fig5}
\end{center}
\end{figure}
\end{remark}

\subsection{From Proposition \ref{t:construction} to Proposition \ref{mainProp2}}\label{s:reduction} Let us take the maps $\tau_{n} \colon \mathcal{L}_{n} \to C^{2}(\R)$ for granted, and use them to complete the proof of Proposition \ref{mainProp2}. We start by quantifying how the maps $\tau_{n}$ "form a Cauchy sequence in $C^{1}(\R)$". We claim that if $N > n \geq 0$, $\ell_{n} \in \mathcal{L}_{n}$, and $\ell_{N} \in \mathcal{L}_{N}$ satisfy $|\ell_{n} - \ell_{N}| \leq 100 \cdot 2^{-n}$, then 
\begin{equation}\label{form111} \|\tau_{n}(\ell_{n}) - \tau_{N}(\ell_{N})\|_{L^{\infty}(\R)} \lesssim \max\{\delta,\theta\} \cdot 2^{-n} \quad \text{and} \quad \|\dot{\tau}_{n}(\ell_{n}) - \dot{\tau}_{N}(\ell_{N})\|_{L^{\infty}(\R)} \lesssim \delta \cdot 2^{-n/2}. \end{equation}
To see this, choose a sequence $\ell_{n + 1},\ldots,\ell_{N}$ of horizontal lines such that $\ell_{j} \in \mathcal{L}_{j}$, and $\ell_{j - 1} \in \{\ell_{j},\ell_{j}^{\uparrow},\ell_{j\downarrow}\}$ for all $j \in \{n + 2,\ldots,N\}$. Then $|\ell_{n + 1} - \ell_{N}| \leq 2^{-n}$, so $|\ell_{n} - \ell_{n + 1}| \leq 101 \cdot 2^{-n}$. It follows that $\ell_{n}$ is $\leq 101$ steps away in $\mathcal{L}_{n}$ from $\ell_{n + 1}$ (if $\ell_{n + 1} \in \mathcal{L}_{n}$) or one of the $\mathcal{L}_{n + 1}$-neighbours of $\ell_{n + 1}$ (if $\ell_{n + 1} \in \mathcal{L}_{n + 1} \, \setminus \, \mathcal{L}_{n}$). For example, if $\ell_{n + 1} \in \mathcal{L}_{n}$, then 
\begin{displaymath} \|\dot{\tau}_{n}(\ell_{n}) - \dot{\tau}_{N}(\ell_{N})\|_{L^{\infty}(\R)} \leq \|\dot{\tau}_{n}(\ell_{n}) - \dot{\tau}_{n}(\ell_{n + 1})\|_{L^{\infty}(\R)} + \sum_{j = n + 2}^{N} \|\dot{\tau}_{j}(\ell_{j}) - \dot{\tau}_{j - 1}(\ell_{j - 1})\|_{L^{\infty}(\R)}. \end{displaymath}
The first term can be estimated by $\leq 101$ applications of axiom \eqref{X3}, and the second term can be estimated using both axioms \eqref{X3} and \eqref{X4}. After summing a geometric series, the result is $\|\dot{\tau}_{n}(\ell_{n}) - \dot{\tau}_{N}(\ell_{N})\|_{L^{\infty}(\R)} \lesssim \delta \cdot 2^{-n/2}$. The other cases, and the first inequality in \eqref{form111}, are proven such a similar manner (also using axiom \eqref{X2}) that we omit the details. 

We now construct the map $\tau \colon \W \to \R^{2}$. Fix $t \in \R$, and choose a sequence of integers $\{k_{n}\}_{n \in \N}$ such that $|t - k_{n} \cdot 2^{-n}| \leq 2^{-n}$ for all $n \in \N$. Then, writing $\ell_{n} := \R \times \{k_{n} \cdot 2^{-n}\} \in \mathcal{L}_{n}$, the inequalities \eqref{form111} imply that $\{\tau_{n}(\ell_{n})\}_{n \in \N}$ is a Cauchy sequence in $C^{1}(\R)$. Hence the limit
\begin{displaymath} \tau(\cdot,t) := \lim_{n \to \infty} \tau_{n}(\ell_{n}) \in C^{1}(\R) \end{displaymath}
exists in the $C^{1}$-norm, and
\begin{displaymath} \dot{\tau}(y,t) = \lim_{n \to \infty} \dot{\tau}_{n}(\ell_{n})(y), \qquad y \in \R. \end{displaymath} 
In particular, we see that $y \mapsto \dot{\tau}(y,t) = \partial_{y} \tau(y,t)$ is $2\Sigma$-Lipschitz by axiom \eqref{X1}, which gives property ($\tau_{1}$). We remark that, as a corollary of \eqref{form111}, we have
\begin{equation}\label{form112} \|\tau_{n}(\ell) - \tau(\cdot,t)\|_{L^{\infty}(\R)} \lesssim \max\{\delta,\theta\} \cdot 2^{-n} \quad \text{and} \quad \|\dot{\tau}_{n}(\ell) - \dot{\tau}(\cdot,t)\|_{L^{\infty}(\R)} \lesssim \delta \cdot 2^{-n/2} \end{equation}
whenever $n \geq 0$, $\ell \in \mathcal{L}_{n}$ and $|\pi_{2}(\ell) - t| \leq 100 \cdot 2^{-n}$. This implies that the definition of $\tau(\cdot,t)$ does not depend on the choice of the sequence $\{k_{n}\}_{n \in \N}$, as long as $|t - k_{n} \cdot 2^{-n}| \leq 2^{-n}$ for all $n \in \N$. This leads to an observation, which will strengthen first inequality in \eqref{form112}. Namely, if $t = \pi_{2}(\ell)$ for some $\ell \in \mathcal{L}_{n}$, i.e. $t$ is a dyadic rational, then $\tau_{m}(\ell)$ is defined for all $m \geq n$, and $\tau(\cdot,t)$ can be expressed as $\tau(\cdot,t) = \lim_{m \to \infty} \tau_{m}(\ell)$. It follows from \eqref{X4} that
\begin{equation}\label{form144} \|\tau_{n}(\ell) - \tau(\cdot,t)\|_{L^{\infty}(\R)} \leq \sum_{m \geq n} \|\tau_{m + 1}(\ell) - \tau_{m}(\ell)\|_{L^{\infty}(\R)} \lesssim \delta \cdot 2^{-n}. \end{equation} 

We then turn to the task of verifying properties ($\tau_{2}$)-($\tau_{4}$) for $\tau$. The separation property ($\tau_{2}$) follows immediately from axiom \eqref{X2}:
\begin{displaymath} \tau(y,t_{2}) - \tau(y,t_{1}) = \lim_{n \to \infty} \tau_{n}(\ell^{1}_{n})(y) - \tau_{n}(\ell^{2}_{n})(y) \in [\theta (t_{2} - t_{1}),\theta^{-1}(t_{2} - t_{1})], \end{displaymath}
where $\{\ell_{n}^{1}\}_{n \in \N},\{\ell_{n}^{2}\}_{n \in \N}$ are sequences of horizontal lines with $|\pi_{2}(\ell_{n}^{1}) - t_{1}| \leq 2^{-n}$ and $|\pi_{2}(\ell_{n}^{2}) - t_{2}| \leq 2^{-n}$. The H\"older continuity property ($\tau_{3}$) follows easily from \eqref{form112} if we choose $\delta \leq c\eta$ for a sufficiently small absolute constant $c > 0$. To see this, let $t_{1},t_{2} \in \R$. The case $|t_{1} - t_{2}| \geq 1$ is a little special: then one needs to check from the case $n = 0$ of the construction of $\tau_{n}$ that the characteristics $\tau_{0}(\ell)$, $\ell \in \mathcal{L}_{0}$, only differ by a constant, hence $\dot{\tau}_{0}(\ell_{1}) \equiv \dot{\tau}_{0}(\ell_{2})$ for $\ell_{1},\ell_{2} \in \mathcal{L}_{0}$. This will indeed be so, see \eqref{form86}-\eqref{form113}, and the discussion below \eqref{form85}. Hence, if $\ell_{1},\ell_{2} \in \mathcal{L}_{0}$ satisfy $|\pi_{2}(\ell_{j}) - t_{j}| \leq 1$, we have
\begin{displaymath} \|\dot{\tau}(\cdot,t_{2}) - \dot{\tau}(\cdot,t_{1})\|_{L^{\infty}(\R)} \leq \|\dot{\tau}_{0}(\ell_{2}) - \dot{\tau}(\cdot,t_{2})\|_{L^{\infty}(\R)} + \|\dot{\tau}_{0}(\ell_{1}) - \dot{\tau}(\cdot,t_{1})\|_{L^{\infty}(\R)} \leq \eta \sqrt{|t_{1} - t_{2}|}. \end{displaymath}
by \eqref{form112}. Consider then the case $|t_{1} - t_{2}| < 1$, and pick $n \geq 0$ such that $2^{-(n + 1)} < |t_{1} - t_{2}| < 2^{-n}$. Let $\ell_{1},\ell_{2} \in \mathcal{L}_{n}$ be lines with $|\pi_{2}(\ell_{j}) - t_{j}| \leq 2^{-n}$. Then $\ell_{1},\ell_{2}$ are neighbours, or perhaps two steps away, in $\mathcal{L}_{n}$, and hence axiom \eqref{X3} implies
\begin{displaymath} \|\dot{\tau}_{n}(\ell_{2}) - \dot{\tau}_{n}(\ell_{1})\|_{L^{\infty}(\R)} \lesssim \delta \cdot 2^{-n/2} \lesssim \delta \cdot \sqrt{|t_{2} - t_{1}|}. \end{displaymath}
It then follows from \eqref{form112}, and the triangle inequality, that $\|\dot{\tau}(\cdot,t_{2}) - \dot{\tau}(\cdot,t_{1})\|_{L^{\infty}(\R)} \leq \eta \sqrt{|t_{2} - t_{1}|}$. The proof of property ($\tau_{3}$) is complete.

Finally, to verify property ($\tau_{4}$), let $Q \in \mathcal{T}$. More precisely, let $m \in \N$ be such that $Q \in \mathcal{T}_{m}$, so $\ell(Q) = 2^{-m}$. Let $w = (x,t) \in \mathcal{C}(Q)$ be a corner, and let $\ell = \pi_{2}^{-1}\{t\} \subset \W$ be the unique horizontal line containing $w$. Since the horizontal edges of all rectangles in $\mathcal{D}_{m}$ are contained on lines in $\mathcal{L}_{2m}$, we have $\ell \in \mathcal{L}_{n}$ and $w \in \mathcal{C}_{\ceil{n/2}}(\ell)$ with $n := 2m$. Next, let $q = q_{Q,\ell}$ be the $(Q,\ell)$-approximate quadric, and fix 
\begin{displaymath} y \in B(f_{2}(w),K\ell(Q)) = B(f_{2}(w),K2^{-n/2}). \end{displaymath}
According to axiom \eqref{X5}, we have 
\begin{equation}\label{form114} |\tau_{n}(\ell)(y) - q(y)| = o(\epsilon) \cdot 2^{-n} \quad \text{and} \quad |\dot{\tau}_{n}(\ell)(y) - \dot{q}(y)| = o(\epsilon) \cdot 2^{-n/2}. \end{equation} 
Choosing $\delta \leq c\eta$ as before, $\epsilon > 0$ small depending on $\eta,K,M,\Sigma$, and combining \eqref{form112}-\eqref{form144} with \eqref{form114}, we find
\begin{displaymath} |\tau(y,t) - q(y)| \leq \eta \cdot \ell(Q)^{2} \quad \text{and} \quad |\dot{\tau}(y,t) - \dot{q}(y)| \leq \eta \cdot \ell(Q), \end{displaymath}
as desired in property ($\tau_{4}$). This completes the proof of Proposition \ref{mainProp2}, modulo the proof of Proposition \ref{t:construction}.

%%%%%%%%%%%

\begin{comment}

\begin{remark} Axiom \eqref{X5} ties the characteristics $\tau_{n}$ to the function $f$. According to Corollary \ref{cor4}, if $Q$ is an $(\epsilon,H,\Sigma)$-vertical rectangle (such as a rectangle in $\mathcal{T}$), if $\ell \subset \W$ is a horizontal line, and $w \in \ell \cap HQ$, then
\begin{equation}\label{form71} |(f_{3}(w) + \tfrac{1}{2}f_{1}(w)f_{2}(w)) - q(f_{2}(w))| \lesssim_{\Sigma} \epsilon \ell(Q)^{2}, \qquad w \in \ell \cap HQ, \end{equation} 
and
\begin{equation}\label{form72} |f_{1}(w) - \dot{q}(f_{2}(w))| \lesssim_{\Sigma} \epsilon \ell(Q), \qquad w \in \ell \cap HQ. \end{equation}
To see the connection to axiom \eqref{X5}, note first that $\mathcal{C}(Q) \subset HQ$. So, if $w \in \mathcal{C}_{\ceil{n/2}}(Q)$ is the corner of $Q \in \mathcal{T}_{\ceil{n/2}}$, then $w \in \ell \cap HQ$, where $\ell \subset \W$ is the unique horizontal line through $w$. Since $y := f_{2}(w) \in B(f_{2}(w),K2^{-n/2})$, it follows from axiom \eqref{X5} and \eqref{form71}-\eqref{form72} that
\begin{displaymath} \tau_{n}(\ell)(y) \approx q_{Q,\ell}(y) \approx f_{3}(w) + \tfrac{1}{2}f_{1}(w)f_{2}(w) \quad \text{and} \quad \dot{\tau}_{n}(\ell)(y) \approx \dot{q}_{Q,\ell}(y) \approx f_{3}(w).  \end{displaymath} 
The conclusion is that, up to small additive errors which tend to zero as $\ell(Q) \to 0$, the values of $\tau_{n}(\ell)$ and $\dot{\tau}_{n}(\ell)$ at $y = f_{2}(w)$ are uniquely determined by $f$ whenever $w$ is a corner point of a rectangle $Q \in \mathcal{T}$. \end{remark}

\end{comment}

%%%%%%%%%%%%%%%%%%

 \section{Constructing the characteristics}
 
 In this section, we prove Proposition \ref{t:construction}.
 
 \subsection{Case $n = 0$} The construction of the characteristics in $\mathcal{X}_{n}$ is inductive, and incorporates information about the characteristics already constructed in $\mathcal{X}_{n - 1}$. The only exception is the case $n = 0$, where we now begin; then $\ceil{n/2} = 0$, and 
 \begin{displaymath} \mathcal{L}_{0} = \{\R \times \{k\} : k \in \Z\} =: \{\ell_{k} : k \in \Z\}. \end{displaymath}
 Let $Q_{0} := [0,1)^{2}$ be the only rectangle in $\mathcal{T}_{0}$. For $k \in \{-1,0,1,2\}$, let $q_{k} := q_{Q_{0},\ell_{k}}$ be the $(Q_{0},\ell_{k})$-approximate quadric. For the same indices $k \in \{-1,0,1,2\}$, we also define
 \begin{equation}\label{form86} \tau_{0}(\ell_{k})(y) := q_{k}(y). \end{equation}
 For the indices $k \notin \{-1,0,1,2\}$, the definition of $\tau_{0}(\ell_{k})$ is a bit arbitrary: let us set
 \begin{equation}\label{form113} \tau_{0}(\ell_{k}) := \tau_{0}(\ell_{0})(y) + C_{M,\Sigma}k, \qquad k \notin \{-1,0,1,2\}, \end{equation} 
 where $C_{M,\Sigma} \sim N^{1/2}$ is a constant to be determined shortly. 
 
 This completes the definition of $\tau_{0} \colon \mathcal{L}_{0} \to C^{1,1}(\R)$, and now we need to check the axioms \eqref{X1}-\eqref{X4}. Axiom \eqref{X1} is easy: since $\mathcal{T}$ is an $(\epsilon,H,\Sigma)$-vertical tree, every $(Q,\ell)$-approximate quadric $q_{Q,\ell}$ with $Q \in \mathcal{T}$ satisfies $\|\ddot{q}_{Q,\ell}\|_{L^{\infty}(\R)} \leq \Sigma$. In particular,
 \begin{equation}\label{form85} \|\ddot{\tau}_{0}(\ell_{k})\|_{L^{\infty}(\R)} \leq \Sigma, \qquad k \in \Z. \end{equation}
 which is even better than \eqref{X1}. Axiom \eqref{X2} requires some work, so we postpone it for a moment. Axiom \eqref{X3} is a triviality: the quadrics $q_{k}$ are associated to the same rectangle $Q_{0}$, hence only differ by a constant (the graphs of every $q_{k}$ arise as $\Pi$-projections of certain horizontal in the vertical plane $\V_{Q_{0}}$). Hence $\dot{\tau}_{0}(\ell_{k}) \equiv \dot{\tau}_{0}(\ell_{k + 1})$ for all $k \in \Z$. 
  
There is very little to verify in axiom \eqref{X5}, but let us do this over-formally to get some familiarity with the axiom. Let $\ell \in \mathcal{L}_{0}$ and fix $y \in \R$. Assume that there exists a rectangle $Q \in \mathcal{T}_{0}$ with corner $w \in \mathcal{C}_{0}(\ell_{\downarrow}) \cup \mathcal{C}_{0}(\ell) \cup \mathcal{C}_{0}(\ell^{\uparrow})$ such that $y \in B(f_{2}(w),K)$. The claim is that $\tau_{0}(\ell)(y) \approx_{0} q_{Q,\ell}(y)$, where $q_{Q,\ell}$ is the $(Q,\ell)$-approximate quadric. 
 
Since $Q_{0}$ is the only rectangle in $\mathcal{T}_{0}$, we have $Q = Q_{0}$, and $\ell \in \{\ell_{-1},\ell_{0},\ell_{1},\ell_{2}\}$. For such lines, in \eqref{form86}, we defined $\tau_{0}(\ell)$ to be the $(Q_{0},\ell)$-approximate quadric. Hence $\tau_{0}(\ell) := q_{Q_{0},\ell} \approx_{0} q_{Q_{0},\ell}$, and the verification of axiom \eqref{X5} is complete.

Axiom \eqref{X4} is meaningless for the time being, since it compares $\tau_{n - 1}(\ell)$ with $\tau_{n}(\ell)$ for $n \geq 1$. So, we finally turn to axiom \eqref{X2}. We first claim that for $k \in \{0,1,2\}$,
\begin{equation}\label{form37aa} \tau_{0}(\ell_{k})(y) - \tau_{0}(\ell_{k - 1})(y) = q_{k}(y) - q_{k - 1}(y) \in [cN^{-1/2},CN^{1/2}], \qquad y \in \R, \end{equation}
where $0 < c \leq C < \infty$ are absolute constants. To see this, we recall that the signature of $\mathcal{T}$, and hence $Q_{0} \in \mathcal{T}$ is $+$, which meant that the $(N,\epsilon)$-QIE
\begin{displaymath} F := F_{Q_{0}} := \pi_{\V_{Q_{0}}} \circ \iota_{Q_{0}} = \psi_{\V_{Q_{0}}} \circ \Pi \circ \iota_{Q_{0}} \colon HQ_{0} \to \W \end{displaymath}
is vertically $\epsilon'$-increasing on the rectangle $BQ_{0} \subset HQ_{0}$, where $\epsilon' = 2N\epsilon$. This fact was established in Lemma \ref{lemma4}, assuming that $H \geq A(1 + \Sigma)^{2}MB$. The quadric straightening map $\psi_{0} := \psi_{\mathbb{V}_{Q_{0}}} \colon \W \to \W$ appearing in the definition of $F$ was introduced just above \eqref{form21}. It had the form
\begin{displaymath} \psi_{0}(y,t) = (y,t - \tfrac{1}{2}ay^{2} - by), \end{displaymath}
where $a = a_{Q_{0}}$ and $b = b_{Q_{0}}$ are the common coefficients of approximate $(Q_{0},\ell)$-quadrics, all of which have the form $q_{\ell}(y) = \tfrac{1}{2}ay^{2} + by + c_{\ell}$. In particular, the quadrics $q_{-1},q_{0},q_{1},q_{2}$ can be written in this general form for certain coefficients $c_{j} := c_{\ell_{j}} \in \R$, and $q_{k} - q_{k - 1} \equiv c_{k} - c_{k - 1}$. So, proving \eqref{form37aa} is equivalent to showing that $c_{k} - c_{k - 1} \sim_{N} 1$. For this purpose, we define $w_{k} := (0,k) \in \W$. Since $F$ is vertically $\epsilon'$-increasing on $BQ_{0}$, and $B \geq 5$, in particular the map $t \mapsto \pi_{2}(F(0,t))$ is $\epsilon'$-increasing on $[k - 1,k]$ for $k \in \{0,1,2\}$. Since $\|\pi_{2}(w_{k}) - \pi_{2}(w_{k - 1})\| > \epsilon'$ (recall from \eqref{form22} that $\epsilon \ll N^{-10}$), and $(\pi_{2} \circ F)|_{\ell} \equiv c_{\ell}$ by the expression \eqref{form45}, this implies
\begin{displaymath} c_{k} = \pi_{2}(F(w_{k})) > \pi_{2}(F(w_{k - 1})) > c_{k - 1}, \qquad k \in \{0,1,2\}. \end{displaymath}
To establish the more quantitative estimate needed for \eqref{form37aa}, we use Lemma \ref{lemma6}, which implies that the map $t \mapsto \pi_{2}(F(0,t))$ is an $(N,2\epsilon)$-QIE on $([k - 1,k],\|\cdot\|)$ (for applying Lemma \ref{lemma6}, we also use that $F$ is a QIE on $HQ_{0}$, a large neighbourhood of $[-1,2]^{2}$). Hence, for $k \in \{0,1,2\}$,
\begin{align*} (2N)^{-1} & \leq N^{-1} - 2\epsilon \leq \|\pi_{2}(F(w_{k})) - \pi_{2}(F(w_{k - 1}))\| = \|c_{k} - c_{k - 1}\| \leq N + 2\epsilon \leq 2N. \end{align*}
This completes the proof of \eqref{form37aa}. To complete the proof of axiom \eqref{X2}, we recall the definition
\begin{displaymath} \tau_{0}(\ell_{k})(y) := \tau_{0}(\ell_{0})(y) + C_{M,\Sigma}k = q_{0}(y) + C_{M,\Sigma}k, \qquad  k \notin \{-1,0,1,2\}. \end{displaymath}
We see from \eqref{form37aa}, and recalling that all the quadrics $q_{k}$ only differ by a constant, that axiom \eqref{X2} is satisfied if $C_{M,\Sigma} \sim N^{1/2}$ is chosen sufficiently large. This completes the proof of Proposition \ref{t:construction} in the case $n = 0$.

\subsection{Case $n \geq 1$}\label{s:genCase} We next assume that the map $\tau_{n - 1} \colon \mathcal{L}_{n - 1} \to C^{1,1}(\R)$ has already been constructed for some $n \geq 1$, and satisfies axioms \eqref{X1}-\eqref{X4}. We now describe the construction of $\tau_{n} \colon \mathcal{L}_{n} \to C^{1,1}(\R)$. Fix $\ell \in \mathcal{L}_{n}$. The construction is a little different in the cases where also $\ell \in \mathcal{L}_{n - 1}$ (so $\tau_{n - 1}(\ell)$ is defined), and where $\ell \in \mathcal{L}_{n} \, \setminus \, \mathcal{L}_{n - 1}$. The constructions can, however, be carried out simultaneously.  

Write $m := \ceil{n/2}$, and let $\mathcal{C}_{m}^{+}(\ell) := \mathcal{C}_{m}(\ell) \cup \mathcal{C}_{m}(\ell^{\uparrow}) \cup \mathcal{C}_{m}(\ell_{\downarrow})$. We would like to define
\begin{displaymath} \mathcal{I}(\ell) := \{I_{n}(w) : w \in \mathcal{C}^{+}_{m}(\ell)\}, \quad \text{where} \quad I_{n}(w) := \bar{B}(f_{2}(w),K2^{-n/2}), \end{displaymath}
but the intervals $I_{n}(w)$, $w \in \mathcal{C}_{m}^{+}(\ell)$, may have plenty of overlap, which leads to technical issues. To avoid these, we prune $\mathcal{I}(\ell)$ as follows. Pick $I \in \mathcal{I}(\ell)$ arbitrarily and see if it is covered by the intervals in $\mathcal{I}(\ell) \, \setminus \, \{I\}$. If it is, discard it; otherwise keep it. The process terminates in $\leq \card \mathcal{I}(\ell) < \infty$ steps. The union of the intervals never changed during the process, and the remaining intervals have overlap bounded by $2$: if three intervals of the same length meet at a common point, then one is covered by the two others.

We keep the notation $\mathcal{I}(\ell)$ for the sub-collection obtained above. The key facts are these:
\begin{equation}\label{form87} \bigcup_{w \in \mathcal{C}_{m}^{+}(\ell)} I_{n}(w) = \bigcup_{I \in \mathcal{I}(\ell)} I \quad \text{and} \quad \sum_{I \in \mathcal{I}(\ell)} \mathbf{1}_{I} \leq 2. \end{equation}
From the second property in \eqref{form87}, we infer that also
\begin{equation}\label{form88} \sum_{I \in \mathcal{I}(\ell)} \mathbf{1}_{2I} \leq 4, \end{equation}
because if $y \in 2I_{1} \cap \cdots \cap 2I_{5}$, then $I_{j} \subset [x - |I|,x + |I|]$ for all $1 \leq j \leq 5$, hence $5|I| \leq \int_{x - |I|}^{x + |I|} \sum \mathbf{1}_{I} \leq 4|I|$ by \eqref{form87}, a contradiction. 

With \eqref{form87}-\eqref{form88} in hand, we construct a partition of unity for $\ell \in \mathcal{L}_{n}$ fixed. Define the (possibly empty) closed sets
\begin{equation}\label{form91} U(\ell) := \bigcup_{I \in \mathcal{I}(\ell)} I = \bigcup_{w \in \mathcal{C}_{m}(\ell)} I_{n}(w) \quad \text{and} \quad U^{\sharp}(\ell) := \bigcup_{I \in \mathcal{I}(\ell)} \kappa_{0} I = \bigcup_{w \in \mathcal{C}_{m}(\ell)} \kappa_{0}I_{n}(w), \end{equation}
where $\kappa_{0} \in (1,\tfrac{4}{5}\sqrt{2})$ is an arbitrary fixed constant; let us for concreteness set 
\begin{displaymath} \kappa_{0} := \tfrac{1}{2}(1 + \tfrac{4}{5}\sqrt{2}). \end{displaymath}
The last set equation in \eqref{form91} requires a little argument, since the selection of the intervals in $\mathcal{I}(\ell)$ had, in principle, nothing to do with the enlargements $\kappa_{0}I$. Fortunately, the same selection process simultaneously yields the last equation in \eqref{form91}. Indeed, if some interval $\kappa_{0}I_{n}(w)$, $w \in \mathcal{C}_{m}(\ell)$, were not covered by the intervals $\kappa_{0}I$ with $I \in \mathcal{I}(\ell)$, then in particular $I_{n}(w) \notin \mathcal{I}(\ell)$. Hence $I_{n}(w) \subset I_{1} \cup \ldots \cup I_{k}$ for certain intervals $I_{1},\ldots,I_{k} \in \mathcal{I}(\ell)$, and it is easy to check that consequently $\kappa_{0}I_{n}(w) \subset \kappa_{0}I_{1} \cup \ldots \cup \kappa_{0}I_{k}$. 

Let $\{\varphi_{I,\ell}\}_{I \in \mathcal{I}(\ell)} \subset C^{\infty}(\R)$ be a family of functions satisfying the following properties:
\begin{displaymath} \mathbf{1}_{U(\ell)} \leq \sum_{I \in \mathcal{I}(\ell)} \varphi_{I,\ell} \leq \mathbf{1}_{U^{\sharp}(\ell)} \quad \text{and} \quad \spt \varphi_{I,\ell} \subset \kappa_{0}I \text{ for all } I \in \mathcal{I}(\ell). \end{displaymath}
Define also
\begin{displaymath} \varphi_{\ell} = 1 - \sum_{I \in \mathcal{I}(\ell)} \varphi_{I,\ell}, \quad \text{hence} \quad \mathbf{1}_{\R \, \setminus \, U^{\sharp}(\ell)} \leq \varphi_{\ell} \leq \mathbf{1}_{\R \, \setminus \, U(\ell)}. \end{displaymath}
Since $\kappa_{0} \leq 2$, the overlap condition \eqref{form88} holds, and $|I| = 2K2^{-n/2}$ for all $I \in \mathcal{I}(\ell)$, the functions $\varphi_{I,\ell}$ and $\varphi_{\ell}$ can be easily chosen to satisfy the following bounds for all $y \in \R$:
\begin{equation}\label{form70} |\dot{\varphi}_{\ell}(y)| + \sum_{I \in \mathcal{I}(\ell)} |\dot{\varphi}_{I,\ell}(y)| \lesssim \tfrac{1}{K} \cdot 2^{n/2} \quad \text{and} \quad |\ddot{\varphi}_{\ell}(y)| + \sum_{I \in \mathcal{I}(\ell)} |\ddot{\varphi}_{I,\ell}(y)| \lesssim \tfrac{1}{K^{2}} \cdot 2^{n}. \end{equation} 

Before defining $\tau_{n}(\ell)$, we associate a quadric to each interval $I \in \mathcal{I}(\ell)$. Recall that every interval $I \in \mathcal{I}(\ell)$ has the form $I = I_{n}(w)$ for some $w_{I} \in \mathcal{C}_{m}(\ell) \cup \mathcal{C}_{m}(\ell^{\uparrow}) \cup \mathcal{C}_{m}(\ell_{\downarrow})$. Let $Q_{I} \in \mathcal{T}_{m}$ be any rectangle with corner $w_{I}$, and let 
\begin{displaymath} q_{I,\ell} := q_{Q_{I},\ell} \end{displaymath}
be the $(Q_{I},\ell)$-approximate quadric. There may be several (up to $4$) admissible choices of $Q_{I}$, and they might produce different quadrics $q_{I,\ell}$. However, we will see in Proposition \ref{prop4} below that the different quadrics nearly agree, so the particular choice does not matter. We are now prepared to define $\tau_{n}(\ell)$. If $\ell \in \mathcal{L}_{n - 1}$, set
\begin{equation}\label{form28} \tau_{n}(\ell) := \sum_{I \in \mathcal{I}(\ell)} \varphi_{I,\ell} \cdot q_{I,\ell} + \varphi_{\ell} \cdot \tau_{n - 1}(\ell). \end{equation}
If $\ell \in \mathcal{L}_{n} \, \setminus \, \mathcal{L}_{n - 1}$, set instead
\begin{equation}\label{form31} \tau_{n}(\ell) := \sum_{I \in \mathcal{I}(\ell)} \varphi_{I,\ell} \cdot q_{I,\ell} + \tfrac{1}{2} \cdot \varphi_{\ell} \cdot [\tau_{n - 1}(\ell^{\uparrow}) + \tau_{n - 1}(\ell_{\downarrow})], \end{equation} 
observing that $\ell^{\uparrow},\ell_{\downarrow} \in \mathcal{L}_{n - 1}$. %The only cases where one of $\ell^{\uparrow}$ or $\ell_{\downarrow}$ is not defined is where $\ell = \R \times \{0\}$ or $\ell = \R \times \{1\}$, but in this case $\ell \in \mathcal{L}_{0} \subset \mathcal{L}_{n - 1}$, and $\tau_{n}(\ell)$ is defined by \eqref{form28}.

\subsubsection{Preparatory lemmas} Before beginning to verify axioms \eqref{X1}-\eqref{X4}, we prove a few lemmas, which will be needed in the verification of multiple axioms. In these lemmas, the definitions of, and relationships between, the constants $\epsilon,H,K,M,N,\Sigma$ are as stated before. We recall from \eqref{allConstants} that $\epsilon \ll \{M,\Sigma\} \ll K \ll H$.

The first lemma uses nothing but the $M$-Lipschitz property of $f$, hence $f_{2}$.
\begin{lemma}\label{lemma11} Let $w_{1},w_{2} \in \W$ with $d_{\mathrm{par}}(w_{1},w_{2}) \leq 100 \cdot 2^{-n/2}$. Then, assuming that $K \geq 100AM$ for $A > (\tfrac{4}{5}\sqrt{2} - \kappa_{0})^{-1}$, we have
\begin{displaymath} B(f_{2}(w_{2}),\kappa_{0}K2^{-n/2}) \subset B(f_{2}(w_{1}),\tfrac{4}{5}K2^{-(n - 1)/2}) \subset B(f_{2}(w_{1}),\kappa_{0}K2^{-(n - 1)/2}). \end{displaymath} 
\end{lemma}

\begin{proof} Since $f_{2}$ is $M$-Lipschitz, and $d_{\mathrm{par}}(w_{1},w_{2}) \leq 100 \cdot 2^{-n/2}$, we have
\begin{displaymath} B(f_{2}(w_{2}),\kappa_{0}K2^{-n/2}) \subset B(f_{2}(w_{1}),\kappa_{0}K2^{-n/2} + 100 \cdot M2^{-n/2}).  \end{displaymath}
Since $K \geq 100AM$, and $\kappa_{0} = \tfrac{1}{2}(1 + \tfrac{4}{5}\sqrt{2})$, here $\kappa_{0}K + 100M \leq (\kappa_{0} + A^{-1})K < \tfrac{4}{5}\sqrt{2} \cdot K$, hence
\begin{displaymath} B(f_{2}(w_{1}),\kappa_{0}K2^{-n/2} + M2^{-(n - 1)/2}) \subset B(f_{2}(w_{1}),\tfrac{4}{5}K2^{-(n - 1)/2}). \end{displaymath}
The second inclusion simply follows from $\tfrac{4}{5} \leq \kappa_{0}$. \end{proof}

The next lemma implies that the "$\approx_{n}$"-notation introduced in \eqref{approxNotation} interacts well with partitions of unity. For technical purposes, we record a slightly more quantitative result:

\begin{lemma}\label{lemma13} Let $\{\varphi_{I}\}_{I \in \mathcal{I}} \subset C^{2}(\R)$ be a family of functions satisfying
\begin{equation}\label{form96} \sum_{I \in \mathcal{I}} \varphi_{I}(y_{0}) = 1, \quad \sum_{I \in \mathcal{I}} |\dot{\varphi}_{I}(y_{0})| \leq c_{1} \cdot 2^{n/2}, \quad \text{and} \quad \sum_{I \in \mathcal{I}} |\ddot{\varphi}_{I}(y_{0})| \leq c_{2} \cdot 2^{n} \end{equation}
for some fixed $y_{0} \in \R$. Let $\{p_{I}\}_{I \in \mathcal{I}} \cup \{q_{I}\}_{I \in \mathcal{I}} \subset C^{2}(\R)$ be further families of functions satisfying
\begin{displaymath} |p_{I}(y_{0}) - q_{I}(y_{0})| \leq \epsilon_{1} \cdot 2^{-n}, \quad |\dot{p}_{I}(y_{0}) - \dot{q}_{I}(y_{0})| \leq \epsilon_{2} \cdot 2^{-n/2}, \quad \text{and} \quad |\ddot{p}_{I}(y_{0}) - \ddot{q}_{I}(y_{0})| \leq \epsilon_{3} \end{displaymath} 
for all $I \in \mathcal{I}$ with $y_{0} \in \spt \varphi_{I}$. Here $c_{1},c_{2},\epsilon_{1},\epsilon_{2},\epsilon_{3} > 0$ are constants. Then, writing $Q_{p} := \sum \varphi_{I}p_{I}$ and $Q_{q} := \sum \varphi_{I}q_{I}$, we have
\begin{equation}\label{form95} \begin{cases} |Q_{p}(y_{0}) - Q_{q}(y_{0})| \leq \epsilon_{1} \cdot 2^{-n},\\ |\dot{Q}_{p}(y_{0}) - \dot{Q}_{q}(y_{0})| \leq (c_{1}\epsilon_{1} + \epsilon_{2}) \cdot 2^{-n/2},\\ |\ddot{Q}_{p}(y_{0}) - \ddot{Q}_{q}(y_{0})| \leq c_{2}\epsilon_{1} + 2c_{1}\epsilon_{2} + \epsilon_{3}. \end{cases} \end{equation} 

\end{lemma}
\begin{proof} Write $Q_{p}(y_{0}) - Q_{q}(y_{0}) = \sum_{I \in \mathcal{I}} \varphi_{I}(y_{0}) \cdot [p_{I}(y_{0}) - q_{I}(y_{0})]$, and note that the sum can be restricted to those $I \in \mathcal{I}$ with $y_{0} \in \spt \varphi_{I}$. Then \eqref{form95} follows by direct computation.  \end{proof}

Quite often the following corollary is enough: 

\begin{cor}\label{cor7} Let $\{\varphi_{I}\}_{I \in \mathcal{I}}$ be a partition of unity satisfying \eqref{form96} with $c_{1},c_{2} \lesssim 1$. Let $\{q_{I}\}_{I \in \mathcal{I}(\ell)} \subset C^{2}(\R)$. Assume that, for some $y_{0} \in \R$ fixed, there exists a function $q = q_{y_{0}} \in C^{2}(\R)$ such that 
\begin{displaymath} q_{I}(y_{0}) \approx_{n} q(y_{0}) \quad \text{for all} \quad I \in \mathcal{I}(\ell) \quad \text{with} \quad y_{0} \in \spt \varphi_{I,\ell}. \end{displaymath} 
Then, $q(y_{0}) \approx_{n} Q(y_{0})$, where $Q = \sum_{I \in \mathcal{I}} \varphi_{I}q_{I}$. \end{cor} 

The next proposition gives a condition for checking that the hypothesis of Corollary \ref{cor7} is valid. The notation $q_{Q,\ell}$ will refer to the $(Q,\ell)$-approximate quadric associated with an $(\epsilon,H,\Sigma)$-vertical rectangle $Q \in \mathcal{D}$.

\begin{proposition}\label{prop4} Let $w_{1},w_{2} \in \mathcal{C}_{\ceil{n/2}} \cup \mathcal{C}_{\ceil{(n - 1)/2}}$ be corners of $Q_{1},Q_{2} \in \mathcal{T}_{\ceil{n/2}} \cup \mathcal{T}_{\ceil{(n - 1)/2}}$. Assume that $w_{1} \in \ell_{1}$ and $w_{2} \in \ell_{2}$, where $\ell_{1},\ell_{2} \subset \W$ are horizontal lines with $|\ell_{1} - \ell_{2}| \leq 100 \cdot 2^{-n}$. Then, if $\ell \subset \W$ is an additional horizontal line with $|\ell - \ell_{1}| \leq 100 \cdot 2^{-n}$, we have
\begin{equation}\label{form65} q_{Q_{1},\ell}(y) \approx_{n} q_{Q_{2},\ell}(y), \qquad y \in B(f_{2}(w_{1}),100K2^{-n/2}) \cap B(f_{2}(w_{2}),100K2^{-n/2}). \end{equation}
\end{proposition}

\begin{proof} Write $m := \ceil{n/2}$. Then $2^{-m} \leq \ell(Q_{j}) \leq 2^{-m + 1}$ for both $j \in \{1,2\}$. If the intersection in \eqref{form65} is empty, there is nothing to prove. In the opposite case, we have
\begin{equation}\label{form66} |f_{2}(w_{1}) - f_{2}(w_{2})| \leq 200K \cdot 2^{-m}. \end{equation}
Based on this, we claim that
\begin{equation}\label{form67} d_{\mathrm{par}}(w_{1},w_{2}) \lesssim \max\{KN,N^{2}\} \cdot 2^{-m}. \end{equation}
Writing $w_{j} := (y_{j},t_{j})$, $j \in \{1,2\}$, the hypothesis of the lemma implies that $\|t_{1} - t_{2}\| \lesssim 2^{-m}$. Therefore, if 
\begin{equation}\label{form74} |y_{1} - y_{2}| \leq \max\{\tfrac{1}{100} \cdot 2^{-m},4N^{2}\|t_{1} - t_{2}\|\}, \end{equation} 
the estimate \eqref{form67} follows immediately. In the opposite case, we rely on Proposition \ref{prop3}, which we recall here for the reader's convenience (with terminology slightly adapted to the present situation):
\begin{proposition}\label{prop3a} Let $Q_{1},Q_{2} \in \mathcal{T}$,
and let $w_{j} = (y_{j},t_{j}) \in \tfrac{1}{2}HQ_{j}$, $j \in \{1,2\}$, with
\begin{displaymath} |y_{1} - y_{2}| > \max\{ \tfrac{1}{100} \cdot \min\{\ell(Q_{1}),\ell(Q_{2})\},4N^{2}\|t_{1} - t_{2}\|\}. \end{displaymath}
Then,
\begin{displaymath} |f_{2}(w_{2}) - f_{2}(w_{1})| \geq \tfrac{1}{2N}d_{\mathrm{par}}(w_{1},w_{2}). \end{displaymath}
\end{proposition}
With the proposition in hand, recalling \eqref{form66}, and assuming that \eqref{form74} fails, we estimate
\begin{displaymath} d_{\mathrm{par}}(w_{1},w_{2}) \leq 2N|f_{2}(w_{1}) - f_{2}(w_{2})| \lesssim KN2^{-m}, \end{displaymath} 
and \eqref{form67} follows. The main consequence of \eqref{form67} is that $\dist_{\mathrm{par}}(Q_{1},Q_{2}) \leq d_{\mathrm{par}}(w_{1},w_{2}) \lesssim \max\{KN,N^{2}\} \ell(Q_{j})$ for $j \in \{1,2\}$, recalling $Q_{1},Q_{2} \in \mathcal{T}_{m}$. Therefore, if $H \gg \max\{KN,N^{2}\}$ is sufficiently large, then $\tfrac{1}{2}HQ_{1} \cap \tfrac{1}{2}HQ_{2} \cap \ell \neq \emptyset$, and we may even pick two "horizontally aligned" points 
\begin{displaymath} v_{1},v_{2} \in \ell \cap \tfrac{1}{2}HQ_{1} \cap \tfrac{1}{2}HQ_{2} \quad \text{with} \quad d_{\mathrm{par}}(v_{1},v_{2}) = 2^{-m}. \end{displaymath}
These points will help us to show that the quadrics $q_{1} := q_{Q_{1},\ell}$ and $q_{2} := q_{Q_{2},\ell}$ are almost the same, which is what we claimed. To make this precise, we will appeal to the following \emph{quadric comparison lemma}, whose proof will be postponed to Appendix \ref{s:quadricComparison}:
\begin{lemma}[Quadric comparison lemma]\label{quadricComparison} Let $q \colon \R \to \R$ be a quadric, and let $I = [a_{1},a_{2}] \subset \R$ be an interval. Write
\begin{displaymath} m := \max\{|q(a_{1})|,|q(a_{2})|\} \quad \text{and} \quad \dot{m} := \max\{|\dot{q}(a_{1})|,|\dot{q}(a_{2})|\}. \end{displaymath}
Then, for all $r \geq 1$ and $t \in rI$,
\begin{equation}\label{quadricEst} |q(t)| \lesssim (m + |I|\dot{m})r^{2}, \quad |\dot{q}(t)| \lesssim (\dot{m} + \tfrac{m}{|I|})r, \quad \text{and} \quad |\ddot{q}(t)| \lesssim \tfrac{m}{|I|^{2}} + \tfrac{\dot{m}}{|I|}.  \end{equation} 
\end{lemma}

We aim to apply the lemma to the quadric $q := q_{1} - q_{2}$ on the interval $I := [a_{1},a_{2}] := [f_{2}(v_{1}),f_{2}(v_{2})]$ (or $I = [f_{2}(v_{2}),f_{2}(v_{1})]$), where $v_{1},v_{2} \in \ell \cap \tfrac{1}{2}HQ_{1} \cap \tfrac{1}{2}HQ_{2}$ were chosen above. To this end, we need to have upper and lower bounds for $|I|$, plus estimates for $|q(a_{j})|,|\dot{q}(a_{j})|$, $j \in \{1,2\}$. Let us start by estimating $|I|$. Since $v_{1},v_{2} \in \ell \cap \tfrac{1}{2}HQ_{1} \cap \tfrac{1}{2}HQ_{2}$, and $d_{\mathrm{par}}(v_{1},v_{2}) = \ell(Q_{1}) = \ell(Q_{2})$, Proposition \ref{prop3a} can be applied, once more:
\begin{displaymath} |I| = |f_{2}(v_{2}) - f_{2}(v_{1})| \geq \tfrac{1}{2N}d_{\mathrm{par}}(v_{1},v_{2}) = \tfrac{1}{2N} \cdot 2^{-m}. \end{displaymath} 
Evidently $|I| \leq M2^{-m}$, since $f_{2}$ is $M$-Lipschitz. 

To estimate $|q(a_{j})|,|\dot{q}(a_{j})|$, $j \in \{1,2\}$, we observe that $v_{1},v_{2} \in \ell \cap HQ_{1} \cap HQ_{2}$, and then apply Corollary \ref{cor4} to both quadrics $q_{i}$ at both points $a_{j} = f_{2}(v_{j})$:
\begin{displaymath} |f_{t}(v_{j}) - q_{i}(a_{j})| \lesssim_{\Sigma} \epsilon \cdot 2^{-n} \quad \text{and} \quad |f_{x}(v_{j}) - \dot{q}_{i}(a_{j})| \lesssim_{\Sigma} \epsilon \cdot 2^{-m}, \qquad i,j \in \{1,2\}. \end{displaymath} 
Using the triangle inequality, and recalling that $q = q_{1} - q_{2}$, we infer
\begin{displaymath} m := \max_{j \in \{1,2\}} |q(a_{j})| \lesssim_{\Sigma} \epsilon \cdot 2^{-n} \quad \text{and} \quad \dot{m} := \max_{j \in \{1,2\}} |q(a_{j})| \lesssim_{\Sigma} \epsilon \cdot 2^{-m}. \end{displaymath}
Now, combining these bounds with the length estimate $|I| \sim_{N} 2^{-m} \sim 2^{-n/2}$, we may apply the quadric comparison lemma: for any $r \geq 1$ and $t \in rI$, 
\begin{equation}\label{form52} \begin{cases} |q_{1}(t) - q_{2}(t)| \lesssim_{r,N,\Sigma} \epsilon \cdot 2^{-n}, \\ |\dot{q}_{1}(t) - \dot{q}_{2}(t)| \lesssim_{r,N,\Sigma} \epsilon \cdot 2^{-m} \\ |\ddot{q}_{1}(t) - \ddot{q}_{2}(t)| \lesssim_{N,\Sigma} \epsilon. \end{cases} \end{equation}
The bound for the second derivatives is naturally independent or $r$, since the second derivative of a quadric is constant. As for the first two estimates, the point "$t$" that we are really interested in is the point $y \in B(f_{2}(w_{1}),100K2^{-n/2}) \cap B(f_{2}(w_{2}),100K2^{-n/2})$ given by \eqref{form65}. Since $d_{\mathrm{par}}(v_{j},w_{j}) \lesssim H2^{-m}$, it follows from the $M$-Lipschitz property of $f_{2}$ that
\begin{displaymath} \dist(y,I) \leq |y - f_{2}(w_{1})| + |f_{2}(v_{1}) - f_{2}(w_{1})| \lesssim (K + HM)2^{-m}. \end{displaymath} 
Since $|I| \sim_{N} 2^{-m}$, and $M \leq N$, we have $y \in rI$ for some $r \sim_{K,N} 1$, and it finally follows from \eqref{form52} that
\begin{displaymath} |q_{1}(y) - q_{2}(y)| \lesssim_{K,N,\Sigma} \epsilon \cdot 2^{-n} \quad \text{and} \quad |\dot{q}_{1}(y) - \dot{q}_{2}(y)| \lesssim_{K,N,\Sigma} \epsilon \cdot 2^{-m}. \end{displaymath} 
The correct bound for $|\ddot{q}_{1}(y) - \ddot{q}_{2}(y)|$ already follows from \eqref{form52}. Since $N \sim_{M,\Sigma} 1$, this completes the proof of the proposition. \end{proof}

The final result in this section is virtually a restatement of Corollary \ref{cor5}; 
\begin{cor}\label{cor6} Let $\ell_{1} < \ell_{2} < \ell_{3} < \ell_{4} < \ell_{5}$ be consecutive lines in $\mathcal{L}_{n}$, $n \geq 1$, such that $\ell_{2},\ell_{4} \in \mathcal{L}_{n - 1}$. Assume that one of the lines $\ell_{j}$ contains the corner "$w$" of a rectangle $Q \in \mathcal{T}_{\ceil{n/2}}$. Then, the dyadic parent of $Q$ has a corner $\widehat{w} \in \ell_{2} \cup \ell_{4}$ which lies in $\mathcal{C}_{\ceil{n/2}} \cap \mathcal{C}_{\ceil{(n - 1)/2}}$. \end{cor}

\begin{proof} Only the last statement was not part of Corollary \ref{cor5}. If $n$ is odd, then $\ceil{(n - 1)/2} = \ceil{n/2} - 1$, and the corners of the dyadic parent $\hat{Q} \in \mathcal{T}_{\ceil{n/2} - 1}$ of $Q$ lie in $\mathcal{C}_{\ceil{(n - 1)/2}}$ by definition (here we used $n \geq 1$, so the dyadic parent of $Q$ still lies in $\mathcal{T}$). Moreover, by the coherence axiom \nref{T3}, all the children of $\hat{Q}$ lie in $\mathcal{T}_{\ceil{n/2}}$, and consequently the corners of $\hat{Q}$ are corners of some rectangles in $\mathcal{T}_{\ceil{n/2}}$. 

If $n$ is even, then $\ceil{(n - 1)/2} = \ceil{n/2}$, hence the siblings of $Q \in \mathcal{T}_{\ceil{n/2}}$ lie simultaneously in $\mathcal{T}_{\ceil{n/2}}$ and $\mathcal{T}_{\ceil{(n - 1)/2}}$. In particular, the corners of $\hat{Q}$ lie in $\mathcal{C}_{\ceil{n/2}} \cap \mathcal{C}_{\ceil{(n - 1)/2}}$. \end{proof}

\begin{lemma}\label{lemma12} Let $n \geq 0$ and $Q \in \mathcal{D}_{\ceil{n/2}}$. Then at least one of the horizontal edges of $Q$ is contained on a line in $\mathcal{L}_{n}$. 
\end{lemma}

\begin{proof} The height of the rectangle $Q$ is $4^{-\ceil{n/2}} \in \{2^{-n},2^{-n - 1}\}$, depending on the parity of $n$. If the height is $2^{-n}$, then both horizontal edges of $Q$ are contained on lines in $\mathcal{L}_{n}$. If the height is $2^{-n - 1}$, then both horizontal edges of $Q$ are contained on lines in $\mathcal{L}_{n + 1}$, and one observes that every second line in $\mathcal{L}_{n + 1}$ is contained in $\mathcal{L}_{n}$. \end{proof}

We then begin to verify the axioms \eqref{X1}-\eqref{X4} in order. While verifying axiom (Xi) for $i \in \{1,\ldots,5\}$ and $\tau_{n}$, we may inductively use all of the axioms (Xj) for $\tau_{n - 1}$. Moreover, to verify axioms \eqref{X1} for $\tau_{n}$, we will use axioms \eqref{X3} and \eqref{X4} for $\tau_{n}$. However, the proofs of axioms \eqref{X3},\eqref{X4} for $\tau_{n}$ will not use axiom \eqref{X1} for $\tau_{n}$, so there is no circular reasoning. 

\subsubsection{Verifying axiom \eqref{X1}} Fix $\ell \in \mathcal{L}_{n}$. The claim is, then, that
\begin{equation}\label{form57} |\dot{\tau}_{n}(\ell)(y_{1}) - \dot{\tau}_{n}(\ell)(y_{2})| \leq 2\Sigma |y_{1} - y_{2}|, \qquad y_{1},y_{2} \in \R, \end{equation}
and we assume inductively that this holds for $\tau_{n - 1}(\ell)$, $\tau_{n - 1}(\ell^{\uparrow})$, and $\tau_{n - 1}(\ell_{\downarrow})$. An easy case occurs if $\{y_{1},y_{2}\} \subset \R \, \setminus \, U^{\sharp}(\ell)$, so $\varphi_{\ell} \equiv 1$ on a neighbourhood of $y_{1},y_{2}$ (this is why we defined $U,U^{\sharp}$ as closed sets). Then, depending on whether $\ell \in \mathcal{L}_{n - 1}$ or $\ell \in \mathcal{L}_{n} \, \setminus \, \mathcal{L}_{n - 1}$, we have either
\begin{displaymath} \dot{\tau}_{n}(\ell)(y_{1}) - \dot{\tau}_{n}(\ell)(y_{2}) = \dot{\tau}_{n - 1}(\ell)(y_{1}) - \dot{\tau}_{n - 1}(\ell)(y_{2}) \end{displaymath}
or
\begin{displaymath} \dot{\tau}_{n}(\ell)(y_{1}) - \dot{\tau}_{n}(\ell)(y_{2}) = \tfrac{1}{2}[\dot{\tau}_{n - 1}(\ell^{\uparrow})(y_{1}) - \dot{\tau}_{n - 1}(\ell^{\uparrow})(y_{2})] + \tfrac{1}{2}[\dot{\tau}_{n - 1}(\ell^{\downarrow})(y_{1}) - \dot{\tau}_{n - 1}(\ell^{\downarrow})(y_{2})]. \end{displaymath} 
In both cases, \eqref{form57} follows from induction and the triangle inequality.

Let us then assume that $y_{1} \in U^{\sharp}(\ell)$, for example. Then, it does not matter if $y_{2} \in U^{\sharp}(\ell)$ or $y_{2} \in \R \, \setminus \, U^{\sharp}(\ell)$, and the verification of \eqref{form57} will not use the inductive assumption \eqref{X1}. Since $y_{1} \in U^{\sharp}(\ell)$, we have
\begin{equation}\label{form59} y_{1} \in \kappa_{0}I_{n}(w) = B(f_{2}(w),\kappa_{0}K2^{-n/2}) \end{equation}
for some $w \in \mathcal{C}_{m}(\ell) \cup \mathcal{C}_{m}(\ell^{\uparrow}) \cup \mathcal{C}_{m}(\ell_{\downarrow})$. This means that $w \in \ell \cup \ell^{\uparrow} \cup \ell_{\downarrow}$ and $w \in \mathcal{C}(Q)$ for some $Q \in \mathcal{T}_{m} = \mathcal{T}_{\ceil{n/2}}$. There are three different cases to consider:
\begin{equation}\label{form61} |y_{1} - y_{2}| \geq \tfrac{1}{10}K \quad \text{or} \quad \tfrac{1}{5}K \cdot 2^{-n/2} < |y_{1} - y_{2}| < \tfrac{1}{10}K \quad \text{or} \quad |y_{1} - y_{2}| \leq \tfrac{1}{5}K\cdot 2^{-n/2}.   \end{equation} 
We handle these cases in order. First, assume $|y_{1} - y_{2}| \geq \tfrac{1}{10}K$. Let $\ell =: \ell_{n} \leq \ell_{n - 1} \leq \ldots \leq \ell_{0}$ be a sequence of horizontal lines such that $\ell_{k} \in \mathcal{L}_{k}$ and $\ell_{k + 1} \in \{\ell_{k},\ell_{k}^{\uparrow}\}$. Then, by axioms \eqref{X3} and \eqref{X4} for $\tau_{0},\ldots,\tau_{n}$, we have
\begin{align*} |\dot{\tau}_{n}(\ell_{n})(y_{j}) - \dot{\tau}_{0}(\ell_{0})(y_{j})| & \leq \sum_{k = 0}^{n - 1} |\dot{\tau}_{k + 1}(\ell_{k + 1})(y_{j}) - \dot{\tau}_{k}(\ell_{k})(y_{j})|\\
& \lesssim \delta \sum_{k = 0}^{n - 1} 2^{-k/2} \sim \delta \leq \frac{10\delta}{K}|y_{1} - y_{2}|, \qquad j \in \{1,2\}. \end{align*} 
Since $K > 1$ in Proposition \ref{t:construction} is chosen small in a manner depending on both $\delta$ and $\Sigma$, it follows that
\begin{displaymath} |\dot{\tau}_{n}(\ell)(y_{2}) - \dot{\tau}_{n}(\ell)(y_{1})| \leq |\dot{\tau}_{0}(\ell_{0})(y_{2}) - \dot{\tau}_{0}(\ell_{0})(y_{1})| + \tfrac{A\delta}{K}|y_{1} - y_{2}| \leq 2\Sigma |y_{2} - y_{1}|. \end{displaymath} 
Here we also used the special definition of $\tau_{0}$, which implies, as seen in \eqref{form85}, that $\dot{\tau}_{0}(\ell)$ is $\Sigma$-Lipschitz, not just $2\Sigma$-Lipschitz, for all $\ell \in \mathcal{L}_{0}$. 

Recalling the case distinction \eqref{form61}, we next assume $\tfrac{1}{5}K \cdot 2^{-n/2} < |y_{1} - y_{2}| < \tfrac{1}{10}K$. Note that this forces $n \geq 2$. We then claim that there exists $0 \leq \hat{n} \leq n - 2$, and a rectangle $\hat{Q} \in \mathcal{T}_{\ceil{\hat{n}/2}}$ with $\hat{Q} \supset Q$ such that
\begin{equation}\label{form58} |y_{1} - y_{2}| \geq 2^{-\hat{n}/2} \quad \text{and} \quad y_{1},y_{2} \in B(f_{2}(\widehat{w}),K2^{-\hat{n}/2}) \end{equation} 
for some $\widehat{w} \in \mathcal{C}(\hat{Q}) \subset \mathcal{C}_{\ceil{\hat{n}/2}}$. Let $n_{1} \in \{0,\ldots,n\}$ be the greatest index with the property $|y_{1} - y_{2}| \leq \tfrac{1}{5}K2^{-n_{1}/2}$. Then $1 \leq n_{1} < n$ by our case assumption, so $\hat{n} := n_{1} - 1 \in \{0,\ldots,n - 2\}$, and 
\begin{displaymath} |y_{1} - y_{2}| > \tfrac{1}{5}K2^{-(n_{1} + 1)/2} \geq 2^{-\hat{n}/2}. \end{displaymath}
So, the first assertion of \eqref{form58} holds.

To prove the second assertion, let $Q_{1} \in \mathcal{T}_{\ceil{n_{1}/2}}$ and $\hat{Q} \in \mathcal{T}_{\ceil{\hat{n}/2}}$ be rectangles with $Q \subset Q_{1} \subset \hat{Q}$. Let $w_{1} \in \mathcal{C}(Q_{1})$ be an arbitrary corner. We also choose a corner $\widehat{w} \in \mathcal{C}(\hat{Q})$, but with some care: recall from Lemma \ref{lemma12} that one of the horizontal edges of $\hat{Q} \in \mathcal{D}_{\ceil{\hat{n}/2}}$ lies on a line $\ell_{\hat{n}} \in \mathcal{L}_{\hat{n}}$. We choose $\widehat{w}$ to be one of the two corners of $\hat{Q}$ on $\ell_{\hat{n}}$.

Certainly $d_{\mathrm{par}}(w_{1},\widehat{w}) \leq 100 \cdot 2^{-n_{1}/2}$. Then, \eqref{form59} and Lemma \ref{lemma11} (applied with $w_{1}$ and $w_{2} := \widehat{w}$) first imply that
\begin{displaymath} y_{1} \in B(f_{2}(w),\kappa_{0}K2^{-n/2}) \subset B(f_{2}(w_{1}),\kappa_{0}K2^{-n_{1}/2}) \subset B(f_{2}(\widehat{w}),\tfrac{4}{5}K2^{-\hat{n}/2}),  \end{displaymath} 
and next
\begin{displaymath} |y_{2} - f_{2}(w_{1})| \leq |y_{1} - y_{2}| + \tfrac{4}{5}K2^{-\hat{n}/2} \leq \tfrac{1}{5}K2^{-\hat{n}/2} + \tfrac{4}{5}K2^{-\hat{n}/2} = K2^{-\hat{n}/2}.  \end{displaymath}
This verifies \eqref{form58} with $\widehat{w},\hat{Q}$, and $\hat{n}$, as above.

This information enables us to use the inductive assumption \eqref{X5} on the characteristics of generation $0 \leq \hat{n} \leq n - 2$. Let us construct a sequence of horizontal lines $\ell_{n},\ldots, \ell_{\hat{n} - 1}$: 
\begin{itemize}
\item[(a)] $\ell_{n} \in \{\ell,\ell^{\uparrow},\ell_{\downarrow}\}$ is the line containing $w$, so $\ell_{n}$ intersects the closures of $Q$ and $\hat{Q} \supset Q$,
\item[(b)]  $\ell_{k} \in \mathcal{L}_{k}$ for all $n \leq k \leq \hat{n} - 1$, and $\ell_{k - 1} \in \{\ell_{k},\ell_{k}^{\uparrow},\ell_{k\downarrow}\}$ for $\hat{n} + 2 \leq k \leq n$ (this notation refers to the $\mathcal{L}_{k}$-neighbours of $\ell_{k}$),
\item[(c)] Every line $\ell_{k}$, $n \leq k \leq \hat{n} - 1$, intersects the closure of $\hat{Q}$,
\end{itemize}
The line $\ell_{n}$ is dictated by (a). Assume that $\ell_{k} \in \mathcal{L}_{k}$ has already been chosen for some $\hat{n} + 2 \leq k \leq n$. Then either $\ell_{k} \in \mathcal{L}_{k - 1}$, and we simply set $\ell_{k - 1} := \ell_{k}$, or otherwise both $\mathcal{L}_{k}$-neighbours of $\ell_{k}$ are in $\mathcal{L}_{k - 1}$. Since the height of the rectangle $\hat{Q} \in \mathcal{T}_{\ceil{\hat{n}/2}} \subset \mathcal{D}_{\ceil{\hat{n}/2}}$ is $4^{-\ceil{\hat{n}/2}} \in \{2^{-\hat{n}},2^{-(\hat{n} + 1)}\}$, depending on the parity of $\hat{n}$, and the vertical separation of $\mathcal{L}_{k}$-neighbours is $2^{-k} \leq 2^{-\hat{n} + 2}$, one of these neighbours meets the closure of $\hat{Q}$. So, $\ell_{k - 1} \in \mathcal{L}_{k - 1}$ may be chosen as in (b) and (c) for $\hat{n} + 2 \leq k \leq n$. 

We earlier defined $\ell_{\hat{n}} \in \mathcal{L}_{\hat{n}}$ to be a line containing the corner $\widehat{w} \in \mathcal{C}(\hat{Q})$. According to (b)-(c), both $\ell_{\hat{n} + 1} \in \mathcal{L}_{\hat{n} + 1}$ and $\ell_{\hat{n}} \in \mathcal{L}_{\hat{n}} \subset \mathcal{L}_{\hat{n} + 1}$ intersect the closure of $\hat{Q}$. If the height of $\hat{Q}$ is $2^{-\hat{n} + 1}$, it follows that $\ell_{\hat{n}} \in \{\ell_{\hat{n} + 1},\ell_{\hat{n} + 1}^{\uparrow},\ell_{\hat{n} + 1\downarrow}\}$. Otherwise, the height of $\hat{Q}$ is $2^{-\hat{n}}$, and $\ell_{\hat{n}}$ may be up to two steps away from $\ell_{\hat{n} + 1}$ in $\mathcal{L}_{\hat{n} + 1}$. In both cases, axioms \eqref{X3}-\eqref{X4} yield 
\begin{equation}\label{form89} \|\dot{\tau}_{k}(\ell_{k}) - \dot{\tau}_{k + 1}(\ell_{k + 1})\|_{L^{\infty}(\R)} \lesssim \delta \cdot 2^{-k/2}, \qquad \hat{n} \leq k \leq n - 1. \end{equation}
From \eqref{form89}, we deduce that for both $j \in \{1,2\}$,
\begin{align} |\dot{\tau}_{n}(\ell)(y_{j}) - \dot{\tau}_{\hat{n}}(\ell_{\hat{n}})(y_{j})| & \leq \sum_{k = \hat{n}}^{n - 1} |\dot{\tau}_{k + 1}(\ell_{k + 1})(y_{j}) - \dot{\tau}_{k}(\ell_{k})(y_{j})| \notag \\
&\label{form90} \leq \delta \sum_{k = \hat{n}}^{n - 1} 2^{-k/2} \sim \delta \cdot 2^{-\hat{n}/2} \leq \delta |y_{1} - y_{2}|, \end{align} 
using the lower bound in \eqref{form58} in the final inequality. Next, we estimate the difference $|\dot{\tau}_{\hat{n}}(\ell_{\hat{n}})(y_{2}) - \dot{\tau}_{\hat{n}}(\ell_{\hat{n}})(y_{1})|$ using the inclusion $y_{1},y_{2} \in B(f_{2}(\widehat{w}),K2^{-\hat{n}/2})$ from \eqref{form58}, and also $\widehat{w} \in \mathcal{C}_{\ceil{\hat{n}/2}}(\ell_{\hat{n}})$. With these facts in hand, axiom \eqref{X5} lets us compare $\dot{\tau}_{\hat{n}}(\ell_{\hat{n}})$ to the $(\hat{Q},\ell_{\hat{n}})$-approximate quadric $q := q_{\hat{Q},\ell_{\hat{n}}}$: 
\begin{displaymath} |\dot{\tau}_{\hat{n}}(\ell_{\hat{n}})(y_{2}) - \dot{\tau}_{\hat{n}}(\ell_{\hat{n}})(y_{1})| \leq |\dot{q}(y_{2}) - \dot{q}(y_{1})| + o(\epsilon) \cdot 2^{-\hat{n}/2} \leq |\dot{q}(y_{2}) - \dot{q}(y_{1})| + o(\epsilon) \cdot |y_{1} - y_{2}|. \end{displaymath} 
Since $\|\ddot{q}\|_{L^{\infty}} \leq \Sigma$ by the assumption that $\hat{Q} \in \mathcal{T}$ is an $(\epsilon,H,\Sigma)$-vertical rectangle, we combine the estimate above with \eqref{form90} to deduce
\begin{displaymath} |\dot{\tau}_{n}(\ell_{n})(y_{2}) - \dot{\tau}_{n}(\ell_{n})(y_{1})| \leq \Sigma |y_{1} - y_{2}| + A(\delta + o(\epsilon))|y_{1} - y_{2}|, \end{displaymath}
where $A \geq 1$ is an absolute constant. Finally, the same estimate holds for $|\dot{\tau}_{n}(\ell)(y_{2}) - \dot{\tau}_{n}(\ell)(y_{1})|$, since $\ell_{n} \in \{\ell,\ell^{\uparrow},\ell_{\downarrow}\}$, and axiom \eqref{X3} implies $\|\dot{\tau}_{n}(\ell) - \dot{\tau}_{n}(\ell_{n})\|_{L^{\infty}(\R)} \leq \delta \cdot 2^{-n/2}$. Choosing $\delta,\epsilon > 0$ small enough, depending on $\Sigma$, we see that $|\dot{\tau}_{n}(\ell)(y_{2}) - \dot{\tau}_{n}(\ell)(y_{1})| \leq 2\Sigma|y_{1} - y_{2}|$, as desired. This completes the proof of axiom \eqref{X1} in the case where $\tfrac{1}{5}K \cdot 2^{-n/2} < |y_{1} - y_{2}| < \tfrac{1}{10}K$.

Recalling the case distinction in \eqref{form61}, it remains to treat the case $|y_{1} - y_{2}| \leq \tfrac{1}{5}K \cdot 2^{-n/2}$. Axioms \eqref{X3}-\eqref{X4} will not be used here. Recall from \eqref{form59} that $w \in \mathcal{C}_{m}(\ell) \cup \mathcal{C}_{m}(\ell^{\uparrow}) \cup \mathcal{C}_{m}(\ell_{\downarrow})$ is the corner of some rectangle $Q \in \mathcal{T}_{m}$ which satisfies $y_{1} \in B(f_{2}(w),\kappa_{0}K2^{-n/2})$. Here $m = \ceil{n/2}$, as always. We will next apply Corollary \ref{cor6} to find a "parent" corner point $\widehat{w}$ close to $w$. The details are slightly differently depending on whether $\ell \in \mathcal{L}_{n} \, \setminus \, \mathcal{L}_{n - 1}$ or $\ell \in \mathcal{L}_{n - 1}$. 

If $\ell \in \mathcal{L}_{n} \, \setminus \, \mathcal{L}_{n - 1}$, let $\ell_{1} < \ldots < \ell_{5}$ be the five consecutive lines of $\mathcal{L}_{n}$ with $\ell_{2} = \ell_{\downarrow} \in \mathcal{L}_{n - 1}$ and $\ell_{4} = \ell^{\uparrow} \in \mathcal{L}_{n - 1}$. Then, recalling that $w \in \mathcal{C}_{m}(\ell) \cup \mathcal{C}_{m}(\ell^{\uparrow}) \cup \mathcal{C}_{m}(\ell_{\downarrow})$ is the corner of $Q \in \mathcal{T}_{m}$, Corollary \ref{cor6} produces a corner $\widehat{w} \in \mathcal{C}_{m} \cap \mathcal{C}_{\ceil{(n - 1)/2}}$ of the dyadic parent of $Q$ which lies on $\ell^{\uparrow} \cup \ell_{\downarrow}$. 

If, on the other hand, $\ell \in \mathcal{L}_{n - 1}$, then choose the five lines in $\mathcal{L}_{n}$ so that $\ell \in \{\ell_{2},\ell_{4}\}$. Then $\ell = \ell_{2}$ and $\ell_{4}$ are consecutive in $\mathcal{L}_{n - 1}$, and Corollary \ref{cor6} produces a corner $\widehat{w} \in \mathcal{C}_{m} \cap \mathcal{C}_{\ceil{(n - 1)/2}}$ of the dyadic parent of $Q$ which lies on $\ell \cup \ell_{4}$.

In both cases, $d_{\mathrm{par}}(w,\widehat{w}) \leq 100 \cdot 2^{-n/2}$, so Lemma \ref{lemma11} implies
\begin{displaymath} y_{1} \in B(f_{2}(w),\kappa_{0}K2^{-n/2}) \subset B(f_{2}(\widehat{w}),\tfrac{4}{5}K \cdot 2^{-(n - 1)/2}). \end{displaymath}
Consequently $|y_{2} - f_{2}(\widehat{w})| \leq \tfrac{1}{5}K \cdot 2^{-n/2} + \tfrac{4}{5}K \cdot 2^{-(n - 1)/2} < K2^{-(n - 1)/2}$, so in fact
\begin{equation}\label{form64} [y_{1},y_{2}] \subset B(f_{2}(\widehat{w}),K2^{-(n - 1)/2}). \end{equation}
We are now in a position to apply the inductive hypothesis \eqref{X5} to the quadric $\tau_{n - 1}(\ell)$ (in case $\ell \in \mathcal{L}_{n - 1}$) or both quadrics $\tau_{n - 1}(\ell^{\uparrow})$ and $\tau_{n - 1}(\ell_{\downarrow})$ (in case $\ell \in \mathcal{L}_{n} \, \setminus \, \mathcal{L}_{n - 1}$). We treat the cases separately for clarity. Let us first assume $\ell \in \mathcal{L}_{n - 1}$, so $\widehat{w} \in \ell \cup \ell_{4}$, where $\ell$ and $\ell_{4}$ are consecutive in $\mathcal{L}_{n - 1}$. Let $\hat{Q} \in \mathcal{T}_{\ceil{(n - 1)/2}}$ be a rectangle with corner $\widehat{w}$. Let $q := q_{\hat{Q},\ell}$ be the $(\hat{Q},\ell)$-approximate quadric. Then, axiom \eqref{X5} applied inductively to $\tau_{n - 1}(\ell)$ implies 
\begin{equation}\label{form62} \tau_{n - 1}(y) \approx_{n - 1} q(y), \qquad y \in [y_{1},y_{2}]. \end{equation}
Similar reasoning works if $\ell \in \mathcal{L}_{n} \, \setminus \, \mathcal{L}_{n - 1}$, so $\widehat{w} \in \ell^{\uparrow} \cup \ell_{\downarrow}$. Then, \eqref{X5} is applied inductively to both $\tau_{n - 1}(\ell^{\uparrow})$ and $\tau_{n - 1}(\ell_{\downarrow})$, and to the $(\hat{Q},\ell^{\uparrow})$- and $(\hat{Q},\ell_{\downarrow})$-approximate quadrics $q^{\uparrow}$ and $q_{\downarrow}$; the meaning of $\hat{Q} \in \mathcal{T}_{\ceil{(n - 1)/2}}$ is the same as before. The results are
\begin{equation}\label{form63} \tau_{n - 1}(\ell^{\uparrow})(y) \approx_{n - 1} q^{\uparrow}(y) \quad \text{and} \quad \tau_{n - 1}(\ell_{\downarrow})(y) \approx_{n - 1} q_{\downarrow}(y), \qquad y \in [y_{1},y_{2}].  \end{equation}
To conclude the proof of axiom \eqref{X1}, we will use \eqref{form62}-\eqref{form63}, and additionally the following claim; in the two possible definitions
\begin{displaymath} \begin{cases} \sum_{I \in \mathcal{I}(\ell)} \varphi_{I,\ell}(y)q_{I,\ell}(y) + \varphi_{\ell}(y) \cdot \tau_{n - 1}(\ell)(y), \\\sum_{I \in \mathcal{I}(\ell)} \varphi_{I,\ell}(y)q_{I,\ell}(y) + \tfrac{1}{2}\varphi_{\ell}(y) \cdot [\tau_{n - 1}(\ell^{\uparrow})(y) + \tau_{n - 1}(\ell_{\downarrow})(y)] \end{cases} \end{displaymath}
of $\tau_{n}(\ell)(y)$, we have $q_{I,\ell}(y) \approx_{n} q(y)$ for all $y \in [y_{1},y_{2}]$ and for all $I \in \mathcal{I}(\ell)$ such that $y \in \spt \varphi_{I,\ell}$. To accomplish this, we appeal to Proposition \ref{prop4}. Note that if $y \in \spt \varphi_{I,\ell}$ for some $y \in [y_{1},y_{2}]$, then 
\begin{displaymath} y \in \kappa_{0}I_{n}(w_{I}) = \bar{B}(f_{2}(w_{I}),\kappa_{0}K2^{-n/2}) \subset B(f_{2}(w_{I}),100K2^{-n/2}) \end{displaymath}
for some $w_{I} \in \mathcal{C}_{m}(\ell_{\downarrow}) \cup \mathcal{C}_{m}(\ell) \cup \mathcal{C}_{m}(\ell^{\uparrow})$, and $q_{I,\ell} = q_{Q_{I},\ell}$ is the $(Q_{I},\ell)$-approximate quadric for some rectangle $Q_{I} \in \mathcal{T}_{m} = \mathcal{T}_{\ceil{n/2}}$ with corner $w_{I}$. On the other hand, from \eqref{form64} we know that $y \in [y_{1},y_{2}] \subset B(f_{2}(\widehat{w}),100K2^{-n/2})$. Since the vertical separation between the lines containing $w_{I}$ and $\widehat{w}$ is certainly $\leq 100 \cdot 2^{-n}$, Proposition \ref{prop4} is applicable with $(Q_{1},w_{1}) := (Q_{I},w_{I}) \in \mathcal{T}_{\ceil{n/2}} \times \mathcal{C}(Q_{I})$ and $(Q_{2},w_{2}) := (\hat{Q},\widehat{w}) \in \mathcal{T}_{\ceil{(n - 1)/2}} \times \mathcal{C}(\hat{Q})$, and implies 
\begin{equation}\label{form68} q_{I,\ell}(y) = q_{Q_{I},\ell}(y) \approx_{n} q_{\hat{Q},\ell}(y) = q(y), \qquad y \in [y_{1},y_{2}]. \end{equation}
Now, all the pieces are in place to prove \eqref{X1} in the final case $|y_{1} - y_{2}| \leq \tfrac{1}{5}K \cdot 2^{-n/2}$. Assume, first, that $\ell \in \mathcal{L}_{n - 1}$, so
\begin{displaymath} \tau_{n}(\ell)(y) = \sum_{I \in \mathcal{I}(\ell)} \varphi_{I,\ell}(y)q_{I,\ell}(y) + \varphi_{\ell}(y) \cdot \tau_{n - 1}(\ell)(y). \end{displaymath} 
Combining \eqref{form62} and \eqref{form68}, we see that $q_{I,\ell}(y) \approx_{n} q(y) \approx_{n - 1} \tau_{n - 1}(\ell)(y)$ for all $I \in \mathcal{I}(\ell)$ such that $y \in \spt \varphi_{I,\ell}$. Now Corollary \ref{cor7} implies
\begin{equation}\label{form78} \tau_{n}(\ell)(y) \approx_{n - 1} q(y), \qquad y \in [y_{1},y_{2}]. \end{equation}
In particular, since $\|\ddot{q}\|_{L^{\infty}} \leq \Sigma$, we have $|\ddot{\tau}_{n}(\ell)(y)| \leq \Sigma + o(\epsilon)$ for $y \in [y_{1},y_{2}]$, and \eqref{X1} then follows by choosing $\epsilon > 0$ small enough, depending on $K,M,\Sigma$.

Assume finally that $\ell \in \mathcal{L}_{n} \, \setminus \, \mathcal{L}_{n - 1}$, so 
\begin{displaymath} \tau_{n}(\ell)(y) = \sum_{I \in \mathcal{I}(\ell)} \varphi_{I,\ell}(y)q_{I,\ell}(y) + \tfrac{1}{2}\varphi_{\ell}(y) \cdot [\tau_{n - 1}(\ell^{\uparrow})(y) + \tau_{n - 1}(\ell_{\downarrow})(y)]. \end{displaymath}
We claim that
\begin{equation}\label{form77} \begin{cases} |\dot{\tau}_{n}(\ell)(y) - \dot{q}(y)| \lesssim_{M} \left(\tfrac{1}{K} + o(\epsilon)\right) 2^{-n/2}, \\  |\ddot{\tau}_{n}(\ell)(y) - \ddot{q}(y)| \lesssim_{M} \tfrac{1}{K} + o(\epsilon), \end{cases} \quad y \in [y_{1},y_{2}]. \end{equation}
Only the latter estimate will be needed here, but the first one is useful later. Together with $\|\ddot{q}\|_{L^{\infty}(\R)} \leq \Sigma$, it implies $\|\ddot{\tau}_{n}(\ell)\|_{L^{\infty}([y_{1},y_{2}])} \leq \Sigma + C_{M}(o(\epsilon) + \tfrac{1}{K})$. Consequently $|\dot{\tau}_{n}(\ell)(y_{2}) - \dot{\tau}_{n}(\ell)(y_{1})| \leq 2\Sigma$ if we choose $K$ large enough depending on $M,\Sigma$, and $\epsilon > 0$ small enough depending on $K,M,\Sigma$, as we assumed in Proposition \ref{t:construction}.

To prove \eqref{form77}, we will use Lemma \ref{lemma13} with $\{\varphi_{I}\}_{I \in \mathcal{I}} := \{\varphi_{I,\ell}\}_{I \in \mathcal{I(\ell)}} \cup \{\tfrac{1}{2}\varphi_{\ell}\} \cup \{\tfrac{1}{2}\varphi_{\ell}\}$,
\begin{displaymath} p_{I} \in \{\varphi_{I,\ell},\tau_{n - 1}(\ell^{\uparrow}),\tau_{n - 1}(\ell_{\downarrow})\} \quad \text{and} \quad q_{I} \equiv q. \end{displaymath}
We start by recapping from \eqref{form63} that $\tau_{n - 1}(\ell^{\uparrow})(y) \approx_{n - 1} q^{\uparrow}(y)$ and $\tau_{n - 1}(\ell_{\downarrow})(y) \approx_{n - 1} q{\downarrow}(y)$ for all $y \in [y_{1},y_{2}]$, where $q^{\uparrow}$ and $q_{\downarrow}$ are the $(\hat{Q},\ell^{\uparrow})$ and $(\hat{Q},\ell_{\downarrow})$-approximate quadrics, respectively. The quadrics $q,q^{\uparrow},q^{\downarrow}$ are all associated with the common $(\epsilon,H,\Sigma)$-vertical rectangle $\hat{Q}$, so they only differ by a constant. In particular, the first and second derivatives of $q,q^{\uparrow},q^{\downarrow}$ agree. Also, it is not hard to see that
\begin{displaymath} \|q - q^{\uparrow}\|_{L^{\infty}(\R)}, \|q - q_{\downarrow}\|_{L^{\infty}(\R)} \lesssim_{M} 2^{-n}, \end{displaymath}
because the graphs of $q,q^{\uparrow},q_{\downarrow}$ are $\Pi$-projections of horizontal lines on $\V_{\hat{Q}}$ which best approximate $f(\ell \cap H\hat{Q})$, $f(\ell^{\uparrow} \cap H\hat{Q})$, and $f(\ell_{\downarrow} \cap H\hat{Q})$ respectively. We will derive much more precise information at \eqref{form47a}, so we refer the reader there for more details. Consequently, for $\hat{\ell} \in \{\ell^{\uparrow},\ell_{\downarrow}\}$, and the appropriate choice of $\hat{q} \in \{q^{\uparrow},q_{\downarrow}\}$, we have
\begin{displaymath} \begin{cases} |\tau_{n - 1}(\hat{\ell})(y) - q(y)| \leq |\tau_{n - 1}(\hat{\ell})(y) - \hat{q}(y)| + \|\hat{q} - q\|_{L^{\infty}(\R)} \lesssim_{M} 2^{-n}, \\ |\dot{\tau}_{n - 1}(\hat{\ell})(y) - \dot{q}(y)| = |\dot{\tau}_{n - 1}(\hat{\ell})(y) - \dot{\hat{q}}(y)| = o(\epsilon) \cdot 2^{-n/2},\\ |\ddot{\tau}_{n - 1}(\hat{\ell})(y) - \ddot{q}(y)| = |\ddot{\tau}_{n - 1}(\hat{\ell})(y) - \ddot{\hat{q}}(y)| = o(\epsilon), \qquad \qquad y \in [y_{1},y_{2}]. \end{cases} \end{displaymath} 
With this information (and of course \eqref{form68}), and recalling the estimates \eqref{form70} for the first and second derivatives of the functions $\varphi_{I,\ell}$ and $\varphi_{\ell}$, we are in a position to apply Lemma \ref{lemma13} with constants $c_{1},c_{2} \sim K^{-1}$, $\epsilon_{1} \lesssim_{M} 1$, and $\epsilon_{2},\epsilon_{3} = o(\epsilon)$. The estimates \eqref{form77} follow immediately, and the proof of axiom \eqref{X1} is complete.

\subsubsection{Verying axioms \eqref{X2}-\eqref{X3}} Recall that axiom \eqref{X2} asks to prove that if $\ell_{1},\ell_{2} \in \mathcal{L}_{n}$ satisfy $\ell_{2} > \ell_{1}$, then 
\begin{equation}\label{form38} \theta (\ell_{2} - \ell_{1}) \leq \tau_{n}(\ell_{2})(y) - \tau_{n}(\ell_{1})(y) \leq \theta^{-1} (\ell_{2} - \ell_{1}), \qquad y \in \R. \end{equation} 
By trivial induction, it suffices to prove these inequalities for any pair of \textbf{consecutive} lines $\ell_{1},\ell_{2} \in \mathcal{L}_{n}$ with $\ell_{2} > \ell_{1}$. Therefore $\ell_{2} - \ell_{1} = 2^{-n}$ (this notation meant that $\pi_{2}(\ell_{2}) - \pi_{2}(\ell_{1}) = 2^{-n}$). So, fix consecutive lines $\ell_{1},\ell_{2} \in \mathcal{L}_{n}$ with $\ell_{2} > \ell_{1}$. Then, either 
\begin{equation}\label{form30} \ell_{1} \in \mathcal{L}_{n - 1} \quad \text{and} \quad \ell_{2} \in \mathcal{L}_{n} \, \setminus \, \mathcal{L}_{n - 1}, \end{equation}
or the same holds with the roles of $\ell_{1},\ell_{2}$ reversed. We assume with no loss of generality that \eqref{form30} holds. Under these assumptions, we will, in addition to \eqref{form38}, prove that
\begin{equation}\label{form76} |\dot{\tau}_{n}(\ell_{2})(y) - \dot{\tau}_{n}(\ell_{1})(y)| \leq \delta \cdot 2^{-n/2}, \qquad y \in \R. \end{equation}
This will establish axiom \eqref{X3}.

To get started, we introduce separate notation for the upper $\mathcal{L}_{n}$-neighbour of $\ell_{2}$:
\begin{displaymath} \ell_{3} := \ell_{2}^{\uparrow} \in \mathcal{L}_{n - 1}. \end{displaymath}
Thus, $\ell_{1} < \ell_{2} < \ell_{3}$, where $\ell_{j}$ is consecutive to $\ell_{j + 1}$ in $\mathcal{L}_{n}$, and $\ell_{1},\ell_{3}$ are neighbours in $\mathcal{L}_{n - 1}$. For $j \in \{1,2,3\}$, we abbreviate 
\begin{displaymath} \varphi_{I,j} := \varphi_{I,\ell_{j}}, \quad \varphi_{j} := \varphi_{\ell_{j}} = 1 - \sum_{I \in \mathcal{I}(\ell_{j})} \varphi_{I,j}, \quad \text{and} \quad q_{I,j} := q_{I,\ell_{j}}. \end{displaymath}
Now, to prove \eqref{form38} and \eqref{form76}, fix $y \in \R$. An easy case occurs if $y \in \R \, \setminus \, [U^{\sharp}(\ell_{1}) \cup U^{\sharp}(\ell_{2})]$. Then $\varphi_{I,j} \equiv 0$ and $\varphi_{j} \equiv 1$ for $j \in \{1,2\}$ and $I \in \mathcal{I}(\ell_{j})$, so
\begin{displaymath} \tau_{n}(\ell_{2})(y) = \tfrac{1}{2} \cdot [\tau_{n - 1}(\ell_{1}) + \tau_{n - 1}(\ell_{3})] \quad \text{and} \quad \tau_{n}(\ell_{1})(y) = \tau_{n - 1}(\ell_{1}) \end{displaymath} 
by the definitions \eqref{form28}-\eqref{form31}, and recalling \eqref{form30}. Using the inductive hypothesis \eqref{X2} for $\tau_{n - 1}(\ell_{3}) - \tau_{n - 1}(\ell_{1})$, we find
\begin{displaymath} \tau_{n}(\ell_{2})(y) - \tau_{n}(\ell_{1})(y) = \tfrac{1}{2} \cdot [\tau_{n - 1}(\ell_{3})(y) - \tau_{n - 1}(\ell_{1})(y)] \geq \tfrac{\theta}{2}(\ell_{3} - \ell_{1}) = \theta (\ell_{2} - \ell_{1}).\end{displaymath}
The upper bound $\tau_{n}(\ell_{2})(y) - \tau_{n}(\ell_{1})(y) \leq \theta^{-1}(\ell_{2} - \ell_{1})$ is proved in the same manner. Similarly, we use the inductive hypothesis \eqref{X3} for $\tau_{n - 1}$ to prove \eqref{form76}:
\begin{displaymath} |\dot{\tau}_{n}(\ell_{2})(y) - \dot{\tau}_{n}(\ell_{1})(y)| = \tfrac{1}{2} \cdot |\dot{\tau}_{n - 1}(\ell_{3})(y) - \dot{\tau}_{n - 1}(\ell_{1})(y)| \leq \tfrac{1}{2} \cdot \delta \cdot 2^{-(n - 1)/2} \leq \delta \cdot 2^{-n/2}. \end{displaymath}
We are now done in the case $y \in \R \, \setminus \, [U^{\sharp}(\ell_{1}) \cup U^{\sharp}(\ell_{2})]$.

Assume next that $y \in U^{\sharp}(\ell_{1}) \cup U^{\sharp}(\ell_{2})$. In this case, the inductive assumptions \eqref{X2}-\eqref{X3} on $\tau_{n - 1}(\ell_{1})$ and $\tau_{n - 1}(\ell_{3})$ will not be used to verify \eqref{form38} and \eqref{form76}. Instead, we will rely solely on \eqref{X5}. By definition of the sets $U^{\sharp}(\ell_{j})$, the following is true for at least one index $j \in \{1,2\}$: there exists $w \in \mathcal{C}_{m}(\ell_{j}) \cup \mathcal{C}_{m}(\ell_{j}^{\uparrow}) \cup \mathcal{C}_{m}(\ell_{j\downarrow})$, where $m = \ceil{n/2}$, such that
\begin{equation}\label{form55} |y - f_{2}(w)| \leq \kappa_{0}K 2^{-n/2}. \end{equation}
The possible lines $\ell_{j},\ell_{j}^{\uparrow},\ell_{j\downarrow}$ occurring here, for $j \in \{1,2\}$, are $\ell_{1\downarrow},\ell_{1},\ell_{2},\ell_{3}$, but to make the following discussion symmetric, we artificially add the possibility $w \in \mathcal{C}_{m}(\ell_{3}^{\uparrow})$. In other words, we only use the information
\begin{equation}\label{form32} w \in \mathcal{C}_{m}(\ell_{1\downarrow}) \cup \mathcal{C}_{m}(\ell_{1}) \cup \mathcal{C}_{m}(\ell_{2}) \cup \mathcal{C}_{m}(\ell_{3}) \cup \mathcal{C}_{m}(\ell_{3}^{\uparrow}), \end{equation}
see Figure \ref{fig4}.
\begin{figure}[h!]
\begin{center}
\begin{overpic}[scale = 0.61]{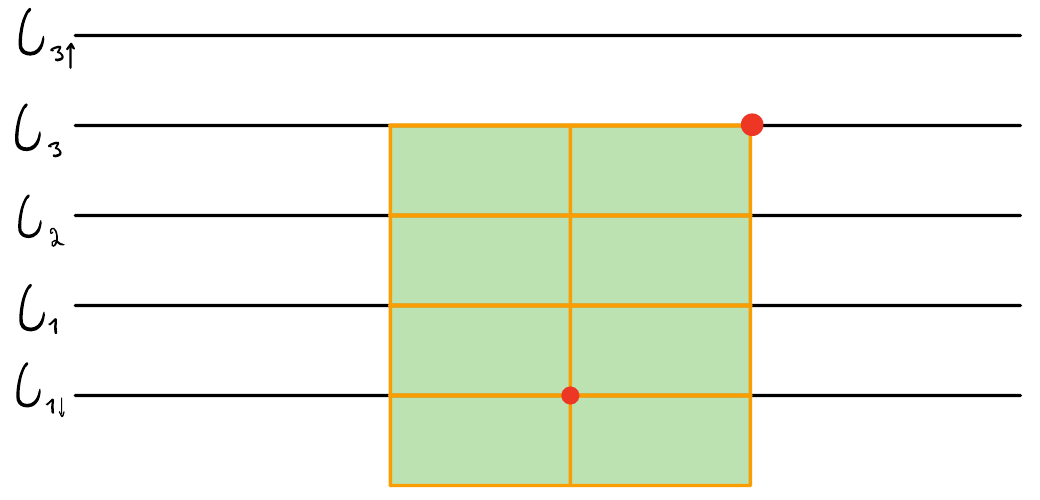}
\put(56,29.5){$\hat{Q}$}
\put(74,30){$\widehat{w}$}
\put(56,6){$w$}
\end{overpic}
\caption{A possible configuration of the points $w \in \mathcal{C}_{m}$, $\widehat{w} \in \mathcal{C}_{m - 1} \cap \mathcal{C}_{m}$ and the rectangle $\hat{Q} \in \mathcal{T}_{\ceil{(n - 1)/2}}$.}\label{fig4}
\end{center}
\end{figure}
Let $Q \in \mathcal{T}_{m}$ be a rectangle with corner $w$. Assumption \eqref{form32} enables us to use Corollary \ref{cor6} much like we did while verifying axiom \eqref{X1}. The conclusion is that there exists a corner $\widehat{w}$ of the dyadic parent of $Q$ which lies on $\ell_{1} \cup \ell_{3}$, and is also the corner of some rectangle $\hat{Q} \in \mathcal{T}_{\ceil{(n - 1)/2}}$ (which is either the dyadic parent or sibling of $Q$, depending on the parity of $n$, as we saw in the proof of Corollary \ref{cor6}). We write $\hat{\ell} \in \{\ell_{1},\ell_{3}\}$ for the line such that $\widehat{w} \in \hat{\ell}$.

Recall from \eqref{form55} that $|y - f_{2}(w)| \leq \kappa_{0}K2^{-n/2}$. Since $d_{\mathrm{par}}(w,\widehat{w}) \leq 100 \cdot 2^{-m} = 100 \cdot 2^{-n/2}$, we deduce from Lemma \ref{lemma11} that 
\begin{equation}\label{form33} y \in B(f_{2}(w),\kappa_{0}K2^{-n/2}) \subset B(f_{2}(\widehat{w}),K2^{-(n - 1)/2}). \end{equation} 
This inclusion enables us to apply the inductive hypothesis \eqref{X5} on $\tau_{n - 1}(\ell_{1})(y)$ and $\tau_{n - 1}(\ell_{3})(y)$. Namely, for both $\ell := \ell_{1}$ and $\ell := \ell_{3}$, the line $\hat{\ell} \in \{\ell_{1},\ell_{3}\}$ satisfies $\hat{\ell} \in \{\ell,\ell^{\uparrow},\ell_{\downarrow}\}$, where $\ell^{\uparrow},\ell_{\downarrow}$ refer to the generation $n - 1$ neighbours of $\ell$. It follows that 
\begin{displaymath} \widehat{w} \in \mathcal{C}_{\ceil{(n - 1)/2}}(\ell) \cup \mathcal{C}_{\ceil{(n - 1)/2}}(\ell^{\uparrow}) \cup  \mathcal{C}_{\ceil{(n - 1)/2}}(\ell_{\downarrow}), \qquad \ell \in \{\ell_{1},\ell_{3}\}.\end{displaymath}
Since also \eqref{form33} holds, we may infer inductively from axiom \eqref{X5} that 
\begin{equation}\label{form36} \tau_{n - 1}(\ell)(y) \approx_{n - 1} q_{\hat{Q},\ell}(y), \qquad \ell \in \{\ell_{1},\ell_{3}\}, \end{equation}
where $q_{\hat{Q},\ell}$ is the $(\hat{Q},\ell)$-approximate quadric. Let us, at this point, recall the definitions
\begin{equation}\label{form53} \tau_{n}(\ell_{1})(y) = \sum_{I \in \mathcal{I}(\ell_{1})} \varphi_{I,1}(y) \cdot q_{I,1}(y) + \varphi_{1}(y) \cdot \tau_{n - 1}(\ell_{1})(y)\end{equation}
and
\begin{equation}\label{form54} \tau_{n}(\ell_{2})(y) = \sum_{I \in \mathcal{I}(\ell_{2})} \varphi_{I,2}(y) \cdot q_{I,2}(y) + \tfrac{1}{2} \varphi_{2}(y) \cdot [\tau_{n - 1}(\ell_{1})(y) + \tau_{n - 1}(\ell_{3})(y)]. \end{equation} 
We just deduced in \eqref{form36} that $\tau_{n - 1}(\ell_{j})(y) \approx_{n - 1} q_{\hat{Q},\ell_{j}}(y)$ for $j \in \{1,3\}$. We would like to know, in addition, that whenever $y \in \spt \varphi_{I,j}$ and $j \in \{1,2\}$, then $q_{I,j}(y) \approx_{n} q_{\hat{Q},\ell_{j}}(y)$. This can be deduced from Proposition \ref{prop4}, essentially by the same argument we used at \eqref{form68}. We repeat the brief details.

Let $j \in \{1,2\}$, and let $I := I_{n}(w_{I}) \in \mathcal{I}(\ell_{j})$ be an interval such that
\begin{displaymath} w_{I} \in \mathcal{C}_{m}(\ell_{j}) \cup \mathcal{C}_{m}(\ell_{j}^{\uparrow}) \cup \mathcal{C}_{m}(\ell_{j\downarrow}) \subset \ell_{1\downarrow} \cup \ell_{1} \cup \ell_{2} \cup \ell_{3}, \end{displaymath}
and 
\begin{equation}\label{form73} y \in \spt \varphi_{I,j} \subset \kappa_{0} I = B(f_{2}(w_{I}),\kappa_{0}K2^{-n/2}) \subset B(f_{2}(w_{I}),100K2^{-n/2}). \end{equation}
Let $Q_{I} \in \mathcal{T}_{m}$ be the rectangle with corner $w_{I}$ such that $q_{I,j}$ is the $(Q_{I},\ell_{j})$-approximate quadric, and recall that $\hat{Q} \in \mathcal{T}_{\ceil{(n - 1)/2}}$ was a rectangle with corner $\widehat{w}$. Combining \eqref{form33} and \eqref{form73}, and noting that the lines containing $w_{I}$, $\widehat{w}$, and $\ell_{j}$, all lie in $\{\ell_{1\downarrow},\ldots,\ell_{3}\}$, we deduce from Proposition \ref{prop4} that 
\begin{equation}\label{form75} q_{I,j}(y) := q_{Q_{I},\ell_{j}}(y) \approx_{n} q_{\hat{Q},\ell_{j}}(y) \text{ for } j \in \{1,2\} \text{ and } I \in \mathcal{I}(\ell_{j}) \text{ with } y \in \spt \varphi_{I,j}. \end{equation}

Let us abbreviate $q_{j} := q_{\hat{Q},\ell_{j}}$ for $j \in \{1,2,3\}$. We are very close to completing the proof of \eqref{form76}, hence \eqref{X3}. Combining \eqref{form75} with \eqref{form36}-\eqref{form53}, it first follows that
\begin{equation}\label{form92} \tau_{n}(\ell_{1})(y) \approx_{n - 1} q_{1}(y). \end{equation}
(This uses Corollary \ref{cor7} in the same way as \eqref{form78}). Similarly, combining \eqref{form75} with \eqref{form36}, \eqref{form54}, one can deduce (using Lemma \ref{lemma13}) that
\begin{equation}\label{form93} \tau_{n}(\ell_{2})(y) \approx_{n - 1} \sum_{I \in \mathcal{I}(\ell_{2})} \varphi_{I,2}(y)q_{2}(y) + \tfrac{1}{2}\varphi_{2}(y) \cdot [q_{1}(y) + q_{3}(y)], \end{equation}
and further
\begin{equation}\label{form94} |\dot{\tau}_{n}(\ell_{2})(y) - \dot{q}_{2}(y)| \lesssim_{M} \left(\tfrac{1}{K} + o(\epsilon)\right) \cdot 2^{-n/2}. \end{equation}
This is proven in exactly the same way as the first assertion in \eqref{form77}, so we omit the details; the main point is, again, that the quadrics $q_{1},q_{2},q_{3}$ only differ by a constant $\lesssim_{M} 2^{-n}$. Combining \eqref{form92} and \eqref{form94}, we estimate
\begin{displaymath} |\dot{\tau}_{n}(\ell_{1})(y) - \dot{\tau}_{n}(\ell_{2})(y)| \lesssim_{M} |\dot{q}_{1}(y) - \dot{q}_{2}(y)| + \left(o(\epsilon) + \tfrac{1}{K}\right) \cdot 2^{-n/2} = \left(o(\epsilon) + \tfrac{1}{K}\right) \cdot 2^{-n/2}. \end{displaymath}
Now \eqref{form76} follows by choosing $K \geq 1$ small enough depending on $\delta,M$, and $\epsilon > 0$ small enough depending on $K,M,\Sigma$. This completes the proof of \eqref{X3}. 

We turn to proving \eqref{X2}. Using \eqref{form92}, and that $\sum_{I \in \mathcal{I}(\ell_{2})} \varphi_{I,2} + \varphi_{2} \equiv 1$, we expand
\begin{displaymath} \tau_{n}(\ell_{1})(y) \approx_{n - 1} q_{1}(y)\left[ \sum_{I \in \mathcal{I}(\ell_{2})} \varphi_{I,2}(y) + \varphi_{2}(y) \right] = \sum_{I \in \mathcal{I}(\ell_{2})} \varphi_{I,2}(y)q_{1}(y) + \varphi_{2}(y)q_{1}(y). \end{displaymath} 
To estimate the difference $\tau_{n}(\ell_{2})(y) - \tau_{n}(\ell_{1})(y)$ we subtract the approximation \eqref{form93} from the one above:
\begin{align} \tau_{n}(\ell_{2})(y) - \tau_{n}(\ell_{1})(y) & \approx_{n - 1} \sum_{I \in \mathcal{I}(\ell_{2})} \varphi_{I,2}(y)q_{2}(y) + \tfrac{1}{2}\varphi_{2}(y) \cdot [q_{1}(y) + q_{3}(y)] \notag\\
&\qquad \qquad - \sum_{I \in \mathcal{I}(\ell_{2})} \varphi_{I,2}(y)q_{1}(y) - \varphi_{2}(y)q_{1}(y) \notag\\
&\label{form56} = \tfrac{1}{2}\varphi_{2}(y) [q_{3}(y) - q_{1}(y)] + [q_{2}(y) - q_{1}(y)] \sum_{I \in \mathcal{I}(\ell_{2})} \varphi_{I,2}(y).  \end{align} 
To conclude the proof of \eqref{X2} from here, it suffices to show that
\begin{equation}\label{form37} q_{3}(y) - q_{1}(y) \sim_{N} 2^{-n} \sim_{N} q_{2}(y) - q_{1}(y), \end{equation}
recalling that $N \sim_{M,\Sigma} 1$. Indeed, plugging these estimates to \eqref{form56}, and using $\varphi_{2} + \sum_{I \in \mathcal{I}(\ell_{2})} \varphi_{I,2} \equiv 1$, we find that $\eqref{form56} \sim_{N} 2^{-n}$. Consequently, the computation preceding \eqref{form56}, and the definition of the "$\approx_{n}$" notation imply that $\tau_{n}(\ell_{2}) - \tau_{n}(\ell_{1})(y)$ is within additive error of $o(\epsilon) \cdot 2^{-n}$, from the expression at \eqref{form56}, which is $\sim_{N} 2^{-n}$. Choosing $\epsilon > 0$ small enough, depending only on $K,M,\Sigma$, it follows that $\tau_{n}(\ell_{2})(y) - \tau_{n}(\ell_{1})(y) \sim_{N} 2^{-n}$, as claimed in \eqref{X2}. This completes the proof of \eqref{X2} in the case $y \in U^{\sharp}(\ell_{1}) \cup U^{\sharp}(\ell_{2})$.  

It remains to establish \eqref{form37}. This argument is very similar to the one we already saw at the end of case $n = 0$. We first recall that the signature of $\mathcal{T}$, and hence $\hat{Q} \in \mathcal{T}$ is $+$, which meant that the $(N,\epsilon \ell(\hat{Q}))$-QIE
\begin{displaymath} F := F_{\hat{Q}} := \pi_{\V_{\hat{Q}}} \circ \iota_{\hat{Q}} = \psi_{\V_{\hat{Q}}} \circ \Pi \circ \iota_{\hat{Q}} \colon H\hat{Q} \to \W \end{displaymath}
is vertically $\epsilon'\ell(\hat{Q})$-increasing on the rectangle $B\hat{Q} \subset H\hat{Q}$, where $\epsilon' = 2N\epsilon$. This fact was established in Lemma \ref{lemma4}, assuming $H \geq A(1 + \Sigma)^{2}MB$. The quadric straightening map $\psi := \psi_{\mathbb{V}_{\hat{Q}}} \colon \W \to \W$ appearing in the definition of $F$ was introduced just above \eqref{form21}. It had the form
\begin{displaymath} \psi(y,t) = (y,t - \tfrac{1}{2}ay^{2} - by), \end{displaymath}
where $a$ and $b$ are the common coefficients of approximate $(\hat{Q},\ell)$-quadrics, all of which have the form $q_{\ell}(y) = \tfrac{1}{2}ay^{2} + by + c_{\ell}$. In particular, the quadrics $q_{1},q_{2},q_{3}$ appearing in \eqref{form37} can be written in this general form for certain coefficients $c_{1},c_{2},c_{3} \in \R$, and $q_{i} - q_{j} \equiv c_{i} - c_{j}$ for $i,j \in \{1,2,3\}$. So, proving \eqref{form37} is equivalent to showing that
\begin{equation}\label{form47a} c_{2} - c_{1} \sim_{N} \sim 2^{-n} \sim_{N} c_{3} - c_{1}. \end{equation}

To prove \eqref{form47a}, we pick (any) $3$ vertically aligned points 
\begin{displaymath} (y_{0},t_{1}),(y_{0},t_{2}),(y_{0},t_{3}) \in B\hat{Q} \quad \text{with} \quad (y_{0},t_{j}) \in \ell_{j}, \, j \in \{1,2,3\}. \end{displaymath}
This can be done for $B \geq 5$, since $\hat{Q}$ had a corner $\widehat{w} \in \ell_{1} \cup \ell_{3}$, and $\dist(\ell_{i},\ell_{j}) \leq 2\ell(\hat{Q})$ for $1 \leq i,j \leq 3$. Evidently $t_{j + 1} = t_{j} + 2^{-n}$, so in particular $\|t_{j + 1} - t_{j}\| = 2^{-n/2}$, for $j \in \{1,2\}$. Since $F$ is vertically $\epsilon'\ell(\hat{Q})$-increasing, in particular the map $t \mapsto \pi_{2}(F(y_{0},t))$ is $\epsilon' \ell(\hat{Q})$-increasing on $([t_{1},t_{2}],\|\cdot\|)$. Since $\|t_{j + 1} - t_{j}\| > \epsilon'\ell(\hat{Q})$ (recall from \eqref{form22} that $\epsilon \ll N^{-10}$), this implies
\begin{displaymath} c_{3} = \pi_{2}(F(y_{0},t_{3})) > c_{2} = \pi_{2}(F(y_{0},t_{2})) > c_{1} = \pi_{2}(F(y_{0},t_{1})). \end{displaymath}
Here we also used formula \eqref{form45}, which says that $(\pi_{2} \circ F)|_{\ell} \equiv c_{\ell}$. To establish the more quantitative estimates \eqref{form47a}, we use Lemma \ref{lemma6}, which implies that the map $t \mapsto \pi_{2}(F(y_{0},t))$ is an $(N,2\epsilon \ell(\hat{Q}))$-QIE on $([t_{1},t_{3}],\|\cdot\|)$. Hence,
\begin{align*} (2N)^{-1}2^{-n/2} & \leq N^{-1}2^{-n/2} - 2\epsilon \ell(\hat{Q}) \\
& \leq \|\pi_{2}(F(y_{0},t_{j + 1})) - \pi_{2}(F(y_{0},t_{j}))\| = \|c_{j + 1} - c_{j}\|\\
& \leq N2^{-n/2} + 2\epsilon \ell(\hat{Q}) \leq 2N \cdot 2^{-n/2}, \quad j \in \{1,2\}. \end{align*}
This completes the proof of \eqref{form47a}, and consequently the proofs of \eqref{form37} and \eqref{X2}.

\subsubsection{Verifying axiom \eqref{X5}} The definition of $\tau_{n}(\ell)$ has been set up so that axiom \eqref{X5} is easy to verify. Write $m := \ceil{n/2}$, $\ell \in \mathcal{L}_{n}$, and let $w \in \mathcal{C}_{m}(\ell) \cup \mathcal{C}_{m}(\ell^{\uparrow}) \cup \mathcal{C}_{m}(\ell_{\downarrow})$. Let $Q \in \mathcal{T}_{m}$ be a rectangle with corner $w$, and let $q := q_{Q,\ell}$ be the $(Q,\ell)$-approximate quadric. The claim is, then, that
\begin{equation}\label{form84} \tau_{n}(y) \approx_{n} q(y), \qquad y \in B(f_{2}(w),K2^{-n/2}) =: I_{0}. \end{equation}
Fix $y \in I_{0}$. Recall from the definition of the functions $\{\varphi_{I,\ell}\}_{I \in \mathcal{I}(\ell)}$ that
\begin{displaymath} \mathbf{1}_{U(\ell)} \leq \sum_{I \in \mathcal{I}(\ell)} \varphi_{I,\ell} \leq 1, \end{displaymath}
so consequently $\varphi_{\ell} = 1 - \sum_{I \in \mathcal{I}(\ell)} \varphi_{I,\ell}$ vanishes on $U(\ell)$. Since $I_{0} \subset U(\ell)$, we have
\begin{displaymath} \tau_{n}(\ell)(y) = \sum_{I \in \mathcal{I}(\ell)} \varphi_{I,\ell}(y) \cdot q_{I,\ell}(y), \qquad y \in I_{0}, \end{displaymath}
regardless of whether $\ell \in \mathcal{L}_{n - 1}$ or $\ell \in \mathcal{L}_{n} \, \setminus \, \mathcal{L}_{n - 1}$. Now, all that remains to be verified is:
\begin{equation}\label{form83} q_{I,\ell}(y) \approx_{n} q(y) \qquad \text{whenever} \quad I \in \mathcal{I}(\ell) \text{ and } y \in \spt \varphi_{I,\ell}. \end{equation} 
Indeed, once \eqref{form83} is settled, the approximation \eqref{form84} follows from Corollary \ref{cor7}.
 
The verification of \eqref{form83} is similar to arguments we have already recorded multiple times, but let us record them once more, since \eqref{X5} is perhaps the most fundamental axiom. Let $I \in \mathcal{I}(\ell)$ be such that $y \in \spt \varphi_{I,\ell}$. Thus, there exists $w_{I} \in \mathcal{C}_{m}(\ell) \cup \mathcal{C}_{m}(\ell^{\uparrow}) \cup \mathcal{C}_{m}(\ell_{\downarrow})$ such that 
\begin{displaymath} y \in B(f_{2}(w_{I}),\kappa_{0}K2^{-n/2}) \cap B(f_{2}(w),K2^{-n/2}). \end{displaymath} 
Let $Q_{I} \in \mathcal{T}_{m}$ be a rectangle with corner $w_{I}$ such that $q_{I,\ell} = q_{Q_{I},\ell}$ is the $(Q_{I},\ell)$-approximate quadric. Since $w,w_{I} \in \ell_{\downarrow} \cup \ell \cup \ell^{\uparrow}$, and the separation of these lines is bounded by $\leq 100 \cdot 2^{-n}$, Proposition \ref{prop4} is applicable and yields \eqref{form83} directly. This completes the verification of axiom \eqref{X5}.

\subsubsection{Verifying axiom \eqref{X4}} Let $\ell \in \mathcal{L}_{n - 1} \subset \mathcal{L}_{n}$. The claim is that
\begin{equation}\label{form80} |\tau_{n}(\ell)(y) - \tau_{n - 1}(\ell)(y)| \leq \delta \cdot 2^{-n} \quad \text{and} \quad |\dot{\tau}_{n}(\ell)(y) - \dot{\tau}_{n - 1}(\ell)(y)| \leq \delta \cdot 2^{-n/2}. \end{equation} 
Since $\ell \in \mathcal{L}_{n - 1}$, the definition of $\tau_{n}(\ell)$ is
\begin{equation}\label{form81} \tau_{n}(\ell)(y) = \sum_{I \in \mathcal{I}(\ell)} \varphi_{I,\ell}(y) \cdot q_{I,\ell}(y) + \varphi_{\ell}(y) \cdot \tau_{n - 1}(\ell)(y), \qquad y \in \R. \end{equation}
If $y \notin U^{\sharp}(\ell)$, then simply $\tau_{n}(\ell)(y) = \tau_{n - 1}(\ell)(y)$ and $\dot{\tau}_{n}(\ell)(y) = \dot{\tau}_{n - 1}(\ell)(y)$, and there is nothing to prove. So, let us assume that $y \in U^{\sharp}(\ell)$. In other words, with $m := \ceil{n/2}$, there exists $w \in \mathcal{C}_{m}(\ell) \cup \mathcal{C}_{m}(\ell^{\uparrow}) \cup \mathcal{C}_{m}(\ell_{\downarrow})$ such that $y \in B(f_{2}(w),\kappa_{0}K2^{-n/2})$. Let $Q \in \mathcal{T}_{m}$ be a rectangle with corner $w$. In this case, the verification of \eqref{form80} will only be based on the inductive information from axiom \eqref{X5}, which implies that $\tau_{n - 1}(\ell)$ is close to a certain quadric around $y$. The details are extremely similar to those already seen during the proof of axiom \eqref{X1} (more precisely the final case where $|y_{1} - y_{2}| \leq \tfrac{1}{5}K \cdot 2^{-n/2}$), so we just recall the skeleton of the argument: one begins by finding a corner point $\widehat{w} \in \mathcal{C}_{\ceil{(n - 1)/2}}$ on either $\ell$, or its $\mathcal{L}_{n - 1}$-neighbour, with the property 
\begin{equation}\label{form82} y \in B(f_{2}(\widehat{w}),K2^{-(n - 1)/2}). \end{equation} 
Then, one selects a rectangle $\hat{Q} \in \mathcal{T}_{\ceil{(n - 1)/2}}$ with corner $\widehat{w}$, and uses \eqref{form82} and \eqref{X5} inductively to show that $\tau_{n - 1}(\ell)(y) \approx_{n - 1} q(y)$, where $q := q_{\hat{Q},\ell}$ is the $(\hat{Q},\ell)$-approximate quadric; see \eqref{form62} for comparison. Next, one proves that all the quadrics $q_{I,\ell}$ appearing in \eqref{form81} with $y \in \spt \varphi_{I,\ell}$ must also satisfy $q_{I,\ell}(y) \approx_{n} q(y)$. For the details, see the proof of \eqref{form68}. From this, it follows from Corollary \ref{cor7} that $\tau_{n}(\ell)(y) \approx_{n - 1} q(y)$ (compare with \eqref{form78}), and finally $\tau_{n}(\ell)(y) \approx_{n - 1} \tau_{n - 1}(\ell)(y)$. This implies \eqref{form80} if $\epsilon > 0$ is chosen sufficiently small, depending only on $\delta,K,M,\Sigma$.

\medskip

The axioms \eqref{X1}-\eqref{X4} have been verified, and the proof of Proposition \ref{t:construction} is complete. Since the proof of Theorem \ref{main} had previously been reduced to establishing Proposition \ref{t:construction}, also the proof of Theorem \ref{main} is complete.
 
%%%%%%%%%%%%%%%%%%%%%%%%%%%%%%%%%

%%%%%%%%%%%%%%%%%%%%%%%%%%%%%%%%%

\appendix

\section{Proof of the quadric comparison lemma}\label{s:quadricComparison}

Here we prove Lemma \ref{quadricComparison}, whose statement we now recall:

\begin{lemma} Let $q \colon \R \to \R$ be a quadratic function, and let $I = [a,b] \subset \R$ be an interval. Write
\begin{displaymath} m := \max\{|q(a)|,|q(b)|\} \quad \text{and} \quad \dot{m} := \max\{|\dot{q}(a)|,|\dot{q}(b)|\}. \end{displaymath}
Then, for all $r \geq 1$ and $t \in rI$,
\begin{displaymath}\ |q(t)| \lesssim (m + |I|\dot{m})r^{2}, \quad |\dot{q}(t)| \lesssim (\dot{m} + \tfrac{m}{|I|})r \quad \text{and} \quad |\ddot{q}(t)| \lesssim \tfrac{m}{|I|^{2}} + \tfrac{\dot{m}}{|I|}.  \end{displaymath} 
\end{lemma}

\begin{proof} We first prove the case $[a,b] = [0,1]$. The quadric $q$ can be written as
\begin{displaymath} q(t) = q(0) + \dot{q}(0)t + \tfrac{1}{2}\ddot{q}(0)t^{2}, \end{displaymath}
where $|q(0)| \leq m$, $|\dot{q}(0)| \leq \dot{m}$, and
\begin{equation}\label{form69} |\ddot{q}(t)| \equiv |\ddot{q}(0)| \leq 2|q(1)| + 2|q(0)| + 2|\dot{q}(0)| \leq 4m + 2\dot{m}. \end{equation}
Consequently,
\begin{equation}\label{form40} |q(t)| \leq m + \dot{m}|t| + (4m + 2\dot{m})t^{2} \quad \text{and} \quad |\dot{q}(t)| \leq \dot{m} + (8m + 4\dot{m})|t|. \end{equation} 
Next, consider the general case, and define $\bar{q}(t) := q(t(b - a) + a)$. Then,
\begin{displaymath} \bar{m} := \max\{|\bar{q}(0)|,|\bar{q}(1)|\} = \max\{|q(a)|,|q(b)|\} = m, \end{displaymath}
and
\begin{displaymath} \dot{\bar{m}} := \max\{|\dot{\bar{q}}(0)|,|\dot{\bar{q}}(1)|\} = |I| \cdot \max\{|\dot{q}(a)|,|\dot{q}(b)|\} = |I|\dot{m}, \end{displaymath}
and $|\ddot{\bar{q}}(t)| \leq 4\bar{m} + 2\dot{\bar{m}} = 4m + 2|I|\dot{m}$ by \eqref{form69}. Using \eqref{form40}, we infer that for all $t \in rI$,
\begin{displaymath} |q(t)| = \left|\bar{q}\left(\tfrac{t - a}{b - a} \right) \right| \leq m + |I|\dot{m} \tfrac{|t - a|}{|b - a|} + (4m + 2|I|\dot{m})\left(\tfrac{t - a}{b - a}\right)^{2} \lesssim (m + |I|\dot{m})r^{2} \end{displaymath} 
and
\begin{displaymath} |\dot{q}(t)| = \tfrac{1}{|I|}\left|\dot{\bar{q}}\left( \tfrac{t - a}{b - a} \right) \right| \leq \dot{m} + \tfrac{1}{|I|}(8m + 4|I|\dot{m})\tfrac{|t - a|}{|b - a|} \lesssim (\dot{m} + \tfrac{m}{|I|})r \end{displaymath}
and
\begin{displaymath} |\ddot{q}(t)| = \tfrac{1}{|I|^{2}} \left|\ddot{\bar{q}}\left(\tfrac{t - a}{b - a} \right) \right| \lesssim \tfrac{m}{|I|^{2}} + \tfrac{\dot{m}}{|I|}, \end{displaymath} 
as claimed.
\end{proof}

\bibliographystyle{plain}
\bibliography{references}

\def\cprime{$'$}
\begin{thebibliography}{10}

\bibitem{zbMATH05015569}
L.~{Ambrosio}, F.~{Serra Cassano}, and D.~{Vittone}.
\newblock {Intrinsic regular hypersurfaces in Heisenberg groups.}
\newblock {\em {J. Geom. Anal.}}, 16(2):187--232, 2006.

\bibitem{MR2989430}
Jonas Azzam and Raanan Schul.
\newblock Hard {S}ard: quantitative implicit function and extension theorems
  for {L}ipschitz maps.
\newblock {\em Geom. Funct. Anal.}, 22(5):1062--1123, 2012.

\bibitem{BCSC}
F.~Bigolin, L.~Caravenna, and F.~Serra~Cassano.
\newblock Intrinsic {L}ipschitz graphs in {H}eisenberg groups and continuous
  solutions of a balance equation.
\newblock {\em Ann. Inst. H. Poincar\'e Anal. Non Lin\'eaire}, 32(5):925--963,
  2015.

\bibitem{MR2603594}
Francesco Bigolin and Davide Vittone.
\newblock Some remarks about parametrizations of intrinsic regular surfaces in
  the {H}eisenberg group.
\newblock {\em Publ. Mat.}, 54(1):159--172, 2010.

\bibitem{2020arXiv200811544B}
Simon {Bortz}, John {Hoffman}, Steve {Hofmann}, Jos{\'e}~Luis {Luna Garcia},
  and Kaj {Nystr{\"o}m}.
\newblock {Coronizations and big pieces in metric spaces}.
\newblock {\em arXiv e-prints}, page arXiv:2008.11544, August 2020.

\bibitem{Calderon}
A.-P. Calder\'on.
\newblock Cauchy integrals on {L}ipschitz curves and related operators.
\newblock {\em Proc. Nat. Acad. Sci. U.S.A.}, 74(4):1324--1327, 1977.

\bibitem{CFO}
V.~{Chousionis}, K.~{F{\"a}ssler}, and T.~{Orponen}.
\newblock {Intrinsic Lipschitz graphs and vertical $\beta$-numbers in the
  Heisenberg group}.
\newblock {\em Amer. J. Math.}, 141(4):1087--1147, Aug 2019.

\bibitem{CFO2}
Vasileios Chousionis, Katrin F{\"a}ssler, and Tuomas Orponen.
\newblock Boundedness of singular integrals on ${C}^{1,\alpha}$ intrinsic
  graphs in the {H}eisenberg group.
\newblock {\em Adv. Math.}, 354:106745, 2019.

\bibitem{MR3678492}
Vasileios Chousionis and Sean Li.
\newblock Nonnegative kernels and 1-rectifiability in the {H}eisenberg group.
\newblock {\em Anal. PDE}, 10(6):1407--1428, 2017.

\bibitem{2020arXiv200411447C}
Vasileios {Chousionis}, Sean {Li}, and Robert {Young}.
\newblock {The strong geometric lemma for intrinsic Lipschitz graphs in
  Heisenberg groups}.
\newblock {\em arXiv e-prints}, page arXiv:2004.11447, April 2020.

\bibitem{MR700980}
R.~R. Coifman, G.~David, and Y.~Meyer.
\newblock La solution des conjecture de {C}alder\'{o}n.
\newblock {\em Adv. in Math.}, 48(2):144--148, 1983.

\bibitem{CMM}
R.~R. Coifman, A.~McIntosh, and Y.~Meyer.
\newblock L'int\'egrale de {C}auchy d\'efinit un op\'erateur born\'e sur
  {$L^{2}$}\ pour les courbes lipschitziennes.
\newblock {\em Ann. of Math. (2)}, 116(2):361--387, 1982.

\bibitem{MR2247905}
Daniel~R. Cole and Scott~D. Pauls.
\newblock {$C^1$} hypersurfaces of the {H}eisenberg group are
  {$N$}-rectifiable.
\newblock {\em Houston J. Math.}, 32(3):713--724, 2006.

\bibitem{2020arXiv200201433C}
Francesca {Corni} and Valentino {Magnani}.
\newblock {Area formula for regular submanifolds of low codimension in
  Heisenberg groups}.
\newblock {\em arXiv e-prints}, page arXiv:2002.01433, February 2020.

\bibitem{DS1}
G.~David and S.~Semmes.
\newblock Singular integrals and rectifiable sets in {${\bf R}^n$}: Au-del\`a
  des graphes lipschitziens.
\newblock {\em Ast\'erisque}, (193):152, 1991.

\bibitem{MR1251061}
G.~David and S.~Semmes.
\newblock {\em Analysis of and on uniformly rectifiable sets}, volume~38 of
  {\em Mathematical Surveys and Monographs}.
\newblock American Mathematical Society, Providence, RI, 1993.

\bibitem{MR744071}
Guy David.
\newblock Op\'{e}rateurs int\'{e}graux singuliers sur certaines courbes du plan
  complexe.
\newblock {\em Ann. Sci. \'{E}cole Norm. Sup. (4)}, 17(1):157--189, 1984.

\bibitem{MR956767}
Guy David.
\newblock Op\'erateurs d'int\'egrale singuli\`ere sur les surfaces
  r\'eguli\`eres.
\newblock {\em Ann. Sci. \'Ecole Norm. Sup. (4)}, 21(2):225--258, 1988.

\bibitem{MR1132876}
Guy David and Stephen Semmes.
\newblock Quantitative rectifiability and {L}ipschitz mappings.
\newblock {\em Trans. Amer. Math. Soc.}, 337(2):855--889, 1993.

\bibitem{2019arXiv190610215D}
Daniela {Di Donato}, Katrin {F{\"a}ssler}, and Tuomas {Orponen}.
\newblock {Metric rectifiability of $\mathbb{H}$-regular surfaces with
  H{\"o}lder continuous horizontal normal}.
\newblock {\em arXiv e-prints}, page arXiv:1906.10215, June 2019.

\bibitem{2018arXiv181013122F}
Katrin {F{\"a}ssler} and Tuomas {Orponen}.
\newblock {Riesz transform and vertical oscillation in the Heisenberg group}.
\newblock {\em arXiv e-prints}, page arXiv:1810.13122, Oct 2018.

\bibitem{2019arXiv191103223F}
Katrin {F{\"a}ssler} and Tuomas {Orponen}.
\newblock {Singular integrals on regular curves in the Heisenberg group}.
\newblock {\em arXiv e-prints}, page arXiv:1911.03223, November 2019.

\bibitem{MR4127898}
Katrin F\"{a}ssler, Tuomas Orponen, and S\'{e}verine Rigot.
\newblock Semmes surfaces and intrinsic {L}ipschitz graphs in the {H}eisenberg
  group.
\newblock {\em Trans. Amer. Math. Soc.}, 373(8):5957--5996, 2020.

\bibitem{FSS}
Bruno Franchi, Raul Serapioni, and Francesco Serra~Cassano.
\newblock Intrinsic {L}ipschitz graphs in {H}eisenberg groups.
\newblock {\em J. Nonlinear Convex Anal.}, 7(3):423--441, 2006.

\bibitem{FSSC2}
Bruno Franchi, Raul Serapioni, and Francesco Serra~Cassano.
\newblock Differentiability of intrinsic {L}ipschitz functions within
  {H}eisenberg groups.
\newblock {\em J. Geom. Anal.}, 21(4):1044--1084, 2011.

\bibitem{MR1069238}
Peter~W. Jones.
\newblock Rectifiable sets and the traveling salesman problem.
\newblock {\em Invent. Math.}, 102(1):1--15, 1990.

\bibitem{MR3512421}
Sean Li and Raanan Schul.
\newblock An upper bound for the length of a traveling salesman path in the
  {H}eisenberg group.
\newblock {\em Rev. Mat. Iberoam.}, 32(2):391--417, 2016.

\bibitem{MR2789472}
Pertti Mattila, Raul Serapioni, and Francesco Serra~Cassano.
\newblock Characterizations of intrinsic rectifiability in {H}eisenberg groups.
\newblock {\em Ann. Sc. Norm. Super. Pisa Cl. Sci. (5)}, 9(4):687--723, 2010.

\bibitem{2019arXiv190811639M}
Andrea {Merlo}.
\newblock {Geometry of $1$-codimensional measures in Heisenberg groups}.
\newblock {\em arXiv e-prints}, page arXiv:1908.11639, August 2019.

\bibitem{2020arXiv200703236M}
Andrea {Merlo}.
\newblock {Marstrand-Mattila rectifiability criterion for $1$-codimensional
  measures in Carnot Groups}.
\newblock {\em arXiv e-prints}, page arXiv:2007.03236, July 2020.

\bibitem{MR3194680}
Roberto Monti.
\newblock Lipschitz approximation of {$\Bbb{H}$}-perimeter minimizing
  boundaries.
\newblock {\em Calc. Var. Partial Differential Equations}, 50(1-2):171--198,
  2014.

\bibitem{MR3682744}
Roberto Monti and Giorgio Stefani.
\newblock Improved {L}ipschitz approximation of {$H$}-perimeter minimizing
  boundaries.
\newblock {\em J. Math. Pures Appl. (9)}, 108(3):372--398, 2017.

\bibitem{2020arXiv200412522N}
Assaf {Naor} and Robert {Young}.
\newblock {Foliated corona decompositions}.
\newblock {\em arXiv e-prints}, page arXiv:2004.12522, April 2020.

\bibitem{MR3984100}
Sebastiano Nicolussi and Francesco Serra~Cassano.
\newblock The {B}ernstein problem for {L}ipschitz intrinsic graphs in the
  {H}eisenberg group.
\newblock {\em Calc. Var. Partial Differential Equations}, 58(4):Paper No. 141,
  28, 2019.

\bibitem{2020arXiv200608293O}
Tuomas {Orponen} and Michele {Villa}.
\newblock {Sub-elliptic boundary value problems in flag domains}.
\newblock {\em arXiv e-prints}, page arXiv:2006.08293, June 2020.

\bibitem{Semmes}
Stephen~W. Semmes.
\newblock A criterion for the boundedness of singular integrals on
  hypersurfaces.
\newblock {\em Trans. Amer. Math. Soc.}, 311(2):501--513, 1989.

\bibitem{MR3587666}
F.~Serra~Cassano.
\newblock Some topics of geometric measure theory in {C}arnot groups.
\newblock In {\em Geometry, analysis and dynamics on sub-{R}iemannian
  manifolds. {V}ol. 1}, EMS Ser. Lect. Math., pages 1--121. Eur. Math. Soc.,
  Z\"urich, 2016.

\bibitem{MR942829}
Jussi V\"{a}is\"{a}l\"{a}.
\newblock Invariants for quasisymmetric, quasi-{M}\"{o}bius and bi-{L}ipschitz
  maps.
\newblock {\em J. Analyse Math.}, 50:201--223, 1988.

\bibitem{2020arXiv200714286V}
Davide {Vittone}.
\newblock {Lipschitz graphs and currents in Heisenberg groups}.
\newblock {\em arXiv e-prints}, page arXiv:2007.14286, July 2020.

\end{thebibliography}

\end{document}